\documentclass[11pt]{amsart}

\usepackage{amsmath,amssymb,amsthm}
\usepackage{enumitem}
\usepackage{tikz}
\usetikzlibrary{arrows.meta}
\usepackage{tikz-cd}
\usepackage[hidelinks,colorlinks=true,linkcolor=blue,citecolor=blue]{hyperref}
\usepackage{geometry}
\theoremstyle{plain}
\newtheorem{theorem}{Theorem}[section]
\newtheorem{proposition}[theorem]{Proposition}
\newtheorem{lemma}[theorem]{Lemma}
\newtheorem{corollary}[theorem]{Corollary}

\theoremstyle{definition}
\newtheorem{definition}[theorem]{Definition}
\newtheorem{example}[theorem]{Example}
\newtheorem{remark}[theorem]{Remark}
\newtheorem{construction}[theorem]{Construction}
\newtheorem{algorithm}[theorem]{Algorithm}
\newtheorem{problem}[theorem]{Problem}

\DeclareMathOperator{\gldim}{gl.dim}
\DeclareMathOperator{\findim}{fin.dim}
\DeclareMathOperator{\pd}{pd}
\DeclareMathOperator{\rad}{rad}
\DeclareMathOperator{\Hom}{Hom}

\DeclareMathOperator{\soc}{soc}

\title[Syzygy-quiver formulas for special biserial algebras]{Syzygy-Quiver Formulas and Finitistic Dimensions\\ of Special Biserial Algebras}

\author{Shantanu Sardar}
\address{Department of Mathematics and Statistics, Shaanxi Normal University, Xi'an, China}
\email{shantanusardar17@gmail.com}


\begin{document}

\begin{abstract}
Let $\Lambda$ be a finite dimensional algebra over a field $\mathbf{k}$.
We introduce the reachable syzygy quiver of $\Lambda$, a directed
multigraph whose vertices are the isomorphism classes of indecomposable
nonprojective direct summands of iterated syzygies of simple modules,
and whose arrows record nonprojective summands of first syzygies. The
projective dimension of a vertex is one plus its directed path height,
allowing the global and finitistic dimensions of $\Lambda$ to be read
off from this height function. This formula holds for every finite
dimensional algebra, but its states are arbitrary modules.

We show that when the first syzygy of an indecomposable module admits
a closed combinatorial description from the bound quiver, this state
space reduces to a smaller combinatorial quiver, making the global and
finitistic dimensions exact computable invariants. We carry this out
in three stages. Over a special biserial algebra, the Liu--Morin
closure theorem identifies the reachable syzygy quiver with a
string-syzygy quiver whose vertices are string modules. Over a string
algebra, this compresses to the relation-chain quiver of Xin and
Zhang, recovering their global dimension formula along with a
parallel formula for the finitistic dimension. Over an arbitrary
biserial algebra, the universal formula holds, but uniform compression
to paths or strings is impossible, with the Skowro\'nski--Waschb\"usch
algebra providing a witness. Finally, we show two string algebras may
share the same Gabriel quiver but have distinct global dimensions,
proving the ideal of relations is indispensable for these invariants,
and we prove that a pointwise homological reduction to a string
algebra on the same quiver preserving the monomials of the relations
is impossible.
\end{abstract}

\maketitle

\section{Introduction}\label{sec:introduction}

\subsection{Context and motivation}\label{subsec:context}

One of the classical homological invariants associated to a finite
dimensional algebra is its global dimension. For a finite dimensional
$\mathbf{k}$-algebra $\Lambda$, the global dimension
$\gldim\Lambda$ is defined as the supremum of the projective dimensions
of the finitely generated right $\Lambda$-modules, and by standard
homological algebra it suffices to take the supremum over the simple
right $\Lambda$-modules; we refer to \cite[Chapter~I]{ASS06} for the
standard background. Another classical invariant is the finitistic
dimension $\findim\Lambda$, defined as the supremum of the projective
dimensions of those right $\Lambda$-modules that admit a finite
projective resolution. It is clear that $\findim\Lambda \le
\gldim\Lambda$, with equality whenever $\gldim\Lambda < \infty$. The
finitistic dimension conjecture asserts that $\findim\Lambda < \infty$
for every finite dimensional algebra. This conjecture remains open in
general, although it has been resolved for many specific classes of
algebras.

The class of special biserial algebras occupies a central position in
the representation theory of finite dimensional algebras. Recall that
a bound quiver algebra $\Lambda = \mathbf{k}Q/I$ is special biserial
if every vertex of $Q$ is the source of at most two arrows and the
target of at most two arrows, and if for every arrow $\alpha$ there is
at most one arrow $\beta$ with $s(\beta) = t(\alpha)$ and
$\alpha\beta \notin I$, and at most one arrow $\gamma$ with
$t(\gamma) = s(\alpha)$ and $\gamma\alpha \notin I$. String algebras
are precisely the special biserial algebras whose ideal of relations
is generated by paths, and gentle algebras are the string algebras
satisfying an additional uniqueness condition on the relations. Every
special biserial algebra is biserial, in the sense that the radical of
every indecomposable projective left or right module is a sum of at
most two uniserial modules whose intersection is either zero or a
simple module. The converse is strictly false, as we discuss in
Section~\ref{sec:biserial}.

A theorem of Xin and Zhang expresses the global dimension of a
string algebra as the maximum length of a minimal relation chain
beginning with an arrow. A natural question is how far this philosophy
extends. In this paper we show that it extends to all special biserial
algebras, provided that the relation-chain quiver is upgraded to a
larger homological state graph whose vertices are string modules
rather than isolated paths. The key technical input is a closure
theorem of Liu and Morin, which states that the first syzygy of a
string module over a special biserial algebra decomposes into a
projective part and a finite direct sum of string modules; we recall
this result as Proposition~\ref{thm:liumorin} in
Section~\ref{sec:background}. This closure property allows us to
replace the universal module state space by a combinatorial state
space of string diagrams, rendering the global and finitistic
dimensions exact invariants computable directly from the bound quiver.

For a general biserial algebra the universal formula remains valid,
but no uniform combinatorial compression to paths or strings is
available. The Skowro\'nski--Waschb\"usch algebra serves as the
standard witness to this boundary: its global dimension is two, and
the first syzygy of the simple module at the source of one of its
arrows is an indecomposable module whose coefficient graph contains a
fundamental triangle, definitively proving that it is not a string
module. A general biserial algebra cannot be realised as an algebraic
quotient of a special biserial algebra: the class of special biserial
algebras is closed under admissible quotients, whereas the
Skowro\'nski--Waschb\"usch algebra is biserial but not special
biserial. We record the short argument in
Remark~\ref{rem:biserialquotient}.

\subsection{Main results}\label{subsec:mainresults}

Throughout, $\mathbf{k}$ is a field and $\Lambda$ is a finite
dimensional $\mathbf{k}$-algebra. We work with right modules, and we
compose paths from left to right. Let $Q = (Q_0, Q_1, s, t)$ be a
finite quiver and let $I \subseteq (\mathbf{k}Q^+)^2$ be an admissible
ideal, so that $\Lambda = \mathbf{k}Q/I$. We assume throughout that
$\Lambda$ is a split basic algebra. This assumption is automatic when
$\mathbf{k}$ is algebraically closed and entails no loss of generality
for the invariants we study, since the global and finitistic
dimensions are Morita invariants.

Our first result is valid for every finite dimensional algebra.

\begin{theorem}\label{thm:universal}
Let $\Lambda$ be a finite dimensional $\mathbf{k}$-algebra. Let
$\Gamma_\Omega(\Lambda)$ be the reachable syzygy quiver of $\Lambda$
(Definition~\ref{def:reachable}), whose vertices are the isomorphism
classes of indecomposable nonprojective right $\Lambda$-modules that
occur as direct summands of $\Omega_\Lambda^n S$ for some simple right
$\Lambda$-module $S$ and some $n \ge 0$, and whose arrow
multiplicities record the indecomposable nonprojective summands of the
first syzygy of each vertex. Let
$\Gamma_\Omega^{\mathrm{fin}}(\Lambda)$ be the finitistic syzygy
quiver (Definition~\ref{def:finitisticquiver}), whose vertices are the isomorphism classes of finitely generated indecomposable nonprojective right $\Lambda$-modules of finite projective dimension. Then the following hold.
\begin{enumerate}[label=\textup{(\arabic*)}]
\item For every vertex $[X]$ of $\Gamma_\Omega(\Lambda)$ one has
$\pd_\Lambda X = 1 + \operatorname{ht}([X])$.
\item $\gldim\Lambda = \sup\bigl(\{0\} \cup
\{1 + \operatorname{ht}([S]) : S \text{ a nonprojective simple right
} \Lambda\text{-module}\}\bigr)$.
\item $\findim\Lambda = \sup\bigl(\{0\} \cup \{1 + \operatorname{ht}([X]) :
[X] \in \Gamma_\Omega^{\mathrm{fin}}(\Lambda)\}\bigr)$.
\item $\gldim\Lambda < \infty$ if and only if
$\Gamma_\Omega(\Lambda)$ is finite and acyclic.
\end{enumerate}
\end{theorem}

The universal formula is formally exact for every finite dimensional
algebra, but its states are arbitrary modules. It becomes a
combinatorial theorem only for a class of algebras on which the states
and the syzygy transitions admit an effective description from the
bound quiver. Special biserial algebras harbour exactly the required
closure property.

\begin{theorem}\label{thm:specialbis}
Let $\Lambda = \mathbf{k}Q/I$ be a finite dimensional special
biserial algebra. Let $\Gamma_\Omega^{\mathrm{str}}(\Lambda)$ be the
string-syzygy quiver (Definition~\ref{def:stringsyzygyquiver}), whose
vertices are the isomorphism classes of indecomposable nonprojective
string modules that occur as direct summands of $\Omega_\Lambda^n S$
for some simple string module $S$ and some $n \ge 0$, and whose arrow
multiplicities record the indecomposable nonprojective string summands
of the first syzygy of each vertex. Then the following hold.
\begin{enumerate}[label=\textup{(\arabic*)}]
\item For every vertex $[X]$ of
$\Gamma_\Omega^{\mathrm{str}}(\Lambda)$ one has
$\pd_\Lambda X = 1 + \operatorname{ht}([X])$.
\item $\gldim\Lambda = \sup\bigl(\{0\} \cup
\{1 + \operatorname{ht}([S_i]) : S_i \text{ a nonprojective simple
string module}\}\bigr)$.
\item $\findim\Lambda \ge \sup\bigl(\{0\} \cup
\{1 + \operatorname{ht}([X]) : [X] \in
\Gamma_\Omega^{\mathrm{str,fin}}(\Lambda)\}\bigr)$
(Definition~\ref{def:finitisticstringsyzygy}), with equality whenever
$\gldim\Lambda < \infty$.
\item $\gldim\Lambda < \infty$ if and only if
$\Gamma_\Omega^{\mathrm{str}}(\Lambda)$ is finite and acyclic.
\end{enumerate}
\end{theorem}

In the strictly monomial case, the string-syzygy quiver compresses to
a smaller quiver whose vertices are nonzero paths, and we recover the
relation-chain formula of Xin and Zhang.

\begin{theorem}\label{thm:string}
Let $\Lambda = \mathbf{k}Q/I$ be a finite dimensional string algebra,
and let $\mathcal{R}(\Lambda)$ be its relation-chain quiver
(Definition~\ref{def:relationchainquiver}), whose vertices are the
nonzero paths reached from arrows by taking minimal right
annihilators, and whose arrows record these annihilators. Then
\[
\gldim\Lambda = \sup\bigl(\{0\} \cup
\{1 + \operatorname{ht}_{\mathcal{R}(\Lambda)}(\alpha) :
\alpha \in Q_1\}\bigr),
\]
and
\[
\findim\Lambda \ge \sup\bigl(\{0\} \cup
\{1 + \operatorname{ht}_{\mathcal{R}(\Lambda)}(p) :
p \in \mathcal{R}(\Lambda) \text{ with }
\operatorname{ht}_{\mathcal{R}(\Lambda)}(p) < \infty\}\bigr),
\]
with equality whenever $\gldim\Lambda < \infty$. Moreover,
$\gldim\Lambda < \infty$ if and only if $\mathcal{R}(\Lambda)$
contains no directed cycle reachable from an arrow.
\end{theorem}

For general biserial algebras, the universal formula of
Theorem~\ref{thm:universal} remains valid, but no uniform
combinatorial compression is known. The Skowro\'nski--Waschb\"usch
algebra serves as the standard obstruction.

\begin{theorem}\label{thm:biserial}
Let $\Lambda$ be a finite dimensional biserial algebra. Then the
following hold.
\begin{enumerate}[label=\textup{(\arabic*)}]
\item The universal formula of Theorem~\ref{thm:universal} holds for
$\Lambda$.
\item The ordinary Gabriel quiver does not determine
$\gldim\Lambda$, even within the class of string algebras.
\item If $\mathbf{k}$ is algebraically closed and $\Lambda$ is
representation finite, then $\Lambda$ is Morita equivalent to a
special biserial algebra, and Theorem~\ref{thm:specialbis} applies
after passing to the basic representative.
\item If $\mathbf{k}$ is algebraically closed, every basic biserial
algebra admits a bisected presentation in the sense of
\cite{VFCB98}, and in particular a bound quiver presentation;
consequently the syzygy transitions can be computed from this
presentation by linear algebra.
\item There exists a representation-infinite biserial algebra for
which no uniform compression of the reachable syzygy quiver to a
path-state or string-state quiver is possible.
\end{enumerate}
\end{theorem}

\subsection{Contents of sections}\label{subsec:contents}

The paper is organised as follows. In Section~\ref{sec:background} we
recall the background material on bound quivers, syzygies, projective
dimension, finitistic dimension, string algebras, special biserial
algebras, biserial algebras, string modules, and the Liu--Morin
closure theorem. In Section~\ref{sec:gabriel} we show that the
ordinary Gabriel quiver does not suffice to determine the global
dimension, even among string algebras. In Section~\ref{sec:universal}
we introduce the reachable syzygy quiver of an arbitrary finite
dimensional algebra and prove the universal height formula. In
Section~\ref{sec:stringsyzygy} we specialise to special biserial
algebras: using the Liu--Morin closure theorem we identify the
universal syzygy quiver with the string-syzygy quiver, prove
Theorem~\ref{thm:specialbis}, and construct an algorithm that computes
the global dimension from the bound quiver. In
Section~\ref{sec:stringalgebra} we treat string algebras, introduce
minimal path annihilators and the relation-chain quiver, and prove
Theorem~\ref{thm:string} as a compression of the string-syzygy
formula. In Section~\ref{sec:binomial} we explain why the path-state
compression of Section~\ref{sec:stringalgebra} is not available in the
presence of binomial relations, using the commutativity diamond as
the canonical paradigm, and we prove a negative result on pointwise
monomial reduction preserving the relations. In
Section~\ref{sec:biserial} we discuss general biserial algebras,
prove Theorem~\ref{thm:biserial}, and use the
Skowro\'nski--Waschb\"usch algebra as a witness that string states are
insufficient. In Section~\ref{sec:examples} we collect several
examples, accompanied by quiver diagrams, illustrating the strict
inclusions
\[
\text{gentle} \subsetneq \text{string} \subsetneq
\text{special biserial} \subsetneq \text{biserial}.
\]
In Section~\ref{sec:future} we state several open problems for future
research.

\section{Background}\label{sec:background}

In this section we recall the background material on bound quivers,
syzygies, and the four classes of algebras that we study. We
introduce bound quiver algebras in
Section~\ref{subsec:boundquivers}, recall projective covers and
syzygies in Section~\ref{subsec:syzygies}, and define the global and
finitistic dimensions in Section~\ref{subsec:gldimdef}. In
Section~\ref{subsec:specialbis} we recall the definitions of string,
special biserial, gentle, and biserial algebras, and we record the
strict inclusions between them. Section~\ref{subsec:liumorin}
introduces string modules and states the Liu--Morin closure theorem,
which is the key technical input for the construction of the
string-syzygy quiver in Section~\ref{sec:stringsyzygy}.

\subsection{Bound quivers and path algebras}\label{subsec:boundquivers}

Throughout, $\mathbf{k}$ is a field and $\Lambda$ is a finite
dimensional $\mathbf{k}$-algebra. A quiver is a quadruple
$Q = (Q_0, Q_1, s, t)$ where $Q_0$ is the set of vertices, $Q_1$ is
the set of arrows, and $s, t : Q_1 \to Q_0$ denote the source and
target functions. We assume throughout that $Q_0$ and $Q_1$ are
finite. We use small roman letters $v, w$, possibly with numerical
subscripts, to denote vertices, and small roman letters $a, b, c, d$
or Greek letters $\alpha, \beta$, possibly with numerical subscripts,
to denote arrows.

A path in $Q$ of positive length is a finite sequence
$a_1 a_2 \ldots a_n$ of elements of $Q_1$ with $n > 0$ such that
$s(a_{j+1}) = t(a_j)$ for each $1 \le j < n$. In addition, there is
a lazy path $e_v$ of length $0$ corresponding to each $v \in Q_0$.
For technical reasons we compose paths from left to right throughout,
so that $\alpha\beta$ denotes the path that first traverses $\alpha$
and then $\beta$. The source and target of a path
$a_1 a_2 \ldots a_n$ are $s(a_1)$ and $t(a_n)$ respectively. The
path algebra $\mathbf{k}Q$ is the $\mathbf{k}$-algebra whose
underlying $\mathbf{k}$-vector space has as a basis the set of all
paths in $Q$, and whose multiplication is induced by path
concatenation: the product $pq$ of two paths $p$ and $q$ is the
concatenation $pq$ if $t(p) = s(q)$, and is zero otherwise.

An ideal $I$ of $\mathbf{k}Q$ is admissible if there exists an
integer $m \ge 2$ such that
\[
(\mathbf{k}Q^+)^m \subseteq I \subseteq (\mathbf{k}Q^+)^2,
\]
where $\mathbf{k}Q^+$ denotes the arrow ideal, that is, the
two-sided ideal generated by the arrows of $Q$. We call the pair
$(Q, I)$ a bound quiver and the quotient algebra
$\Lambda = \mathbf{k}Q/I$ a bound quiver algebra. Since $I$ is
admissible, $\Lambda$ is finite dimensional.

We work with right $\Lambda$-modules. For $v \in Q_0$, the
corresponding indecomposable projective right $\Lambda$-module is
$P_v = e_v \Lambda$, and the corresponding simple right
$\Lambda$-module is $S_v = P_v / \rad P_v$, where $\rad P_v$
denotes the radical of $P_v$. We abbreviate $\rad \Lambda$ for the
Jacobson radical of $\Lambda$ considered as a right module over
itself.

For an algebra $\Lambda$ that is not given by a bound quiver, we say
that $\Lambda$ is basic if $\Lambda / \rad \Lambda$ is a product of
copies of division rings, and split basic if it is a product of
copies of $\mathbf{k}$. Over an algebraically closed field every
finite dimensional algebra is Morita equivalent to a split basic
algebra, and every split basic algebra is isomorphic to a bound
quiver algebra. Since both the global and the finitistic dimension
are Morita invariants, we may assume throughout that
$\Lambda = \mathbf{k}Q/I$ is split basic. Over an algebraically
closed field this entails no loss of generality; over a general
field we impose this as an explicit hypothesis, or extend scalars
when only the global dimension is concerned.

\subsection{Projective covers, syzygies, and projective dimension}
\label{subsec:syzygies}

Let $\Lambda$ be a finite dimensional $\mathbf{k}$-algebra and let
$M$ be a finitely generated right $\Lambda$-module. A projective
cover of $M$ is a pair $(P(M), \pi_M)$ where $P(M)$ is a projective
right $\Lambda$-module and $\pi_M : P(M) \to M$ is a surjective
$\Lambda$-module homomorphism with superfluous kernel
$\ker \pi_M \subseteq \rad P(M)$. Projective covers exist and are unique up to isomorphism whenever $\Lambda$ is an Artinian ring; we
refer to \cite[Chapter~I]{ASS06} for the standard construction.

The first syzygy of $M$ is defined as
\[
\Omega_\Lambda M := \ker \pi_M,
\]
and the $n$-th syzygy is defined inductively by
$\Omega_\Lambda^0 M = M$ and
$\Omega_\Lambda^{n+1} M = \Omega_\Lambda(\Omega_\Lambda^n M)$ for
$n \ge 0$.

The projective dimension of $M$, denoted $\pd_\Lambda M$, is the
smallest nonnegative integer $n$ such that $\Omega_\Lambda^n M$ is
projective, and is defined to be $\infty$ if no such $n$ exists. We
adopt the convention $\pd_\Lambda 0 = -\infty$, although the zero
module will not appear as a state in the quivers we construct below.

The following two lemmas are standard homological facts that will be
used repeatedly.

\begin{lemma}\label{lem:simplegldim}
Let $\Lambda$ be a finite dimensional $\mathbf{k}$-algebra and let
$S_1, \ldots, S_r$ be a complete set of pairwise nonisomorphic
simple right $\Lambda$-modules. Then
\[
\gldim\Lambda = \sup_{1 \le i \le r} \pd_\Lambda S_i.
\]
\end{lemma}

\begin{proof}
The inequality $\ge$ is clear. Let $n = \sup_i \pd_\Lambda S_i$ and
assume that $n < \infty$. Let $J = \rad \Lambda$. Since $\Lambda$
is finite dimensional, $J$ is nilpotent, so $J^N = 0$ for some
integer $N$. Every finitely generated right $\Lambda$-module $M$
admits the finite radical filtration
\[
M \supseteq MJ \supseteq MJ^2 \supseteq \cdots \supseteq MJ^N = 0.
\]
Each factor $MJ^q / MJ^{q+1}$ is a module over the semisimple
algebra $\Lambda/J$, hence a finite direct sum of simple modules.
Since a finite direct sum of projective resolutions of length at
most $n$ is a projective resolution of length at most $n$, every
semisimple module has projective dimension at most $n$.

For any short exact sequence $0 \to L \to M \to N \to 0$, the long
exact sequence in $\operatorname{Ext}$ gives
$\pd_\Lambda M \le \max\{\pd_\Lambda L, \pd_\Lambda N\}$. Indeed,
if $\pd_\Lambda L \le n$ and $\pd_\Lambda N \le n$, then
$\operatorname{Ext}_\Lambda^{i+1}(N, X) = 0$ and
$\operatorname{Ext}_\Lambda^{i+1}(L, X) = 0$ for all $i \ge n$,
which forces $\operatorname{Ext}_\Lambda^{i+1}(M, X) = 0$. Applying
this inequality successively upward through the radical filtration,
starting from the zero module at the bottom, shows that
$\pd_\Lambda M \le n$. Therefore $\gldim\Lambda \le n$.
\end{proof}

\begin{lemma}\label{lem:additivity}
For finitely generated right $\Lambda$-modules $M$ and $N$ one has
\[
\Omega_\Lambda(M \oplus N) \cong \Omega_\Lambda M \oplus
\Omega_\Lambda N.
\]
More generally, the first syzygy of a finite direct sum is the direct
sum of the first syzygies of the summands.
\end{lemma}

\begin{proof}
The direct sum of projective covers $P(M) \to M$ and $P(N) \to N$
is a projective cover for $M \oplus N$. Taking exact kernels gives
the assertion. The general statement follows by induction on the
number of summands.
\end{proof}

\subsection{Global and finitistic dimensions}\label{subsec:gldimdef}

Let $\Lambda$ be a finite dimensional $\mathbf{k}$-algebra. The
(right) global dimension of $\Lambda$ is
\[
\gldim\Lambda :=
\sup\{\pd_\Lambda M : M \text{ a finitely generated right }
\Lambda\text{-module}\}.
\]
For finite dimensional algebras this coincides with the usual
supremum over all right $\Lambda$-modules; see
\cite[Chapter~I]{ASS06}. By Lemma~\ref{lem:simplegldim}, it is the
supremum of the projective dimensions of the simple right
$\Lambda$-modules.

The finitistic dimension of $\Lambda$ is
\[
\findim\Lambda :=
\sup\{\pd_\Lambda M : M \text{ a finitely generated right }
\Lambda\text{-module with } \pd_\Lambda M < \infty\}.
\]
It is clear that $\findim\Lambda \le \gldim\Lambda$, with equality
whenever $\gldim\Lambda < \infty$. The finitistic dimension
conjecture asserts that $\findim\Lambda < \infty$ for every finite
dimensional algebra; we refer to \cite[Section~I.6]{ASS06} for the
classical formulation, and to \cite{EHIS04} for the study of the
finitistic dimension in related classes of algebras.

\begin{remark}\label{rem:gldimvsfindimbasic}
If $\gldim\Lambda < \infty$, then every finitely generated right
$\Lambda$-module has finite projective dimension, and hence the
finitistic dimension coincides with the global dimension. If
$\gldim\Lambda = \infty$, the finitistic dimension may be either
finite or infinite.
\end{remark}

\subsection{String, special biserial, and biserial algebras}
\label{subsec:specialbis}

In this subsection we recall the four classes of algebras that we
study. Throughout, $\Lambda = \mathbf{k}Q/I$ is a basic finite
dimensional algebra.

\begin{definition}\label{def:specialbis}
A basic finite dimensional algebra $\Lambda = \mathbf{k}Q/I$ is
special biserial if the following conditions hold.
\begin{enumerate}[label=\textup{(SB\arabic*)}]
\item Every vertex of $Q$ is the source of at most two arrows and the
target of at most two arrows.
\item For every arrow $\alpha \in Q_1$, there is at most one arrow
$\beta$ with $s(\beta) = t(\alpha)$ and $\alpha\beta \notin I$, and
at most one arrow $\gamma$ with $t(\gamma) = s(\alpha)$ and
$\gamma\alpha \notin I$.
\end{enumerate}
An arbitrary finite dimensional algebra is called special biserial
if it is Morita equivalent to a basic special biserial algebra.
\end{definition}

\begin{definition}\label{def:string}
A special biserial algebra $\Lambda = \mathbf{k}Q/I$ is a string
algebra if the ideal $I$ is generated by paths. Equivalently, the
presentation is monomial in addition to satisfying conditions
\textup{(SB1)} and \textup{(SB2)} of
Definition~\ref{def:specialbis}.
\end{definition}

\begin{definition}\label{def:gentle}
A string algebra $\Lambda = \mathbf{k}Q/I$ is a gentle algebra if
$I$ is generated by paths of length two and, for every arrow
$\alpha \in Q_1$, there is at most one arrow $\beta \in Q_1$ with
$s(\beta) = t(\alpha)$ and $\alpha\beta \in I$, and at most one
arrow $\gamma \in Q_1$ with $t(\gamma) = s(\alpha)$ and
$\gamma\alpha \in I$.
\end{definition}

Recall that a module is called uniserial if its submodule lattice is
a finite chain, that is, its submodules are totally ordered by
inclusion.

\begin{definition}\label{def:biserial}
A finite dimensional algebra $\Lambda$ is biserial if for every
indecomposable projective right module $P$ and every indecomposable
projective left $\Lambda$-module $L$, one can write
\[
\rad P = U + V, \qquad \rad L = U' + V',
\]
where $U, V, U', V'$ are uniserial modules and each of the
intersections $U \cap V$ and $U' \cap V'$ is either $0$ or simple. We
allow one of $U, V$ and one of $U', V'$ to be zero.
\end{definition}

The condition is stated for both left and right modules, following
\cite{SW83}.

The following strict inclusions follow from the definitions:
\[
\text{gentle} \subsetneq \text{string} \subsetneq
\text{special biserial} \subsetneq \text{biserial}.
\]
The first inclusion is proper, as Examples~\ref{ex:gentle} and
\ref{ex:stringnotgentle} in Section~\ref{sec:examples} show. The
second inclusion is proper, as the commutativity diamond in
Section~\ref{sec:binomial} shows. The third inclusion is proper, as
the Skowro\'nski--Waschb\"usch algebra of
Section~\ref{sec:biserial} shows.

\begin{remark}\label{rem:gabrielinsuff}
The Gabriel quiver $Q$ records the vertices and arrows but not the
ideal of relations $I$. It is therefore insufficient for the
computation of the global dimension. In Section~\ref{sec:gabriel} we
construct two string algebras with the same Gabriel quiver but with
different global dimensions.
\end{remark}

\subsection{String modules and the Liu--Morin closure theorem}
\label{subsec:liumorin}

We now recall the notion of a string module and the closure theorem
of Liu and Morin. Our presentation follows \cite[Section~2.2]{BR87}
and \cite[Section~2.4]{Lak16}.

\begin{definition}\label{def:stringdiagram}
Let $\Lambda = \mathbf{k}Q/I$ be a bound quiver algebra. A
$\Lambda$-string diagram is a finite connected quiver $W$ whose
underlying unoriented graph is a simple path (possibly with one
vertex), together with a quiver morphism $\ell : W \to Q$ such that
no two distinct edges of $W$ incident to the same vertex with the
same orientation (both outgoing or both incoming) map to the same
arrow of $Q$, and such that the associated $\mathbf{k}Q$-module
$M(W)$, defined below, is annihilated by the ideal $I$.
\end{definition}

Let $W$ be a string diagram. The associated right module $M(W)$ has
a $\mathbf{k}$-basis $\{z_x : x \in V(W)\}$. The idempotent $e_i$
acts by $e_i \cdot z_x = z_x$ if $\ell(x) = i$, and by
$e_i \cdot z_x = 0$ otherwise. For an arrow $\alpha \in Q_1$, the
action sends $z_x$ to $z_y$ if $W$ contains an oriented edge
labelled $\alpha$ from $x$ to $y$, and sends $z_x$ to $0$ otherwise.
The condition that $I$ annihilates $M(W)$ ensures that for every
relation $\rho \in I$, the induced linear transformation
$M(W) \xrightarrow{\rho} M(W)$ is zero.

More generally, given a right $\Lambda$-module $M$ and a
$\mathbf{k}$-basis $\mathcal{B}$ of $M$ consisting of homogeneous
elements $b \in M e_{v_b}$, the \emph{coefficient graph} of $M$
with respect to $\mathcal{B}$ is the directed graph with vertex set
$\mathcal{B}$ and a directed edge $b \xrightarrow{\alpha} b'$
labelled by $\alpha \in Q_1$ whenever $b'$ appears with a nonzero
scalar coefficient in the expansion of $b \cdot \alpha$ in
$\mathcal{B}$. A string module is a module that admits a homogeneous
basis whose coefficient graph is a $\Lambda$-string diagram.

\begin{remark}\label{rem:relationcondition}
For a monomial relation $p \in I$, the condition that $p$
annihilates $M(W)$ means that $W$ contains no directed subpath
labelled by $p$. For a binomial relation
$\alpha\beta - \lambda\gamma\delta \in I$ with
$\lambda \in \mathbf{k}^\ast$, the condition requires that
$(\alpha\beta - \lambda\gamma\delta) \cdot z_x = 0$ for every vertex
$x$ of $W$; since the underlying undirected graph of $W$ is a simple
path, this means that $W$ contains neither $\alpha\beta$ nor
$\gamma\delta$ as a directed subpath.
\end{remark}

A right module is a string module if it is isomorphic to $M(W)$ for
some string diagram $W$. A one-vertex diagram generates a simple
module. A disconnected union of string diagrams represents the direct
sum of its connected components.

\begin{example}\label{ex:stringdiagramsimple}
For every $v \in Q_0$, the simple module $S_v$ is a string module:
its string diagram is the one-vertex quiver whose sole vertex maps
to $v$. $\triangle$
\end{example}

\begin{example}\label{ex:stringdiagramprojective}
For $v \in Q_0$, the indecomposable projective module
$P_v = e_v \Lambda$ is a string module if and only if its coefficient
graph (with respect to the basis of surviving paths from $v$) has
underlying undirected graph a simple path. The property of $P_v$
being a string module neither implies nor is implied by the
uniseriality of $\rad P_v$.

Stringness does not imply uniseriality: for the string algebra
$\mathbf{k}Q$ with $Q$ the quiver with two parallel arrows
$\alpha, \beta : 1 \to 2$, the projective
$P_1 = \langle e_1, \alpha, \beta\rangle$ is a string module (its
coefficient graph has underlying undirected graph the simple path
$\alpha - e_1 - \beta$), but $\rad P_1 = S_2 \oplus S_2$ is not
uniserial.

Uniseriality does not imply stringness: the projective
$P_1 = \langle e_1, a, ac, z\rangle$ of the
Skowro\'nski--Waschb\"usch algebra of Section~\ref{sec:biserial} has
uniserial radical $aC$, but each vertex weight space $P_1 e_i$
($i = 1, 2, 3, 4$) is at most one-dimensional, so its homogeneous
basis is unique up to nonzero scalar multiples. The underlying
undirected graph of its coefficient graph is therefore rigidly the
union of a triangle spanned by $a$, $ac$, $z$ with a leaf $e_1$
attached at $a$, which is not a simple path, so $P_1$ is not a
string module.
$\triangle$
\end{example}

The diagrammatic description is equivalent to the usual description
of string modules using arrows and their formal inverses; we refer
to \cite[Section~2.4]{Lak16} for the details. Its advantage here is
that the relation condition is visible as an algebraic condition on
the module $M(W)$.

A band is a cyclic string, that is, a cyclic quiver with a morphism
to $Q$ satisfying the same relation conditions as a string diagram,
which is not a proper power of a smaller cyclic string. For a band
$\mathfrak{b}$, an integer $n \ge 1$, and an irreducible monic
polynomial $f \in \mathbf{k}[T]$ with $f(0) \ne 0$, there is a band
module $B(\mathfrak{b}, n, f)$; over an algebraically closed field
these are written $B(\mathfrak{b}, n, \lambda)$ with
$\lambda \in \mathbf{k}^\ast$. We refer to \cite[Section~2.3]{Lak16}
for the construction. The following classification result was
established by Butler and Ringel for string algebras \cite{BR87}
and extended to all special biserial algebras by Wald and
Waschb\"usch \cite{WW85}.

\begin{theorem}[\cite{BR87,WW85}]\label{thm:butlerringel}
Let $\Lambda$ be a special biserial algebra. Then every finite
dimensional indecomposable right $\Lambda$-module is exactly one of
the following:
\begin{enumerate}[label=\textup{(\roman*)}]
\item a string module;
\item a band module;
\item a non-uniserial projective-injective module.
\end{enumerate}
The third class occurs exclusively when $\Lambda$ possesses binomial
relations (commutativity relations), and contains only finitely many
isomorphism classes. In particular, if $\Lambda$ is a string
algebra, every indecomposable module is uniquely a string module or
a band module.
\end{theorem}

In the next two paragraphs we assume for simplicity that
$\mathbf{k}$ is algebraically closed, so that band modules are
written $B(\mathfrak{b}, n, \lambda)$ with
$\lambda \in \mathbf{k}^\ast$. The general case is analogous, with
$\lambda$ replaced by an irreducible polynomial $f$. We follow
\cite[Section~2.4]{Lak16} and \cite[Section~4]{CB98}. Adopting the
standard indexing of this literature, which writes strings
right-to-left independent of the algebraic path composition
convention of Section~\ref{subsec:boundquivers}, let
$\mathfrak{u} = \alpha_n \ldots \alpha_2 \alpha_1$ be a string and
let $\mathfrak{v} = \alpha_i \ldots \alpha_j$ with
$1 \le j \le i \le n$ be a substring of $\mathfrak{u}$. We say that
$\mathfrak{v}$ is an image substring of $\mathfrak{u}$ if either
$\alpha_{i+1}$ is inverse or $i = n$, and either $\alpha_{j-1}$ is
direct or $j = 1$. Dually, $\mathfrak{v}$ is a factor substring of
$\mathfrak{u}$ if the analogous conditions hold after reversing the
string. If $\mathfrak{v}$ is an image substring of $\mathfrak{u}$,
there is a canonical monomorphism
$M(\mathfrak{v}) \hookrightarrow M(\mathfrak{u})$; if
$\mathfrak{v}$ is a factor substring of a string $\mathfrak{w}$,
there is a canonical epimorphism
$M(\mathfrak{w}) \twoheadrightarrow M(\mathfrak{v})$. A standard
graph map $f : M(\mathfrak{w}) \to M(\mathfrak{u})$ is a composition
of such two morphisms, and the intermediate string $\mathfrak{v}$ is
the string associated to $f$.

For a band module $B(\mathfrak{b}, n, \lambda)$ and a finite factor
substring $\mathfrak{v}$ of ${}^{\infty}\mathfrak{b}^{\infty}$, there are
canonical maps $B(\mathfrak{b}, n, \lambda) \to M(\mathfrak{v})$
indexed by $\Hom_{\mathbf{k}}(\mathbf{k}^n, \mathbf{k})$. Graph
maps of type $\mathrm{BS}$ are formal compositions of these with
graph maps existing between string modules. Graph maps of type
$\mathrm{SB}$ and $\mathrm{BB}$ are constructed dually. The
following foundational mapping theorem is due to Krause and
Crawley-Boevey.

\begin{theorem}[\cite{Kra91,CB98}]\label{thm:krausecrawley}
Assume that $\mathbf{k}$ is algebraically closed. Let $\Lambda$ be
a special biserial algebra and let $M, N$ be finite dimensional
indecomposable right $\Lambda$-modules that are either string
modules or band modules. Then the $\mathbf{k}$-vector space
$\Hom_\Lambda(M, N)$ is endowed with a $\mathbf{k}$-basis consisting
entirely of graph maps.
\end{theorem}

The following theorem of Liu and Morin is the key technical input
that allows us to replace arbitrary module states by string states
in Section~\ref{sec:stringsyzygy}.

\begin{proposition}[{Liu--Morin closure, \cite[Proposition~2.2]{LM04}}]
\label{thm:liumorin}
Let $\Lambda$ be a finite dimensional special biserial algebra over
a field $\mathbf{k}$, and let $X$ be a finite dimensional string
module. Then the first syzygy decomposes as
\[
\Omega_\Lambda X \cong P \oplus X_1 \oplus \cdots \oplus X_m,
\]
where $P$ is a projective module and each $X_j$ is an
indecomposable string module. Moreover, the string diagrams
representing the $X_j$ can be constructed explicitly from a string
diagram for $X$ using only the bound quiver presentation of
$\Lambda$.
\end{proposition}

\begin{remark}\label{rem:liumorinlimits}
The proposition is stated for finite dimensional string modules over
a finite dimensional special biserial algebra. The hypothesis of
finite dimensionality on $\Lambda$ is essential; the result does not
apply to infinite dimensional algebras.
\end{remark}

Proposition~\ref{thm:liumorin} has the following consequence, which
will be used repeatedly.

\begin{corollary}\label{cor:stringclosure}
Let $\Lambda$ be a special biserial algebra and let $S$ be a simple
string module. Then for every $n \ge 0$, every indecomposable
nonprojective direct summand of $\Omega_\Lambda^n S$ is a string
module.
\end{corollary}

\begin{proof}
Every simple module is a string module by
Example~\ref{ex:stringdiagramsimple}. Applying
Proposition~\ref{thm:liumorin} inductively to the syzygies, and
using the additivity of $\Omega_\Lambda$ from
Lemma~\ref{lem:additivity}, gives the assertion.
\end{proof}

\section{The ordinary Gabriel quiver is insufficient}\label{sec:gabriel}

In this section we show that the ideal of relations $I$ is
indispensable for the computation of the global dimension. We
compare two string algebras with the same Gabriel quiver: the global
dimension of one is finite, while that of the other is infinite.
This shows that no formula for the global dimension depending only
on the Gabriel quiver can exist, even in the class of string
algebras. In Section~\ref{subsec:twostring} we introduce the two
algebras and compute their global dimensions, and in
Section~\ref{subsec:gabrielconsequences} we draw the consequences
for special biserial and biserial algebras. This motivates the use
of the bound quiver $(Q, I)$ throughout the rest of the paper.

\subsection{Two string algebras with the same Gabriel quiver}
\label{subsec:twostring}

Throughout this subsection $Q$ denotes the oriented $2$-cycle with
vertices $1$, $2$ and arrows
\[
\alpha : 1 \to 2, \qquad \beta : 2 \to 1,
\]
as in Figure~\ref{fig:twocycle}. We consider the two bound quiver
algebras
\[
A := \mathbf{k}Q/(\alpha\beta), \qquad
B := \mathbf{k}Q/(\alpha\beta, \beta\alpha),
\]
where $(\alpha\beta)$ denotes the two-sided ideal of $\mathbf{k}Q$
generated by the path $\alpha\beta$, and similarly for
$(\alpha\beta, \beta\alpha)$.

\vspace{-1.0em}
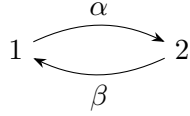
\begin{figure}[h]
\centering
\begin{tikzpicture}[scale=1.1,>=Stealth]
\node (1) at (0,0) {$1$};
\node (2) at (2,0) {$2$};
\draw[->] (1) to[bend left=25] node[above] {$\alpha$} (2);
\draw[->] (2) to[bend left=25] node[below] {$\beta$} (1);
\end{tikzpicture}
\caption{The oriented $2$-cycle $Q$.}
\label{fig:twocycle}
\end{figure}

We first record an isomorphism that will be used in the proof of the
main theorem of this subsection.

\begin{lemma}\label{lem:isomorphism}
In the algebra $A = \mathbf{k}Q/(\alpha\beta)$ the right $A$-modules
$P_1 = e_1A$ and $\beta A$ are isomorphic. In particular, $\beta A$
is a projective right $A$-module.
\end{lemma}

\begin{proof}
Define a $\mathbf{k}$-linear map
\[
\phi : P_1 \longrightarrow A, \qquad x \longmapsto \beta x.
\]
In the left-to-right path composition convention, the concatenation
$\beta x$ is a well-defined path for every $x \in e_1 A$, because
$t(\beta) = 1 = s(x)$. Since $A$ is associative, $\phi$ is a right
$A$-module homomorphism:
\[
\phi(x\cdot a) = \beta(xa) = (\beta x)a = \phi(x)\cdot a
\]
for all $x \in P_1$ and $a \in A$.

The image of $\phi$ is contained in $\beta A$. Since $P_1$ is
generated as a right $A$-module by $e_1$ and $\phi(e_1) = \beta$, the
image is exactly $\beta A$. Hence $\phi$ is a surjection
$P_1 \to \beta A$.

We show that $\phi$ is injective. The right $A$-module $P_1 = e_1A$
has $\mathbf{k}$-basis $\{e_1, \alpha\}$: the only paths in $Q$
starting at $1$ are $e_1$, $\alpha$, $\alpha\beta$, $\alpha\beta\alpha$,
$\ldots$, and all paths of length at least two contain $\alpha\beta$
as a subpath and hence vanish in $A$. Let
$x = \lambda e_1 + \mu\alpha \in P_1$ with $\lambda, \mu \in
\mathbf{k}$, and suppose that $\phi(x) = 0$ in $A$, that is,
$\lambda\beta + \mu\beta\alpha = 0$. The paths $\beta$ and
$\beta\alpha$ both start at $2$, both are nonzero in $A$ (neither
contains $\alpha\beta$ as a subpath), and are distinct
$\mathbf{k}$-basis elements of $A$. Hence their linear combination
is zero in $A$ if and only if both coefficients vanish, so
$\lambda = \mu = 0$ and $x = 0$. Therefore $\phi$ is injective, and
consequently an isomorphism of right $A$-modules. The final
assertion follows because $P_1 = e_1A$ is projective.
\end{proof}

\begin{theorem}\label{thm:twostring}
The algebras $A = \mathbf{k}Q/(\alpha\beta)$ and
$B = \mathbf{k}Q/(\alpha\beta, \beta\alpha)$ are string algebras with
the same Gabriel quiver $Q$, but
\[
\gldim A = 2, \qquad \gldim B = \infty.
\]
Consequently the ordinary Gabriel quiver alone cannot determine the
global dimension, even in the class of string algebras.
\end{theorem}

\begin{proof}
Both ideals $(\alpha\beta)$ and $(\alpha\beta, \beta\alpha)$ are
monomial and admissible. For $(\alpha\beta)$ the admissibility is
verified as follows: the only paths of length three in $Q$ are
$\alpha\beta\alpha$ and $\beta\alpha\beta$, and both contain
$\alpha\beta$ as a subpath, so
$(\mathbf{k}Q^+)^3 \subseteq (\alpha\beta) \subseteq (\mathbf{k}Q^+)^2$.
For $(\alpha\beta, \beta\alpha)$ the admissibility is immediate,
since $(\alpha\beta, \beta\alpha) = (\mathbf{k}Q^+)^2$.

The quiver $Q$ satisfies conditions \textup{(SB1)} and \textup{(SB2)}
of Definition~\ref{def:specialbis}: each vertex is the source of
exactly one arrow and the target of exactly one arrow, and each
arrow has exactly one possible predecessor and exactly one possible
successor. Since the defining ideals are monomial, both $A$ and $B$
are string algebras in the sense of Definition~\ref{def:string}.
Since $Q$ has no parallel arrows, the Gabriel quiver of both
algebras is $Q$ itself.

We compute the global dimensions of $A$ and $B$ separately.

\emph{The algebra $A = \mathbf{k}Q/(\alpha\beta)$.} The
$\mathbf{k}$-basis of $A$ consists of the residue classes of the
paths that do not lie in $(\alpha\beta)$. Since $(\alpha\beta)$ is
generated by the single path $\alpha\beta$, and every path of length
at least three in $Q$ contains $\alpha\beta$ as a subpath, the
$\mathbf{k}$-basis of $A$ is
\[
\{e_1, e_2, \alpha, \beta, \beta\alpha\}.
\]
The indecomposable projective right $A$-modules have the
$\mathbf{k}$-bases
\[
P_1 = \langle e_1, \alpha\rangle, \qquad
P_2 = \langle e_2, \beta, \beta\alpha\rangle,
\]
consisting of the paths starting at $1$ and $2$ respectively. Their
radicals are
\[
\rad P_1 = \alpha A = \langle \alpha\rangle, \qquad
\rad P_2 = \beta A = \langle \beta, \beta\alpha\rangle.
\]
The right $A$-module $\alpha A$ is one-dimensional, and its radical
is zero because $\alpha \cdot \alpha$ and $\alpha \cdot \beta$ are not
nonzero paths in $A$. Hence $\alpha A$ is a simple module, supported
at vertex $2$, so $\alpha A \cong S_2$. Therefore
$\Omega_A S_1 \cong S_2$.

By Lemma~\ref{lem:isomorphism}, the right $A$-module $\beta A$ is
isomorphic to $P_1$ and hence is projective. Therefore
$\Omega_A S_2 \cong \beta A \cong P_1$, and the minimal projective
resolution of $S_2$ has length $1$. Consequently $\pd_A S_2 = 1$.

The minimal projective resolution of $S_1$ is then
\[
0 \to P_1 \xrightarrow{\beta\cdot} P_2 \xrightarrow{\alpha\cdot} P_1
\to S_1 \to 0,
\]
where the map $P_1 \to P_2$ is left multiplication by
$\beta : 2 \to 1$ and the map $P_2 \to P_1$ is left multiplication
by $\alpha : 1 \to 2$ (so in each case left multiplication by a path
$p : i \to j$ gives a map $P_j \to P_i$).

We verify exactness at the middle term. An element of $P_2$ has the
form $\lambda e_2 + \mu\beta + \nu\beta\alpha$ with
$\lambda, \mu, \nu \in \mathbf{k}$. Its image under $\alpha\cdot$ is
$\lambda\alpha$, since $\alpha\beta = 0$ in $A$ and hence also
$\alpha\beta\alpha = 0$. So $\alpha\cdot$ sends $x$ to zero if and
only if $\lambda = 0$. Hence
\[
\ker(\alpha\cdot) = \{\mu\beta + \nu\beta\alpha : \mu, \nu \in \mathbf{k}\}
= \beta A,
\]
which equals the image of $\beta\cdot : P_1 \to P_2$. By
Lemma~\ref{lem:isomorphism}, the map $\beta\cdot$ is injective.
Therefore the resolution is exact, and $\pd_A S_1 = 2$. By
Lemma~\ref{lem:simplegldim},
\[
\gldim A = \sup\{2, 1\} = 2.
\]

\emph{The algebra $B = \mathbf{k}Q/(\alpha\beta, \beta\alpha)$.} Since
both defining relations are paths of length two, and every path of
length at least two in $Q$ contains $\alpha\beta$ or $\beta\alpha$
as a subpath, the $\mathbf{k}$-basis of $B$ is
\[
\{e_1, e_2, \alpha, \beta\}.
\]
The indecomposable projective right $B$-modules have the
$\mathbf{k}$-bases
\[
P_1 = \langle e_1, \alpha\rangle, \qquad
P_2 = \langle e_2, \beta\rangle,
\]
with the right $B$-module structure determined by
\[
\alpha \cdot e_2 = \alpha, \quad \alpha \cdot \beta = 0, \qquad
\beta \cdot e_1 = \beta, \quad \beta \cdot \alpha = 0.
\]
The radicals are
\[
\rad P_1 = \alpha B = \langle \alpha\rangle \cong S_2, \qquad
\rad P_2 = \beta B = \langle \beta\rangle \cong S_1.
\]
Hence $\Omega_B S_1 \cong S_2$ and $\Omega_B S_2 \cong S_1$.
Therefore $\Omega_B^2 S_1 \cong S_1$ and
$\Omega_B^2 S_2 \cong S_2$. Since $S_1$ and $S_2$ are nonprojective
(their projective covers $P_1$ and $P_2$ have nonzero radicals), no
iterated syzygy of either simple module is projective.
Consequently
\[
\pd_B S_1 = \pd_B S_2 = \infty,
\]
and by Lemma~\ref{lem:simplegldim},
\[
\gldim B = \infty.
\]
\end{proof}

\begin{remark}\label{rem:resolutioncomparison}
For $A$, the minimal projective resolution of $S_1$ has length two,
whereas for $B$ it is periodic of period two:
\[
\cdots \to P_2 \xrightarrow{\alpha\cdot} P_1 \xrightarrow{\beta\cdot}
P_2 \xrightarrow{\alpha\cdot} P_1 \to S_1 \to 0.
\]
The periodic resolution of $B$ reflects the fact that
$\Omega_B^2 S_1 \cong S_1$ and $\Omega_B^2 S_2 \cong S_2$. The two
algebras therefore have different homological behaviour, even though
their Gabriel quivers coincide.
\end{remark}

\subsection{Consequences for special biserial and biserial algebras}
\label{subsec:gabrielconsequences}

Theorem~\ref{thm:twostring} has the following immediate
consequences.

\begin{corollary}\label{cor:gabrielinsufficient}
There exist two finite dimensional string algebras with isomorphic
Gabriel quivers and distinct global dimensions. Consequently no
formula for the global dimension depending only on the Gabriel
quiver can be valid, even in the class of string algebras, and
hence also in the classes of special biserial and biserial algebras.
\end{corollary}

\begin{proof}
Immediate from Theorem~\ref{thm:twostring}.
\end{proof}

\begin{corollary}\label{cor:notderived}
The algebras $A$ and $B$ of Theorem~\ref{thm:twostring} are not
derived equivalent.
\end{corollary}

\begin{proof}
For a finite dimensional algebra $\Lambda$ over a field,
$\gldim\Lambda < \infty$ if and only if every object of the bounded
derived category $D^b(\operatorname{mod}\Lambda)$ is isomorphic to a
bounded complex of finitely generated projective modules. Since any
derived equivalence preserves the subcategory of perfect complexes,
the finiteness of the global dimension is invariant under derived
equivalence. Since $\gldim A = 2 \neq \infty = \gldim B$, the two
algebras cannot be derived equivalent.
\end{proof}

\begin{remark}\label{rem:cartan}
The two algebras $A$ and $B$ have the same Gabriel quiver and the
same simple modules, but their Cartan matrices are not congruent
over $\mathbb{Z}$. We compute the Cartan matrix with the convention
$C_{ij} = \dim_{\mathbf{k}} e_i \Lambda e_j$, which records the
dimension of the space of paths from $i$ to $j$.

For $A$, the entries are $C_{11} = 1$, $C_{12} = 1$, $C_{21} = 1$, and $C_{22} = 2$. So the Cartan matrix of $A$ is
\[
C_A = \begin{pmatrix} 1 & 1 \\ 1 & 2 \end{pmatrix}.
\]
For $B$, the path $\beta\alpha$ lies in the ideal
$(\alpha\beta, \beta\alpha)$ and is therefore zero in $B$, and the
same computation gives
\[
C_B = \begin{pmatrix} 1 & 1 \\ 1 & 1 \end{pmatrix}.
\]
Derived equivalent finite dimensional algebras have congruent Cartan
matrices over $\mathbb{Z}$; in particular, their determinants must
coincide. Here $\det C_A = 1$ while $\det C_B = 0$. This provides a
second invariant distinguishing the two algebras, and it records the
homological consequence of the survival of the path $\beta\alpha$ in
$A$ and of its vanishing in $B$.
\end{remark}

\begin{remark}\label{rem:boundquiverinput}
Corollary~\ref{cor:gabrielinsufficient} shows that the ideal of
relations $I$ is indispensable for the computation of the global
dimension. The correct input is the bound quiver $(Q, I)$, not the
Gabriel quiver $Q$ alone. This motivates the constructions of
Sections~\ref{sec:universal}--\ref{sec:stringsyzygy}, where the
reachable syzygy quiver and the string-syzygy quiver are built
explicitly from the bound quiver presentation.
\end{remark}

\section{The reachable syzygy quiver of a finite dimensional algebra}
\label{sec:universal}

In this section we associate to an arbitrary finite dimensional
algebra a directed multigraph whose vertices are the isomorphism
classes of indecomposable nonprojective modules that appear as
summands of iterated syzygies of simple modules, and we show that the
projective dimension of a vertex equals one plus its directed path
height in this graph. This gives a universal formula for the global
and finitistic dimensions in terms of heights in the graph. The
formula is exact for every finite dimensional algebra, but its state
space consists of arbitrary modules, and hence it is not yet a
combinatorial statement. The rest of the paper is devoted to
identifying classes of algebras on which the quiver can be replaced
by a combinatorial object.

We follow the conventions of Section~\ref{sec:background}.
Throughout, $\Lambda$ is a finite dimensional $\mathbf{k}$-algebra,
and all modules are right $\Lambda$-modules.

\subsection{The reachable syzygy quiver}\label{subsec:reachable}

We construct a directed multigraph from the syzygies of the simple
right $\Lambda$-modules.

\begin{definition}\label{def:reachable}
The reachable syzygy quiver of $\Lambda$, denoted
$\Gamma_\Omega(\Lambda)$, is the directed multigraph defined as
follows.
\begin{enumerate}[label=\textup{(\roman*)}]
\item A vertex of $\Gamma_\Omega(\Lambda)$ is an isomorphism class
$[X]$ of an indecomposable nonprojective right $\Lambda$-module $X$
that occurs as a direct summand of $\Omega_\Lambda^n S$ for some
simple right $\Lambda$-module $S$ and some $n \ge 0$.
\item The multiplicity of an arrow $[X] \to [Y]$ is the multiplicity
of $Y$ among the indecomposable nonprojective direct summands of
$\Omega_\Lambda X$.
\end{enumerate}
Projective summands of $\Omega_\Lambda X$ are discarded.
\end{definition}

The definition is well posed. For every simple right $\Lambda$-module
$S$ and every $n \ge 0$, the syzygy $\Omega_\Lambda^n S$ is a finite
dimensional right $\Lambda$-module: indeed, $\Omega_\Lambda^n S$ is a
submodule of the projective cover $P(\Omega_\Lambda^{n-1} S)$, which
is finite dimensional over the finite dimensional algebra $\Lambda$.
Hence, by the Krull--Schmidt theorem, $\Omega_\Lambda^n S$ decomposes
into a finite number of indecomposable summands. A projective simple
module contributes no vertices: if $S$ is projective, then
$\Omega_\Lambda^0 S = S$ is projective and $\Omega_\Lambda^n S = 0$
for all $n \ge 1$.

\begin{remark}\label{rem:mult}
The multiplicities are not used in the height computations of this
section, but they record the complete nonprojective part of each
first syzygy. We retain them because they are used in
Section~\ref{sec:stringsyzygy} to compare the reachable syzygy quiver
of a special biserial algebra with its string-syzygy quiver.
\end{remark}

\begin{example}\label{ex:reachablesimple}
Let $\Lambda = \mathbf{k}Q/I$ be a string algebra. Suppose that $S$
is a simple right $\Lambda$-module and that $\Omega_\Lambda S$
decomposes as $P \oplus X$ with $P$ projective and $X$ an
indecomposable nonprojective string module. Then $[S] \to [X]$ is an
arrow of $\Gamma_\Omega(\Lambda)$. If $X$ happens to be projective,
no arrow is drawn and $[S]$ is a sink. $\triangle$
\end{example}

\begin{remark}\label{rem:reachableinfinite}
The set of vertices of $\Gamma_\Omega(\Lambda)$ need not be finite.
When $\Lambda$ has infinite global dimension and
$\Gamma_\Omega(\Lambda)$ is acyclic, the quiver must be infinite by
Theorem~\ref{thm:universalheight}\textup{(3)} below. On the other
hand, the algebra $B$ of Theorem~\ref{thm:twostring} has infinite
global dimension and its reachable syzygy quiver is finite,
consisting of two vertices joined by a directed $2$-cycle.
\end{remark}

\begin{lemma}\label{lem:reachablelocalfinite}
Every vertex of $\Gamma_\Omega(\Lambda)$ has finite out-degree.
\end{lemma}

\begin{proof}
Let $[X]$ be a vertex. Then $X$ is a finitely generated right
$\Lambda$-module, since it occurs as a direct summand of an iterated
syzygy of a simple module, and syzygies of finitely generated modules
over a finite dimensional algebra remain finitely generated. Hence
$\Omega_\Lambda X$ is finitely generated, and therefore has finitely
many indecomposable direct summands. Consequently only finitely many
arrows emanate from $[X]$.
\end{proof}

\subsection{Directed height}\label{subsec:height}

In this subsection we define directed height for vertices of a
directed multigraph, and we record the elementary facts about height
that will be used below. Throughout, a directed path in a multigraph
is a finite sequence of composable arrows; vertices and arrows of the
path may repeat.

\begin{definition}\label{def:height}
Let $x$ be a vertex of a directed multigraph. The directed height of
$x$ is
\[
\operatorname{ht}(x) :=
\sup\{\ell : \text{there is a directed path of } \ell \text{ edges
originating at } x\}.
\]
The value of $\operatorname{ht}(x)$ belongs to
$\mathbb{N}_0 \cup \{\infty\}$. A path of length zero, consisting
only of the vertex itself, is allowed, so a sink has height $0$.
\end{definition}

\begin{remark}\label{rem:pathsrepeat}
A directed path in Definition~\ref{def:height} is a composable
sequence of arrows in the quiver-theoretic sense, and repetitions of
vertices and arrows are allowed. This is essential for detecting
infinite height: a directed cycle reachable from a vertex produces
paths of arbitrary length, hence yields infinite height at that
vertex.
\end{remark}

\begin{example}\label{ex:heightcycle}
Let $\Gamma$ be the directed graph with two vertices $x$, $y$ and
arrows $x \to y$ and $y \to x$. Then
$\operatorname{ht}(x) = \operatorname{ht}(y) = \infty$, because the
cycle can be traversed arbitrarily many times. If the arrow
$y \to x$ is removed, then $\operatorname{ht}(x) = 1$ and
$\operatorname{ht}(y) = 0$. $\triangle$
\end{example}

\begin{remark}\label{rem:heightcomparison}
If $\operatorname{ht}(x) = h < \infty$, then there is a directed path
of length $h$ starting at $x$ and no directed path of length $h + 1$
starting at $x$. If $\operatorname{ht}(x) = \infty$, then for every
$\ell \ge 0$ there is a directed path of length $\ell$ starting at
$x$.
\end{remark}

The following lemma connects directed paths in the reachable syzygy
quiver to the indecomposable nonprojective summands of iterated
syzygies.

\begin{lemma}\label{lem:pathsyzygy}
Let $\Lambda$ be a finite dimensional $\mathbf{k}$-algebra and let
$[X]$, $[Y]$ be vertices of $\Gamma_\Omega(\Lambda)$. For every
$n \ge 0$ the following are equivalent.
\begin{enumerate}[label=\textup{(\roman*)}]
\item There is a directed path of length $n$ from $[X]$ to $[Y]$ in
$\Gamma_\Omega(\Lambda)$.
\item The module $Y$ occurs as an indecomposable nonprojective direct
summand of $\Omega_\Lambda^n X$.
\end{enumerate}
If multiplicities are retained, the number of distinct directed walks
of length $n$ from $[X]$ to $[Y]$ equals the multiplicity of $Y$ in
the nonprojective part of $\Omega_\Lambda^n X$.
\end{lemma}

\begin{proof}
We argue by induction on $n$.

\emph{The case $n = 0$.} A path of length zero from $[X]$ to $[Y]$
exists if and only if $[X] = [Y]$, that is, $X \cong Y$. Since
$\Omega_\Lambda^0 X = X$, this happens if and only if $Y$ occurs as
an indecomposable nonprojective direct summand of
$\Omega_\Lambda^0 X$.

\emph{The inductive step.} Suppose the statement holds for some
$n \ge 0$. We prove it for $n + 1$.

Suppose first that there is a directed path of length $n + 1$ from
$[X]$ to $[Y]$. Factor it as a directed path of length $n$ from
$[X]$ to some intermediate vertex $[Z]$, followed by an arrow
$[Z] \to [Y]$. By the induction hypothesis, $Z$ occurs as an
indecomposable nonprojective direct summand of $\Omega_\Lambda^n X$.
Write $\Omega_\Lambda^n X \cong Z \oplus Z'$ for some module $Z'$.
By Lemma~\ref{lem:additivity},
\[
\Omega_\Lambda^{n+1} X
= \Omega_\Lambda(\Omega_\Lambda^n X)
\cong \Omega_\Lambda Z \oplus \Omega_\Lambda Z'.
\]
Since $[Z] \to [Y]$ is an arrow of $\Gamma_\Omega(\Lambda)$, the
module $Y$ occurs as an indecomposable nonprojective direct summand
of $\Omega_\Lambda Z$. Hence $Y$ occurs as an indecomposable
nonprojective direct summand of $\Omega_\Lambda^{n+1} X$. This
proves \textup{(i)} $\Rightarrow$ \textup{(ii)}.

Conversely, suppose that $Y$ occurs as an indecomposable
nonprojective direct summand of $\Omega_\Lambda^{n+1} X$. Write
$\Omega_\Lambda^{n+1} X = \Omega_\Lambda(\Omega_\Lambda^n X)$, and
decompose $\Omega_\Lambda^n X \cong \bigoplus_i M_i$ with each $M_i$
indecomposable. By Lemma~\ref{lem:additivity},
$\Omega_\Lambda^{n+1} X \cong \bigoplus_i \Omega_\Lambda M_i$. Hence
$Y$ occurs as an indecomposable nonprojective direct summand of
$\Omega_\Lambda M_j$ for some $j$. If $M_j$ were projective, then
$\Omega_\Lambda M_j = 0$, which is impossible since $Y \neq 0$.
Therefore $M_j$ is nonprojective. Set $[M_j] = [Z]$. By the
induction hypothesis, there is a directed path of length $n$ from
$[X]$ to $[Z]$. Since $Y$ occurs as an indecomposable nonprojective
direct summand of $\Omega_\Lambda Z$, there is an arrow
$[Z] \to [Y]$ of $\Gamma_\Omega(\Lambda)$. Concatenating gives a
directed path of length $n + 1$ from $[X]$ to $[Y]$. This proves
\textup{(ii)} $\Rightarrow$ \textup{(i)}.

For the statement about multiplicities, suppose that the number of
directed walks of length $n$ from $[X]$ to $[Z]$ equals the
multiplicity $m_Z$ of $Z$ in the nonprojective part of
$\Omega_\Lambda^n X$, and that the multiplicity of the arrow
$[Z] \to [Y]$ equals $n_{Z, Y}$. Then the number of directed walks of
length $n + 1$ from $[X]$ to $[Y]$ equals
$\sum_{[Z]} m_Z n_{Z, Y}$. By additivity of $\Omega_\Lambda$ and by
the Krull--Schmidt decomposition, this sum equals the multiplicity of
$Y$ in the nonprojective part of $\Omega_\Lambda^{n+1} X$.
\end{proof}

\begin{remark}\label{rem:pathsyzygyconsequence}
Lemma~\ref{lem:pathsyzygy} has the following consequence. For a
vertex $[X]$ of $\Gamma_\Omega(\Lambda)$, the iterated syzygy
$\Omega_\Lambda^n X$ is projective if and only if there is no
directed path of length $n$ starting at $[X]$.
\end{remark}

\subsection{The height formula for global dimension}
\label{subsec:gldim}

The goal of this subsection is to prove the universal height formula.

\begin{theorem}\label{thm:universalheight}
Let $\Lambda$ be a finite dimensional $\mathbf{k}$-algebra. Then the
following hold.
\begin{enumerate}[label=\textup{(\arabic*)}]
\item For every vertex $[X]$ of $\Gamma_\Omega(\Lambda)$,
\[
\pd_\Lambda X = 1 + \operatorname{ht}([X]).
\]
\item Let $\mathcal{S}_{\mathrm{np}}$ denote the set of isomorphism
classes of nonprojective simple right $\Lambda$-modules. Then
\[
\gldim\Lambda = \sup\bigl(\{0\} \cup
\{1 + \operatorname{ht}([S]) : S \in \mathcal{S}_{\mathrm{np}}\}\bigr).
\]
\item The global dimension $\gldim\Lambda$ is finite if and only if
$\Gamma_\Omega(\Lambda)$ is finite and acyclic.
\item If $\Gamma_\Omega(\Lambda)$ is finite, then
$\gldim\Lambda = \infty$ if and only if a directed cycle of
$\Gamma_\Omega(\Lambda)$ is reachable from a simple state.
\end{enumerate}
\end{theorem}

\begin{proof}
\emph{Part \textup{(1)}.} Let $[X]$ be a vertex of
$\Gamma_\Omega(\Lambda)$ and set $h = \operatorname{ht}([X])$.

Suppose first that $h < \infty$. By Definition~\ref{def:height}
there is a directed path of length $h$ starting at $[X]$, and there
is no directed path of length $h + 1$ starting at $[X]$. By
Lemma~\ref{lem:pathsyzygy}, the iterated syzygy $\Omega_\Lambda^h X$
has a nonprojective indecomposable direct summand, while
$\Omega_\Lambda^{h+1} X$ has no nonprojective indecomposable direct
summand. Hence $\Omega_\Lambda^{h+1} X$ is projective, while
$\Omega_\Lambda^h X$ is not. Therefore $\pd_\Lambda X = h + 1$.

Suppose now that $h = \infty$. By Definition~\ref{def:height}, for
every $\ell \ge 0$ there is a directed path of length $\ell$ starting
at $[X]$. By Lemma~\ref{lem:pathsyzygy}, $\Omega_\Lambda^\ell X$ has
a nonprojective indecomposable direct summand for every $\ell \ge 0$.
Hence no iterated syzygy of $X$ is projective, and
$\pd_\Lambda X = \infty$. Since $1 + \infty = \infty$, the identity
$\pd_\Lambda X = 1 + \operatorname{ht}([X])$ holds in this case as
well.

\emph{Part \textup{(2)}.} By Lemma~\ref{lem:simplegldim},
\[
\gldim\Lambda = \sup_{S \text{ simple}} \pd_\Lambda S.
\]
For a projective simple module $S$, $\pd_\Lambda S = 0$, and such a
simple contributes nothing to the supremum beyond the value $0$. For
a nonprojective simple module $S$, part \textup{(1)} gives
$\pd_\Lambda S = 1 + \operatorname{ht}([S])$. Combining these two
observations gives the displayed formula. If
$\mathcal{S}_{\mathrm{np}}$ is empty, that is, if every simple module
is projective, then the supremum of $\{0\}$ equals $0$, which is the
correct value for a semisimple algebra.

\emph{Part \textup{(3)}.} Suppose first that
$\Gamma_\Omega(\Lambda)$ is finite and acyclic. Let $N$ denote the
number of vertices of $\Gamma_\Omega(\Lambda)$. Since the quiver is acyclic, every directed path in 
$\Gamma_\Omega(\Lambda)$ has length at most $N - 1$ (with the empty quiver 
having global dimension $0$): each step along the path arrives at a vertex that
has not been visited earlier, since a repetition would produce a
directed cycle. In particular, $\operatorname{ht}(x) \le N - 1$ for
every vertex $x$ of $\Gamma_\Omega(\Lambda)$. Hence the supremum in
part \textup{(2)} is at most $N$, and $\gldim\Lambda < \infty$.

Conversely, suppose that $\gldim\Lambda = d < \infty$. For every
nonprojective simple module $S$, part \textup{(2)} gives
$\pd_\Lambda S \le d$, and part \textup{(1)} gives
$\operatorname{ht}([S]) = \pd_\Lambda S - 1 \le d - 1$.

We first show that $\Gamma_\Omega(\Lambda)$ is finite. Let $W_0$
denote the set of vertices of $\Gamma_\Omega(\Lambda)$ of the form
$[S]$ with $S$ nonprojective simple. The set $W_0$ is finite. For
$k \ge 0$, define $W_{k+1}$ as the union of $W_k$ with the set of
all out-neighbours of vertices in $W_k$. Each $W_k$ is finite by
induction on $k$, using Lemma~\ref{lem:reachablelocalfinite}. We
claim that every vertex of $\Gamma_\Omega(\Lambda)$ is contained in
$\bigcup_{k=0}^{d-1} W_k$. Indeed, let $[X]$ be a vertex. By
Definition~\ref{def:reachable}, the module $X$ is an indecomposable
nonprojective direct summand of $\Omega_\Lambda^n S$ for some simple
module $S$ and some $n \ge 0$. If $S$ is projective, then
$\Omega_\Lambda^0 S = S$ is projective and $\Omega_\Lambda^n S = 0$
for $n \ge 1$, so $S$ contributes no vertices; hence $S$ must be
nonprojective. By Lemma~\ref{lem:pathsyzygy}, there is a directed
path of length $n$ from $[S]$ to $[X]$ in $\Gamma_\Omega(\Lambda)$.
Since $[S] \in W_0$ and $\operatorname{ht}([S]) \le d - 1$, the
length $n$ satisfies $n \le d - 1$. Hence
$[X] \in W_n \subseteq \bigcup_{k=0}^{d-1} W_k$. Therefore
$\Gamma_\Omega(\Lambda)$ is contained in the finite set
$\bigcup_{k=0}^{d-1} W_k$, and is finite.

We now show that $\Gamma_\Omega(\Lambda)$ is acyclic. Suppose for
contradiction that there is a directed cycle. Every vertex of
$\Gamma_\Omega(\Lambda)$ is reachable from a nonprojective simple
state by the argument of the previous paragraph, so the cycle is
reachable from some nonprojective simple state $[S]$. Traversing the
cycle repeatedly produces directed paths from $[S]$ of arbitrary
length. Hence $\operatorname{ht}([S]) = \infty$, which by part
\textup{(1)} implies $\pd_\Lambda S = \infty$. This contradicts
$\gldim\Lambda = d < \infty$. Therefore $\Gamma_\Omega(\Lambda)$ is
acyclic.

\emph{Part \textup{(4)}.} Suppose that $\Gamma_\Omega(\Lambda)$ is
finite. If a directed cycle is reachable from a nonprojective simple
state $[S]$, then traversing the cycle repeatedly produces directed
paths from $[S]$ of arbitrary length, so
$\operatorname{ht}([S]) = \infty$, and part \textup{(2)} gives
$\gldim\Lambda = \infty$.

Conversely, suppose that $\gldim\Lambda = \infty$. By part
\textup{(2)} there is a nonprojective simple module $S$ with
$\operatorname{ht}([S]) = \infty$. Since $\Gamma_\Omega(\Lambda)$ is
finite, the set of vertices reachable from $[S]$ is finite. Because
$\operatorname{ht}([S]) = \infty$, there are directed paths from
$[S]$ of arbitrary length. Take a directed path from $[S]$ of length
strictly greater than the number of vertices reachable from $[S]$.
By the pigeonhole principle, some vertex of $\Gamma_\Omega(\Lambda)$
is visited at least twice along this path; the segment of the path
between two consecutive visits to that vertex is a directed cycle
reachable from $[S]$.
\end{proof}

\begin{remark}\label{rem:universalnotcombinatorial}
Part \textup{(2)} of Theorem~\ref{thm:universalheight} gives an exact
formula for the global dimension of an arbitrary finite dimensional
algebra in terms of the reachable syzygy quiver. The formula is not
yet a combinatorial statement, because the vertices of
$\Gamma_\Omega(\Lambda)$ are arbitrary modules. In
Section~\ref{sec:stringsyzygy} we will see that for special biserial
algebras, the reachable syzygy quiver can be replaced by a quiver
whose vertices are exclusively string modules, rendering the formula
combinatorial.
\end{remark}

\subsection{The height formula for finitistic dimension}
\label{subsec:findim}

The formula of Theorem~\ref{thm:universalheight} takes the supremum
over simple modules. The finitistic dimension requires a larger
quiver, because modules of finite projective dimension need not be
reachable from simple modules.

\begin{definition}\label{def:finitisticquiver}
The finitistic syzygy quiver of $\Lambda$, denoted
$\Gamma_\Omega^{\mathrm{fin}}(\Lambda)$, is the directed multigraph
whose vertices are the isomorphism classes $[X]$ of finitely generated indecomposable nonprojective right $\Lambda$-modules $X$ with $\pd_\Lambda X < \infty$, and whose arrow multiplicities record the
indecomposable nonprojective direct summands of the first syzygy of
each vertex. Projective summands are discarded.
\end{definition}

The definition is well posed because the vertex set is closed under
taking first syzygies. Indeed, if $X$ is an indecomposable
nonprojective right $\Lambda$-module with
$\pd_\Lambda X = n < \infty$, then $\Omega_\Lambda X$ has projective
dimension $n - 1$, and every indecomposable nonprojective direct
summand of $\Omega_\Lambda X$ has projective dimension at most
$n - 1 < \infty$. Hence every nonprojective summand of $\Omega_\Lambda X$ is finitely generated and is a vertex of $\Gamma_\Omega^{\mathrm{fin}}(\Lambda)$, and every projective summand
is discarded. The finitistic syzygy quiver is locally finite by the
same argument as in Lemma~\ref{lem:reachablelocalfinite}.

\begin{remark}\label{rem:finitisticquivercontainsreachable}
Every vertex of $\Gamma_\Omega(\Lambda)$ that has finite projective
dimension is a vertex of $\Gamma_\Omega^{\mathrm{fin}}(\Lambda)$.
The finitistic quiver may contain additional vertices that do not
occur in the iterated syzygies of any simple module.
\end{remark}

The proof of Lemma~\ref{lem:pathsyzygy} uses only the additivity of
$\Omega_\Lambda$ and the basic definition of the quiver's edges; it
therefore applies to $\Gamma_\Omega^{\mathrm{fin}}(\Lambda)$ without
change, since the vertex set of
$\Gamma_\Omega^{\mathrm{fin}}(\Lambda)$ is closed under taking
nonprojective indecomposable summands of first syzygies. We will use
this observation without further mention.

\begin{lemma}\label{lem:findimvertexfiniteheight}
Every vertex of $\Gamma_\Omega^{\mathrm{fin}}(\Lambda)$ has finite
height in $\Gamma_\Omega^{\mathrm{fin}}(\Lambda)$.
\end{lemma}

\begin{proof}
Let $[X]$ be a vertex of
$\Gamma_\Omega^{\mathrm{fin}}(\Lambda)$. By definition,
$\pd_\Lambda X = n < \infty$. We show that
$\operatorname{ht}([X]) = n - 1$.

The iterated syzygy $\Omega_\Lambda^{n-1} X$ is nonprojective: this
is clear for $n = 1$, since $\Omega_\Lambda^0 X = X$, and for
$n \ge 2$ it follows from the minimality of $n$ in the definition of
projective dimension. Every indecomposable nonprojective direct
summand of $\Omega_\Lambda^{n-1} X$ has finite projective dimension
at most $1$, hence is a vertex of
$\Gamma_\Omega^{\mathrm{fin}}(\Lambda)$. Choose such a summand $Y$.
Then $[Y]$ has out-degree $0$ in
$\Gamma_\Omega^{\mathrm{fin}}(\Lambda)$: indeed,
$\Omega_\Lambda Y$ is a direct summand of $\Omega_\Lambda^n X$,
which is projective, so $\Omega_\Lambda Y$ is projective, and every
indecomposable summand of $\Omega_\Lambda Y$ is projective and is
discarded. Hence $\operatorname{ht}([Y]) = 0$.

By the analogue of Lemma~\ref{lem:pathsyzygy} for
$\Gamma_\Omega^{\mathrm{fin}}(\Lambda)$, there is a directed path of
length $n - 1$ from $[X]$ to $[Y]$ in
$\Gamma_\Omega^{\mathrm{fin}}(\Lambda)$. On the other hand, the
iterated syzygy $\Omega_\Lambda^n X$ is projective, so it has no
nonprojective indecomposable direct summand, and by the same
analogue of Lemma~\ref{lem:pathsyzygy} there is no directed path of
length $n$ starting at $[X]$. Hence
$\operatorname{ht}([X]) = n - 1$, which is finite.
\end{proof}

\begin{theorem}\label{thm:findimheight}
Let $\Lambda$ be a finite dimensional $\mathbf{k}$-algebra. Then
\[
\findim\Lambda
= \sup\bigl(\{0\} \cup \{1 + \operatorname{ht}([X]) :
[X] \in \Gamma_\Omega^{\mathrm{fin}}(\Lambda)\}\bigr).
\]
In particular, $\findim\Lambda < \infty$ if and only if the set of
heights of the vertices of
$\Gamma_\Omega^{\mathrm{fin}}(\Lambda)$ is bounded.
\end{theorem}

\begin{proof}
Let $M$ be a finitely generated indecomposable right
$\Lambda$-module of finite projective dimension. If $M$ is projective, then
$\pd_\Lambda M = 0$. If $M$ is nonprojective, then $M$ is a vertex
of $\Gamma_\Omega^{\mathrm{fin}}(\Lambda)$ by
Definition~\ref{def:finitisticquiver}, and
Lemma~\ref{lem:findimvertexfiniteheight} gives
$\pd_\Lambda M = 1 + \operatorname{ht}([M])$.

Conversely, every indecomposable module $X$ whose isomorphism class
is a vertex of $\Gamma_\Omega^{\mathrm{fin}}(\Lambda)$ has finite
projective dimension by definition, and
$1 + \operatorname{ht}([X]) = \pd_\Lambda X$ by
Lemma~\ref{lem:findimvertexfiniteheight}.

Therefore
\[
\sup\{\pd_\Lambda M : M \text{ indecomposable, }
\pd_\Lambda M < \infty\}
= \sup\{1 + \operatorname{ht}([X]) :
[X] \in \Gamma_\Omega^{\mathrm{fin}}(\Lambda)\}.
\]

Since every finitely generated module of finite projective dimension is a finite direct sum of indecomposable modules of finite projective dimension, and since the
projective dimension of a finite direct sum equals the maximum of
the projective dimensions of its summands, the supremum over all
modules of finite projective dimension equals the supremum over the
indecomposable ones. This proves the displayed formula.

The finiteness criterion follows. The supremum is finite if and only
if the set of heights of the vertices of
$\Gamma_\Omega^{\mathrm{fin}}(\Lambda)$ is bounded. By
Lemma~\ref{lem:findimvertexfiniteheight}, every vertex has finite
height, so the set of heights is finite if and only if it is
bounded, which is the case if and only if the projective dimensions
of the indecomposable modules of finite projective dimension are
bounded above.
\end{proof}

\begin{remark}\label{rem:gldimvsfindim}
If $\gldim\Lambda < \infty$, then every finitely generated right $\Lambda$-module has
finite projective dimension, so $\Gamma_\Omega(\Lambda)$ is a full subquiver of
$\Gamma_\Omega^{\mathrm{fin}}(\Lambda)$. The two quivers need not coincide, but their
height suprema agree, and Theorems~\ref{thm:universalheight}
and~\ref{thm:findimheight} give $\findim\Lambda = \gldim\Lambda$.
This is consistent with Remark~\ref{rem:gldimvsfindimbasic}. If
$\gldim\Lambda = \infty$, the finitistic dimension may be either
finite or infinite, as the examples in
Section~\ref{sec:examples} show.
\end{remark}

\begin{remark}\label{rem:notquiver}
The reachable syzygy quiver $\Gamma_\Omega(\Lambda)$ of
Definition~\ref{def:reachable} has as its vertices isomorphism
classes of modules, not combinatorial objects such as paths or
strings. For a direct computation of the global dimension from the
bound quiver, one requires a class of algebras for which the
following two conditions hold.
\begin{enumerate}[label=\textup{(\alph*)}]
\item The first syzygy of every module appearing in the construction
decomposes into a finite direct sum of modules from a class that can
be described combinatorially from the bound quiver.
\item The projective covers, the radicals, and the syzygy
transitions can be computed explicitly from the bound quiver within
this class.
\end{enumerate}
In Section~\ref{sec:stringsyzygy} we will see that special biserial
algebras satisfy both conditions, with the class of string modules
playing the role of the combinatorial state space.
\end{remark}

\section{The string-syzygy quiver of a special biserial algebra}
\label{sec:stringsyzygy}

In this section we specialise the universal formula of
Section~\ref{sec:universal} to finite dimensional special biserial
algebras. The key input is Proposition~\ref{thm:liumorin}, which
states that the first syzygy of a string module decomposes into a
projective part and a finite direct sum of string modules. Using this
closure property we replace the arbitrary module states of the
reachable syzygy quiver by string modules, and show that the global
dimension is given by heights in a combinatorial quiver. For the
finitistic dimension we obtain a lower bound in general and an
equality when the global dimension is finite. We also construct an
algorithm that produces the transition from a string module to its
nonprojective string summands directly from the bound quiver, and we
record a consequence for symmetric special biserial algebras and
Brauer graph algebras.

Theorems~\ref{thm:gldimspecialbis} and~\ref{thm:findimspecialbis}
below together establish the statements of
Theorem~\ref{thm:specialbis}. The proof in each case is independent
of the numbering in the introduction.

We follow the conventions of Sections~\ref{sec:background}
and~\ref{sec:universal}. Throughout, $\Lambda = \mathbf{k}Q/I$ is a
finite dimensional special biserial algebra.

\subsection{The string-syzygy quiver}\label{subsec:stringsyzygyquiver}

We define the combinatorial quiver that replaces the reachable
syzygy quiver of Definition~\ref{def:reachable} in the special
biserial case.

\begin{definition}\label{def:stringsyzygyquiver}
Let $\Lambda$ be a finite dimensional special biserial algebra. The
string-syzygy quiver of $\Lambda$, denoted
$\Gamma_\Omega^{\mathrm{str}}(\Lambda)$, is the directed multigraph
defined as follows.
\begin{enumerate}[label=\textup{(\roman*)}]
\item A vertex of $\Gamma_\Omega^{\mathrm{str}}(\Lambda)$ is an
isomorphism class $[X]$ of an indecomposable nonprojective string
module $X$ that occurs as a direct summand of $\Omega_\Lambda^n S$
for some simple right $\Lambda$-module $S$ and some $n \ge 0$.
\item The multiplicity of an arrow $[X] \to [Y]$ is the multiplicity
of $Y$ among the indecomposable nonprojective string summands of
$\Omega_\Lambda X$.
\end{enumerate}
Projective summands of $\Omega_\Lambda X$ are discarded.
\end{definition}

The definition is well posed. Every simple right $\Lambda$-module is
a string module by Example~\ref{ex:stringdiagramsimple}, and by
Proposition~\ref{thm:liumorin} every indecomposable nonprojective
summand of an iterated syzygy of a string module is again a string
module. The following proposition makes the relation between the two
quivers precise.

\begin{proposition}\label{prop:universalequalsstring}
Let $\Lambda$ be a finite dimensional special biserial algebra. Then
the reachable syzygy quiver $\Gamma_\Omega(\Lambda)$ of
Definition~\ref{def:reachable} and the string-syzygy quiver
$\Gamma_\Omega^{\mathrm{str}}(\Lambda)$ of
Definition~\ref{def:stringsyzygyquiver} have the same vertex set and
the same arrow multiplicities. In particular, the two quivers
coincide as directed multigraphs.
\end{proposition}

\begin{proof}
Every vertex of $\Gamma_\Omega^{\mathrm{str}}(\Lambda)$ is a vertex
of $\Gamma_\Omega(\Lambda)$ by definition, since a string module that
occurs as a direct summand of $\Omega_\Lambda^n S$ is an
indecomposable nonprojective module that occurs as a direct summand
of the same syzygy.

Conversely, let $[X]$ be a vertex of $\Gamma_\Omega(\Lambda)$. By
Definition~\ref{def:reachable}, $X$ is an indecomposable
nonprojective right $\Lambda$-module that occurs as a direct summand
of $\Omega_\Lambda^n S$ for some simple right $\Lambda$-module $S$
and some $n \ge 0$. By Corollary~\ref{cor:stringclosure}, $X$ is a
string module. Hence $[X]$ is a vertex of
$\Gamma_\Omega^{\mathrm{str}}(\Lambda)$.

The equality of arrow multiplicities follows because both quivers
record the multiplicities of the indecomposable nonprojective
summands of the first syzygy of each vertex, and by
Corollary~\ref{cor:stringclosure} every such summand is a string
module.
\end{proof}

\begin{remark}\label{rem:stringsyzygylocalfinite}
Let $X$ be a vertex of $\Gamma_\Omega^{\mathrm{str}}(\Lambda)$,
represented by a string module $M(W)$ with string diagram $W$. The
out-degree of $[X]$ in $\Gamma_\Omega^{\mathrm{str}}(\Lambda)$ is the
number of indecomposable nonprojective string summands of
$\Omega_\Lambda X$. By Proposition~\ref{thm:liumorin} this number is
finite. In particular, the string-syzygy quiver is locally finite.
\end{remark}

\subsection{The exact formula for global dimension}
\label{subsec:gldimstringsyzygy}

We state and prove the formula for the global dimension of a special
biserial algebra.

\begin{theorem}\label{thm:gldimspecialbis}
Let $\Lambda = \mathbf{k}Q/I$ be a finite dimensional special
biserial algebra. Then the following hold.
\begin{enumerate}[label=\textup{(\arabic*)}]
\item For every vertex $[X]$ of
$\Gamma_\Omega^{\mathrm{str}}(\Lambda)$ one has
\[
\pd_\Lambda X
= 1 + \operatorname{ht}_{\Gamma_\Omega^{\mathrm{str}}(\Lambda)}([X]).
\]
\item $\gldim\Lambda = \sup\bigl(\{0\} \cup
\{1 + \operatorname{ht}_{\Gamma_\Omega^{\mathrm{str}}(\Lambda)}([S_i]) :
S_i \text{ a nonprojective simple string module}\}\bigr)$.
\item The global dimension $\gldim\Lambda$ is finite if and only if
the string-syzygy quiver
$\Gamma_\Omega^{\mathrm{str}}(\Lambda)$ is finite and acyclic.
\item If $\Gamma_\Omega^{\mathrm{str}}(\Lambda)$ is finite, then
$\gldim\Lambda = \infty$ if and only if a directed cycle of
$\Gamma_\Omega^{\mathrm{str}}(\Lambda)$ is reachable from a simple
state.
\end{enumerate}
\end{theorem}

\begin{proof}
By Proposition~\ref{prop:universalequalsstring}, the reachable
syzygy quiver $\Gamma_\Omega(\Lambda)$ and the string-syzygy quiver
$\Gamma_\Omega^{\mathrm{str}}(\Lambda)$ coincide as directed
multigraphs. The height of a vertex is computed in the same graph, so
the height formula of Theorem~\ref{thm:universalheight}\textup{(1)}
applies without change and gives
\[
\pd_\Lambda X = 1 + \operatorname{ht}([X])
\]
for every vertex $[X]$ of $\Gamma_\Omega^{\mathrm{str}}(\Lambda)$.
This is part \textup{(1)}. The formula for the global dimension in
part \textup{(2)} is obtained from
Theorem~\ref{thm:universalheight}\textup{(2)} by restricting the
supremum to the nonprojective simple states, which are precisely the
simple string modules. Parts \textup{(3)} and \textup{(4)} follow
from Theorem~\ref{thm:universalheight}\textup{(3)} and \textup{(4)}
together with the identification of the two quivers.
\end{proof}

\begin{remark}\label{rem:gldimspecialbisalgorithm}
Theorem~\ref{thm:gldimspecialbis} gives a combinatorial formula for
the global dimension of a special biserial algebra. Unlike the
universal formula of Theorem~\ref{thm:universalheight}, whose states
are general indecomposable modules, the formula here uses only
string modules as states. These can be computed from the bound quiver using
Proposition~\ref{thm:liumorin}, as we describe in
Section~\ref{subsec:algorithm}.
\end{remark}

\subsection{The string-syzygy bound for finitistic dimension}
\label{subsec:findimstringsyzygy}

The finitistic dimension requires a quiver whose vertices include all
string modules of finite projective dimension. We define this quiver
and obtain a lower bound that becomes an equality whenever the
global dimension is finite.

\begin{definition}\label{def:finitisticstringsyzygy}
Let $\Lambda$ be a finite dimensional special biserial algebra. The
finitistic string-syzygy quiver of $\Lambda$, denoted
$\Gamma_\Omega^{\mathrm{str,fin}}(\Lambda)$, is the directed
multigraph whose vertices are the isomorphism classes of
indecomposable nonprojective string modules $X$ with
$\pd_\Lambda X < \infty$, and whose arrow multiplicities record the
indecomposable nonprojective string summands of the first syzygy of
each vertex.
\end{definition}

The definition is well posed because the vertex set is closed under
taking first syzygies. If $X$ is an indecomposable nonprojective
string module with $\pd_\Lambda X < \infty$, then every
indecomposable nonprojective summand of $\Omega_\Lambda X$ is a
string module by Proposition~\ref{thm:liumorin} and has projective
dimension at most $\pd_\Lambda X - 1 < \infty$.

\begin{lemma}\label{lem:findimvertexfiniteheightstr}
Every vertex $[X]$ of $\Gamma_\Omega^{\mathrm{str,fin}}(\Lambda)$
has finite height in $\Gamma_\Omega^{\mathrm{str,fin}}(\Lambda)$,
and
\[
\pd_\Lambda X = 1 +
\operatorname{ht}_{\Gamma_\Omega^{\mathrm{str,fin}}(\Lambda)}([X]).
\]
\end{lemma}

\begin{proof}
By Proposition~\ref{thm:liumorin},
$\Gamma_\Omega^{\mathrm{str,fin}}(\Lambda)$ is a full subquiver of
$\Gamma_\Omega^{\mathrm{fin}}(\Lambda)$ that is closed under
out-neighbours. Consequently, the directed paths originating at a
string vertex $[X]$ in $\Gamma_\Omega^{\mathrm{str,fin}}(\Lambda)$
are identical to those in $\Gamma_\Omega^{\mathrm{fin}}(\Lambda)$.
The claim then follows immediately from
Lemma~\ref{lem:findimvertexfiniteheight}.
\end{proof}

\begin{theorem}\label{thm:findimspecialbis}
Let $\Lambda = \mathbf{k}Q/I$ be a finite dimensional special
biserial algebra. Then
\[
\findim\Lambda \ge \sup\bigl(\{0\} \cup
\{1 + \operatorname{ht}_{\Gamma_\Omega^{\mathrm{str,fin}}(\Lambda)}([X]) :
[X] \in \Gamma_\Omega^{\mathrm{str,fin}}(\Lambda)\}\bigr).
\]
If $\gldim\Lambda < \infty$, then equality holds.
\end{theorem}

\begin{proof}
Since $\Gamma_\Omega^{\mathrm{str,fin}}(\Lambda)$ is a subquiver of
$\Gamma_\Omega^{\mathrm{fin}}(\Lambda)$, and by
Lemma~\ref{lem:findimvertexfiniteheightstr} every vertex $[X]$ of
$\Gamma_\Omega^{\mathrm{str,fin}}(\Lambda)$ satisfies
$\pd_\Lambda X = 1 + \operatorname{ht}([X])$, the module $X$ is an
indecomposable module of finite projective dimension with
$\pd_\Lambda X = 1 + \operatorname{ht}([X])$. Taking the supremum
over all vertices yields the lower bound, by
Theorem~\ref{thm:findimheight}.

We turn to the equality when $\gldim\Lambda < \infty$. In this case
$\findim\Lambda = \gldim\Lambda$ by
Remark~\ref{rem:gldimvsfindim}. Every nonprojective simple string module is a vertex of $\Gamma_\Omega^{\mathrm{str,fin}}(\Lambda)$, since its
projective dimension is at most $\gldim\Lambda < \infty$.
Consequently, by Theorem~\ref{thm:gldimspecialbis}\textup{(2)},
\begin{align*}
\gldim\Lambda
&= \sup\bigl(\{0\} \cup \{1 + \operatorname{ht}([S_i]) :
S_i \text{ a nonprojective simple string module}\}\bigr) \\
&\le \sup\bigl(\{0\} \cup \{1 + \operatorname{ht}([X]) :
[X] \in \Gamma_\Omega^{\mathrm{str,fin}}(\Lambda)\}\bigr).
\end{align*}

Conversely, every vertex $[X]$ of
$\Gamma_\Omega^{\mathrm{str,fin}}(\Lambda)$ satisfies
$\pd_\Lambda X \le \gldim\Lambda$ by the definition of the global
dimension, and hence $1 + \operatorname{ht}([X]) \le \gldim\Lambda$.
Therefore the right-hand side is at most $\gldim\Lambda$, and the
two inequalities give equality.
\end{proof}

\begin{remark}\label{rem:findimspecialbisreachable}
The finitistic string-syzygy quiver may contain vertices that are
not reachable from the simple modules. For the global dimension this
does not matter, because by Lemma~\ref{lem:simplegldim} it suffices
to consider the simple states. For the finitistic dimension, all
string modules of finite projective dimension must be considered,
including those that do not appear in the minimal resolutions of
simple modules. Equality is stated under the stronger hypothesis $\gldim\Lambda < \infty$, which guarantees that every string module has finite projective dimension.
\end{remark}

\begin{remark}\label{rem:findimspecialbisgldim}
If $\gldim\Lambda < \infty$, then $\findim\Lambda = \gldim\Lambda$.
This is a general property of finite dimensional algebras, recorded
in Remark~\ref{rem:gldimvsfindim}; for special biserial algebras it
follows from Theorem~\ref{thm:findimspecialbis} together with
Theorem~\ref{thm:gldimspecialbis}.
\end{remark}

\subsection{Computing a transition from the bound quiver}
\label{subsec:algorithm}

We describe the construction that computes the string-syzygy quiver
directly from the bound quiver presentation. The construction is a
direct transcription of the proof of
Proposition~\ref{thm:liumorin}, and can be implemented
algorithmically.

\begin{construction}[One string-syzygy transition]
\label{con:stringtransition}
Let $\Lambda = \mathbf{k}Q/I$ be a finite dimensional special
biserial algebra and let $W$ be a string diagram with associated
string module $X = M(W)$. The following procedure computes the
indecomposable nonprojective string summands of $\Omega_\Lambda X$.
\begin{enumerate}[label=\textup{(\arabic*)}]
\item Identify the source vertices $x_1, \ldots, x_r$ of the quiver
$W$ (those of in-degree $0$ in $W$); the residue classes
$\overline{z}_{x_1}, \ldots, \overline{z}_{x_r}$ form a canonical
$\mathbf{k}$-basis of the top $X / X\rad\Lambda$, where $z_{x_j}$
lies at quiver vertex $i_j = \ell(x_j)$.
\item Form the projective cover
\[
P(X) = \bigoplus_{j=1}^r P_{i_j},
\]
and define the surjection $\pi : P(X) \to X$ by sending the
generator $e_{i_j}$ of the $j$-th summand to $z_{x_j}$.
\item Express each $P_{i_j}$ in a path basis modulo $I$. The matrix
of $\pi$ is determined by following the labelled edges of $W$ and
reducing paths using the relations in $I$.
\item Compute the kernel $\ker \pi$ by finite linear algebra over
$\mathbf{k}$.
\item By Proposition~\ref{thm:liumorin}, the nonprojective
components of this kernel are string diagrams $W_1, \ldots, W_m$
such that
\[
\Omega_\Lambda M(W)
\cong P \oplus M(W_1) \oplus \cdots \oplus M(W_m),
\]
where $P$ is projective.
\item Insert an arrow $[W] \to [W_j]$ in
$\Gamma_\Omega^{\mathrm{str}}(\Lambda)$ for each $j$, retaining the
multiplicity.
\end{enumerate}
\end{construction}

\begin{remark}\label{rem:binomialconstruct}
Step~\textup{(4)} of Construction~\ref{con:stringtransition} is
where a binomial relation affects the computation. If
$p - \lambda q \in I$, then the kernel can contain a vector of the
form $u - \lambda v$ supported on two distinct path branches.
Replacing the binomial relation by the two zero relations
$p = 0 = q$ changes the kernel and may change the global dimension;
we illustrate this in Section~\ref{sec:binomial}.
\end{remark}

\begin{algorithm}[Breadth-first computation]\label{alg:breadthfirst}
Let $\Lambda = \mathbf{k}Q/I$ be a finite dimensional special
biserial algebra. The following procedure computes the global
dimension of $\Lambda$.
\begin{enumerate}[label=\textup{(\arabic*)}]
\item Initialise a queue with all nonprojective simple string
diagrams.
\item When a state $[W]$ is removed from the queue, apply
Construction~\ref{con:stringtransition} and discard projective
summands. For each nonprojective string summand $[W_j]$, add the
transition arrow $[W] \to [W_j]$ to the graph; if $[W_j]$ has not
been visited previously, record $[W_j]$ as a vertex and push it onto
the queue. If a newly added arrow creates a directed cycle reachable
from a simple state, output $\gldim\Lambda = \infty$ and terminate.
\item If the queue empties, the reachable part of the
string-syzygy quiver is finite. Topologically order it and set
\[
d(W) := 1 + \max\{d(Y) : [W] \to [Y]\},
\]
with $d(W) = 1$ for a sink.
\item Output $\max\{d(S) : S \text{ a nonprojective simple}\}$, or
$0$ if there is no such simple.
\end{enumerate}
\end{algorithm}

\begin{remark}\label{rem:algorithmtermination}
If the string-syzygy quiver is finite and acyclic,
Algorithm~\ref{alg:breadthfirst} computes its height, terminates,
and outputs the exact value of the global dimension. If a directed
cycle is reachable from a simple state, the algorithm outputs
$\gldim\Lambda = \infty$ and terminates. If the quiver is infinite
and acyclic, then by Theorem~\ref{thm:gldimspecialbis} the global
dimension is infinite, but the algorithm does not terminate. The
algorithm therefore terminates if and only if the global dimension
is finite or a directed cycle is detected.
\end{remark}

\subsection{Symmetric special biserial algebras and Brauer graph algebras}\label{subsec:symmetric}

We record a consequence of the results above for symmetric
algebras. Recall that a finite dimensional $\mathbf{k}$-algebra
$\Lambda$ is symmetric if there is an isomorphism
$\Lambda \cong \Lambda^\ast$ of $\Lambda$-$\Lambda$-bimodules,
where $\Lambda^\ast = \Hom_{\mathbf{k}}(\Lambda, \mathbf{k})$ is the
$\mathbf{k}$-dual. A symmetric algebra is self-injective.

\begin{corollary}\label{cor:symmetricinfinite}
Let $\Lambda$ be a finite dimensional symmetric special biserial
algebra. If $\Lambda$ is not semisimple, then
$\gldim\Lambda = \infty$. In particular, every nonsemisimple
Brauer graph algebra has infinite global dimension.
\end{corollary}

\begin{proof}
We first show that for a symmetric algebra $\Lambda$, every right
$\Lambda$-module of finite projective dimension is projective. Let
$M$ be a right $\Lambda$-module of finite projective dimension $n$.
If $n = 0$, then $M$ is projective. Assume for contradiction that
$n \ge 1$, and consider a minimal projective resolution
\[
0 \to P_n \xrightarrow{d_n} P_{n-1} \to \cdots \to P_1 \to P_0
\to M \to 0.
\]
Let $\Omega_\Lambda^{n-1} M := \operatorname{coker} d_n$, which
coincides with $M$ when $n = 1$. Since $\pd_\Lambda M = n$, the
module $\Omega_\Lambda^{n-1} M$ is not projective, while its first
syzygy $\Omega_\Lambda^n M \cong P_n$ is projective.

Because $\Lambda$ is symmetric, it is self-injective, so the
projective module $P_n$ is also injective. Consequently, the short
exact sequence
\[
0 \to P_n \xrightarrow{d_n} P_{n-1} \to \Omega_\Lambda^{n-1} M
\to 0
\]
splits, implying that $\Omega_\Lambda^{n-1} M$ is isomorphic to a
direct summand of $P_{n-1}$ and hence is projective. This
contradicts the assumption that $\pd_\Lambda M = n$. Therefore
$n = 0$, so $M$ is projective.

It follows that $\findim\Lambda = 0$. If $\Lambda$ is not
semisimple, then there exists a nonprojective simple right
$\Lambda$-module $S$; by the argument above, $\pd_\Lambda S$ cannot
be finite, so $\pd_\Lambda S = \infty$, and
$\gldim\Lambda = \infty$.

The assertion for Brauer graph algebras follows because a Brauer
graph algebra is a symmetric special biserial algebra; see
\cite[Theorem~1.1]{Sch15}.
\end{proof}

\begin{remark}\label{rem:brauersymmetric}
Corollary~\ref{cor:symmetricinfinite} shows that no nontrivial
Brauer graph algebra has finite global dimension. This is
consistent with the fact that Brauer graph algebras are
self-injective, and self-injective algebras of finite global
dimension are semisimple. The corollary also applies to the trivial
extensions of gentle algebras studied in \cite{Sch15}.
\end{remark}

\section{The relation-chain quiver of a string algebra}
\label{sec:stringalgebra}

In Section~\ref{sec:stringsyzygy} we showed that for a special
biserial algebra the reachable syzygy quiver can be replaced by the
string-syzygy quiver, whose vertices are string modules. For a string
algebra the compression goes further: the state space reduces to a
quiver whose vertices are nonzero paths. We recover the
relation-chain formula of Xin and Zhang for the global dimension,
and we obtain a corresponding formula for the finitistic dimension.
The compression uses the monomiality of the ideal of relations. When
all relations are paths, the kernel of the map
$P_{t(p)} \to P_{s(p)}$ sending $x$ to $px$ is generated by minimal
right annihilators of $p$, and each such annihilator is again a path.
Once binomial relations are allowed, the kernel may be generated by a
linear combination of paths, and the path-state compression fails; we
illustrate this in Section~\ref{sec:binomial}.

Throughout this section $\Lambda = \mathbf{k}Q/I$ is a finite
dimensional string algebra, so that $I$ is generated by paths. We
follow the conventions of Sections~\ref{sec:background}
and~\ref{sec:universal}.

\subsection{Minimal path annihilators}\label{subsec:annihilators}

Let $\operatorname{Pa}(\Lambda)$ denote the set of residue classes of
paths in $\mathbf{k}Q$ that do not belong to $I$, including the
trivial paths $e_v$ for $v \in Q_0$. Let
$\operatorname{Pa}_{>0}(\Lambda)$ denote the subset of nontrivial
paths. Since $I$ is monomial, the nontrivial elements of
$\operatorname{Pa}(\Lambda)$ together with the trivial paths form a
$\mathbf{k}$-basis of $\Lambda$.

For $p \in \operatorname{Pa}_{>0}(\Lambda)$, let $\mathcal{M}(p)$
denote the set of minimal right annihilators of $p$, as in the
following definition.

\begin{definition}\label{def:minimalannihilator}
Let $p, u \in \operatorname{Pa}_{>0}(\Lambda)$ with $t(p) = s(u)$.
We say that $u$ is a minimal right annihilator of $p$ if
\[
pu = 0 \text{ in } \Lambda,
\]
and no proper nontrivial initial subpath $u'$ of $u$ satisfies
$pu' = 0$.
\end{definition}

\begin{remark}\label{rem:annihilatorfinite}
Since $\Lambda$ is finite dimensional and $I$ is monomial, the set of
paths not in $I$ is finite. Hence the set of minimal right
annihilators of a given path $p$ is finite.
\end{remark}

\begin{definition}\label{def:relationchain}
A minimal relation chain is a finite or infinite sequence
\[
W = (w_1, w_2, w_3, \ldots)
\]
of nontrivial paths in $\operatorname{Pa}_{>0}(\Lambda)$ such that
$w_{j+1} \in \mathcal{M}(w_j)$ for every $j$. The length of $W$ is
the number of entries; for a finite chain of length $m$ we write
$|W| = m$.
\end{definition}

We will primarily be interested in chains whose initial term is an
arrow.

\begin{definition}\label{def:relationchainquiver}
The relation-chain quiver of $\Lambda$, denoted
$\mathcal{R}(\Lambda)$, is the directed multigraph whose vertices are
the nonzero paths reached from arrows by repeatedly taking minimal
right annihilators, and whose arrows are
\[
p \longrightarrow u
\]
whenever $u$ is a minimal right annihilator of $p$.
\end{definition}

\begin{remark}\label{rem:relationchainfinite}
Since $\Lambda$ is finite dimensional, only finitely many paths
survive modulo $I$, so $\mathcal{R}(\Lambda)$ is finite.
\end{remark}

The following lemma will be used in the proof of
Proposition~\ref{prop:relationchainsyzygy}.

\begin{lemma}\label{lem:atmosttwo}
For every $p \in \operatorname{Pa}_{>0}(\Lambda)$, the set
$\mathcal{M}(p)$ has at most two elements.
\end{lemma}

\begin{proof}
Let $u, u'$ be two distinct minimal right annihilators of $p$
beginning with the same arrow $\alpha$. Let $w$ be the longest common
initial subpath of $u$ and $u'$. Since $u \ne u'$ and both begin with $\alpha$, we have $|w| \ge 1$. 
Neither $u$ nor $u'$ can be a proper initial subpath of the other, 
by the minimality condition in Definition~\ref{def:minimalannihilator}; 
hence $w$ is a proper initial subpath of both, and the decompositions 
below are valid. Write
\[
u = w\beta u_1, \qquad u' = w\gamma u_2,
\]
where $\beta, \gamma$ are distinct arrows starting at $t(w)$, and let
$\eta$ be the last arrow of $w$. We claim that $\eta\beta \notin I$
and $\eta\gamma \notin I$, which would violate condition
\textup{(SB2)} of Definition~\ref{def:specialbis} at $\eta$ and
contradict the hypothesis that $\Lambda$ is special biserial.

\emph{The claim $\eta\beta \notin I$.} Suppose for contradiction that
$\eta\beta \in I$. We distinguish two cases.

\emph{Case \textup{(i)}: $u_1$ non-trivial.} Then $w\beta$ is a
proper initial subpath of $u$, so by the minimality of $u$ we have
$pw\beta \ne 0$. Since $I$ is a two-sided ideal in $\mathbf{k}Q$ and
$\eta\beta \in I$, the path $w\beta = w'\eta\beta$ lies in $I$,
where $w = w'\eta$. Hence $pw\beta \in I$, contradicting
$pw\beta \ne 0$.

\emph{Case \textup{(ii)}: $u_1$ trivial, so $u = w\beta$.} Since
$u \in \operatorname{Pa}_{>0}(\Lambda)$, the path $u = w\beta$ does
not lie in $I$. The path $\eta\beta$ is a subpath of $w\beta = u$.
If $\eta\beta \in I$, then $\eta\beta$ contains a generator of the
monomial ideal $I$ as a subpath. The only subpaths of the two-arrow
path $\eta\beta$ are $\eta$, $\beta$, and $\eta\beta$ itself. Since
$\eta$ and $\beta$ are arrows, they do not lie in $I$ (the ideal $I$
is admissible, so $I \subseteq (\mathbf{k}Q^+)^2$). Hence
$\eta\beta$ itself must be a generator of $I$, forcing
$u = w\beta \in I$, contradicting
$u \in \operatorname{Pa}_{>0}(\Lambda)$.

Both cases give a contradiction, so $\eta\beta \notin I$. The same
argument applied to $u'$ gives $\eta\gamma \notin I$.

We have shown that $\eta$ admits two distinct continuations
$\beta, \gamma$ with both $\eta\beta, \eta\gamma \notin I$. This
contradicts condition \textup{(SB2)}. Therefore distinct minimal
right annihilators of $p$ begin with distinct arrows. By condition
\textup{(SB1)}, at most two arrows leave $t(p)$, and hence
$|\mathcal{M}(p)| \le 2$.
\end{proof}

\subsection{The kernel of a principal path map}
\label{subsec:kernel}

Let $p \in \operatorname{Pa}_{>0}(\Lambda)$. Consider the right
$\Lambda$-module map
\[
\mu_p : P_{t(p)} \longrightarrow P_{s(p)}, \qquad x \longmapsto px.
\]
The image of $\mu_p$ is the principal right ideal $p\Lambda$, which
we regard as a right $\Lambda$-module generated by $p$.

\begin{lemma}\label{lem:kernel}
For $p \in \operatorname{Pa}_{>0}(\Lambda)$ one has
\[
\ker \mu_p \cong \bigoplus_{u \in \mathcal{M}(p)} u\Lambda,
\]
where $\mathcal{M}(p)$ is the set of minimal right annihilators of
$p$.
\end{lemma}

\begin{proof}
The image of $\mu_p$ is $p\Lambda$ by definition. Since $I$ is
monomial, the nontrivial nonzero paths together with the trivial
paths form a $\mathbf{k}$-basis of $\Lambda$. An element
$x \in P_{t(p)} = e_{t(p)}\Lambda$ belongs to $\ker \mu_p$ exactly
when every path basis term occurring in $x$ is killed by left
multiplication by $p$.

Let $v$ be a nonzero path starting at $t(p)$ with $pv = 0$. Choose
the shortest nontrivial initial subpath $u$ of $v$ such that
$pu = 0$. Then $u$ is a minimal right annihilator of $p$ by
Definition~\ref{def:minimalannihilator}, and $v = ur$ for some path
$r$ in $\Lambda$, so $v \in u\Lambda$. This shows that the kernel is
the sum of the right ideals $u\Lambda$ for $u \in \mathcal{M}(p)$.

We show that this sum is direct. By the proof of
Lemma~\ref{lem:atmosttwo}, distinct minimal right annihilators of $p$
begin with distinct arrows at $t(p)$. The radical of $P_{t(p)}$
decomposes as the direct sum $\bigoplus_\alpha \alpha\Lambda$ over
the arrows $\alpha$ starting at $t(p)$. Since the paths in
$u\Lambda$ and $u'\Lambda$ for distinct minimal right annihilators
belong to disjoint summands of this direct sum, their intersection is
zero. Therefore their sum is direct:
\[
\ker \mu_p = \bigoplus_{u \in \mathcal{M}(p)} u\Lambda.
\qedhere
\]
\end{proof}

\begin{remark}\label{rem:projectivecoverofpA}
Lemma~\ref{lem:kernel} shows in particular that the map
$\mu_p : P_{t(p)} \to p\Lambda$ is a projective cover of $p\Lambda$.
Indeed, the map is surjective by definition of $p\Lambda$, and its
kernel lies in $\rad P_{t(p)}$ because every element of the kernel
has all path components of positive length.
\end{remark}

\subsection{Relation chains as the syzygy unfolding}
\label{subsec:syzygyunfolding}

We identify the successive syzygies of a simple module with the
successive vertices of the relation-chain quiver.

\begin{proposition}\label{prop:relationchainsyzygy}
For every vertex $v \in Q_0$ one has
\[
\Omega_\Lambda S_v = \bigoplus_{\substack{\alpha \in Q_1 \\
s(\alpha) = v}} \alpha\Lambda.
\]
For every nonzero path $p \in \operatorname{Pa}_{>0}(\Lambda)$ one
has
\[
\Omega_\Lambda(p\Lambda) \cong
\bigoplus_{u \in \mathcal{M}(p)} u\Lambda.
\]
For a vertex $u$ of $\mathcal{R}(\Lambda)$, the principal right ideal 
$u\Lambda$ is projective if and only if $u$ is a terminal vertex of the 
relation-chain quiver $\mathcal{R}(\Lambda)$. Consequently, the rooted unfolding of $\mathcal{R}(\Lambda)$ from
the arrows leaving $v$ records all principal-ideal summands in the
minimal projective resolution of $S_v$, with projective summands
occurring as terminal leaves.
\end{proposition}

\begin{proof}
We prove the three assertions in turn.

\emph{First assertion.} The radical of $P_v$ has the monomial path
basis consisting of the paths of positive length starting at $v$,
partitioned by the first arrow of the path. Hence
\[
\rad P_v = \bigoplus_{\substack{\alpha \in Q_1 \\
s(\alpha) = v}} \alpha\Lambda,
\]
and $\Omega_\Lambda S_v = \rad P_v$ because $P_v \to S_v$ is a
projective cover with kernel $\rad P_v$.

\emph{Second assertion.} The map
$\mu_p : P_{t(p)} \to p\Lambda$ is a projective cover of $p\Lambda$
by Remark~\ref{rem:projectivecoverofpA}, and its kernel is
$\bigoplus_{u \in \mathcal{M}(p)} u\Lambda$ by
Lemma~\ref{lem:kernel}. Hence
\[
\Omega_\Lambda(p\Lambda) \cong \ker \mu_p
\cong \bigoplus_{u \in \mathcal{M}(p)} u\Lambda.
\]

\emph{Third assertion.} A principal right ideal $u\Lambda$ for a vertex $u \in \operatorname{Pa}_{>0}(\Lambda)$ is projective if and only if
$u\Lambda \cong P_w$ for some $w \in Q_0$, equivalently if and only
if the projective cover $\mu_u : P_{t(u)} \to u\Lambda$ has zero
kernel. By Lemma~\ref{lem:kernel} this happens if and only if
$\mathcal{M}(u) = \emptyset$, that is, if and only if $u$ has no
minimal right annihilator. By
Definition~\ref{def:relationchainquiver}, this is exactly the
condition that $u$ is a terminal vertex of $\mathcal{R}(\Lambda)$.

\emph{Fourth assertion.} The final statement follows by iterating
the second assertion, using Lemma~\ref{lem:additivity} and the first
assertion to initiate the iteration.
\end{proof}

\begin{remark}\label{rem:relationchainterminal}
The third assertion of Proposition~\ref{prop:relationchainsyzygy}
makes precise the sense in which the relation-chain quiver unfolds
the minimal projective resolution of $S_v$. Starting from the
arrows leaving $v$, one follows the arrows of $\mathcal{R}(\Lambda)$;
the resolution terminates at the point where all paths in the
current layer are terminal. A chain of length $m$ in
$\mathcal{R}(\Lambda)$ beginning with an arrow corresponds to a
summand that survives through $m$ layers of the resolution and
contributes $m$ to the projective dimension of the corresponding
simple module.
\end{remark}

\subsection{The Xin--Zhang relation-chain formula} \label{subsec:zhengzhang}

We state and prove the relation-chain formula for the global
dimension of a string algebra.

\begin{theorem}[Xin--Zhang \cite{ZZ26}]\label{thm:zhengzhang}
Let $\Lambda = \mathbf{k}Q/I$ be a finite dimensional string algebra.
Then
\[
\gldim\Lambda = \sup\bigl(\{0\} \cup
\{|W| : W \text{ a minimal relation chain beginning with an arrow}\}\bigr),
\]
where $|W| = \infty$ for an infinite chain. Equivalently,
\[
\gldim\Lambda = \sup\bigl(\{0\} \cup
\{1 + \operatorname{ht}_{\mathcal{R}(\Lambda)}(\alpha) :
\alpha \in Q_1\}\bigr).
\]
Furthermore $\gldim\Lambda = \infty$ if and only if there is an
infinite minimal relation chain beginning with an arrow. Since
$\mathcal{R}(\Lambda)$ is finite, this is equivalent to a directed
cycle of $\mathcal{R}(\Lambda)$ being reachable from an arrow.
\end{theorem}

The statement is due to \cite{ZZ26}; the proof we give below is
different, and proceeds by unfolding the relation-chain quiver.

\begin{proof}
By Lemma~\ref{lem:simplegldim} it suffices to compute the projective
dimension of each simple right $\Lambda$-module.

Fix $v \in Q_0$. By Proposition~\ref{prop:relationchainsyzygy}, the
minimal projective resolution of $S_v$ has successive nonprojective
layers
\[
\Omega_\Lambda^1 S_v = \bigoplus_{s(\alpha) = v} \alpha\Lambda,
\qquad
\Omega_\Lambda^2 S_v = \bigoplus_{s(\alpha) = v}
\bigoplus_{u \in \mathcal{M}(\alpha)} u\Lambda,
\qquad \dots
\]
Each summand corresponds to a minimal relation chain beginning with
an arrow leaving $v$: a summand of the form $u\Lambda$ appearing at
level $n$ corresponds to a chain $(w_1, \ldots, w_n)$ where $w_1$ is
an arrow leaving $v$ and $w_{j+1} \in \mathcal{M}(w_j)$ for each $j$.
The chain has length $n$; the corresponding summand survives through
$n$ layers of the resolution. By the third assertion of
Proposition~\ref{prop:relationchainsyzygy}, the resolution terminates
at the first level at which all summands are projective, that is, at
the first level at which all chains reach a terminal vertex.

Therefore
\[
\pd_\Lambda S_v = \sup\{|W| : W \text{ a minimal relation chain
beginning with an arrow leaving } v\},
\]
with the convention that the supremum over the empty set equals $0$
if $v$ is not the source of any arrow, in which case $S_v = P_v$ is
projective.

Taking the supremum over all $v \in Q_0$ gives the first displayed
formula.

For the second formula, note that each arrow $\alpha$ of $Q$ is a
vertex of $\mathcal{R}(\Lambda)$, and the minimal relation chains
beginning with $\alpha$ are exactly the directed paths in
$\mathcal{R}(\Lambda)$ beginning at $\alpha$. A chain of length $m$
corresponds to a directed path of length $m - 1$ beginning at
$\alpha$, so $m = 1 + (\text{length of the path})$. Taking the
supremum over all $\alpha$ gives the second displayed formula.

Finally, if $\mathcal{R}(\Lambda)$ contains a directed cycle
reachable from an arrow $\alpha$, then traversing the cycle
repeatedly produces directed paths of arbitrary length starting at
$\alpha$, hence minimal relation chains of arbitrary length, hence
$\pd_\Lambda S_{s(\alpha)} = \infty$ and $\gldim\Lambda = \infty$.
Conversely, if $\gldim\Lambda = \infty$, then some simple module has
infinite projective dimension, so there are minimal relation chains
of arbitrary length beginning with an arrow. Since
$\mathcal{R}(\Lambda)$ is finite, this forces a directed cycle
reachable from an arrow by the pigeonhole principle.
\end{proof}

\begin{remark}\label{rem:zhengzhangcomparison}
Theorem~\ref{thm:zhengzhang} is the specialisation of
Theorem~\ref{thm:gldimspecialbis} to the monomial case. In the
string-syzygy quiver, states are string modules; in the
relation-chain quiver, states are individual paths. The two quivers
are related by the compression map of
Section~\ref{subsec:compression}. This compression is available
precisely because, in the monomial case, the kernel of a principal
path map is a direct sum of principal path ideals.
\end{remark}

\subsection{Finitistic dimension for string algebras}
\label{subsec:findimstring}

The relation-chain quiver also yields a combinatorial lower bound
for the finitistic dimension of a string algebra, which becomes an
exact equality whenever the global dimension is finite.

\begin{lemma}\label{lem:pathidealsyzygyheight}
Let $\Lambda = \mathbf{k}Q/I$ be a finite dimensional string algebra
and let $p \in \operatorname{Pa}_{>0}(\Lambda)$ be a vertex of
$\mathcal{R}(\Lambda)$. Then
\[
\pd_\Lambda(p\Lambda) = \operatorname{ht}_{\mathcal{R}(\Lambda)}(p),
\qquad
\pd_\Lambda(P_{s(p)}/p\Lambda) = 1 +
\operatorname{ht}_{\mathcal{R}(\Lambda)}(p).
\]
In particular, $P_{s(p)}/p\Lambda$ has finite projective dimension if
and only if $\operatorname{ht}_{\mathcal{R}(\Lambda)}(p) < \infty$.
\end{lemma}

\begin{proof}
By Proposition~\ref{prop:relationchainsyzygy},
$\Omega_\Lambda(p\Lambda) \cong \bigoplus_{u \in \mathcal{M}(p)}
u\Lambda$, and a summand $u\Lambda$ is projective if and only if
$u$ is a terminal vertex of $\mathcal{R}(\Lambda)$ (that is,
$\mathcal{M}(u) = \emptyset$). By induction on $k \ge 0$, the $k$-th
syzygy decomposes over all directed paths of length $k$ in
$\mathcal{R}(\Lambda)$ starting at $p$:
\[
\Omega_\Lambda^k(p\Lambda) \cong
\bigoplus_{p = u_0 \to u_1 \to \cdots \to u_k} u_k\Lambda.
\]
Consequently, $\Omega_\Lambda^k(p\Lambda)$ is projective if and only
if every directed path of length $k$ starting at $p$ ends at a
terminal vertex, which holds if and only if there is no directed
path of length $k + 1$ in $\mathcal{R}(\Lambda)$ starting at $p$
(equivalently, $\operatorname{ht}_{\mathcal{R}(\Lambda)}(p) \le k$).
The smallest nonnegative integer $k$ with this property is
$\operatorname{ht}_{\mathcal{R}(\Lambda)}(p)$, and thus
$\pd_\Lambda(p\Lambda) = \operatorname{ht}_{\mathcal{R}(\Lambda)}(p)$.

Since $p$ has positive length, $p\Lambda \subseteq \rad P_{s(p)}$.
The canonical projection $P_{s(p)} \to P_{s(p)}/p\Lambda$ is
therefore a projective cover with kernel
$\Omega_\Lambda(P_{s(p)}/p\Lambda) = p\Lambda$. Consequently,
\[
\pd_\Lambda(P_{s(p)}/p\Lambda) = 1 + \pd_\Lambda(p\Lambda)
= 1 + \operatorname{ht}_{\mathcal{R}(\Lambda)}(p). \qedhere
\]
\end{proof}

\begin{theorem}\label{thm:findimstring}
Let $\Lambda = \mathbf{k}Q/I$ be a finite dimensional string
algebra. Then
\[
\findim\Lambda \ge \sup\bigl(\{0\} \cup
\{1 + \operatorname{ht}_{\mathcal{R}(\Lambda)}(p) :
p \in \mathcal{R}(\Lambda) \text{ with }
\operatorname{ht}_{\mathcal{R}(\Lambda)}(p) < \infty\}\bigr).
\]
If $\gldim\Lambda < \infty$, then equality holds.
\end{theorem}

\begin{proof}
For every vertex $p \in \mathcal{R}(\Lambda)$ with
$\operatorname{ht}_{\mathcal{R}(\Lambda)}(p) < \infty$,
Lemma~\ref{lem:pathidealsyzygyheight} shows that the module
$M_p := P_{s(p)}/p\Lambda$ has finite projective dimension equal to
$1 + \operatorname{ht}_{\mathcal{R}(\Lambda)}(p)$. Taking the
supremum over all such $p$ yields the lower bound for
$\findim\Lambda$.

If $\gldim\Lambda < \infty$, then every vertex of
$\mathcal{R}(\Lambda)$ has finite height, and
$\findim\Lambda = \gldim\Lambda$. Since every arrow $\alpha \in Q_1$
is a vertex of $\mathcal{R}(\Lambda)$, the right-hand side is at
least $\sup\bigl(\{0\} \cup
\{1 + \operatorname{ht}_{\mathcal{R}(\Lambda)}(\alpha) :
\alpha \in Q_1\}\bigr) = \gldim\Lambda$ by
Theorem~\ref{thm:zhengzhang}, while no module can have projective
dimension exceeding $\gldim\Lambda$. Thus equality holds.
\end{proof}

\begin{remark}\label{rem:findimstringnaive}
The formula in Theorem~\ref{thm:findimstring} is stated in terms of
the height condition $\operatorname{ht}_{\mathcal{R}(\Lambda)}(p) <
\infty$ rather than the length of an individual finite minimal
relation chain starting at $p$. This distinction is essential for
two reasons.

First, a finite chain $(w_1, \ldots, w_m)$ that is not maximal can
be extended further. The algebra
$B = \mathbf{k}Q/(\alpha\beta, \beta\alpha)$ of
Theorem~\ref{thm:twostring} illustrates this: $\mathcal{R}(B)$ is a
directed $2$-cycle, so no vertex of $\mathcal{R}(B)$ has finite
height. By the Butler--Ringel classification
(Theorem~\ref{thm:butlerringel}) and the fact that $B$ has no bands,
the indecomposable right $B$-modules are $S_1$, $S_2$, $P_1$, $P_2$.
Since $\pd_B S_1 = \pd_B S_2 = \infty$ by
Theorem~\ref{thm:twostring} and $\pd_B P_1 = \pd_B P_2 = 0$, every
nonprojective $B$-module has infinite projective dimension, and
hence $\findim B = 0$.

Second, by Lemma~\ref{lem:atmosttwo}, a path $p$ may have two
minimal right annihilators $q$ and $r$; if $q$ is terminal while $r$
lies on a directed cycle, then $(p, q)$ is a maximal chain ending at
a terminal vertex, yet $\Omega_\Lambda(p\Lambda) \cong q\Lambda
\oplus r\Lambda$ has infinite projective dimension because
$\operatorname{ht}_{\mathcal{R}(\Lambda)}(r) = \infty$. Requiring
$\operatorname{ht}_{\mathcal{R}(\Lambda)}(p) < \infty$ ensures that
every branch of the syzygy resolution of $p\Lambda$ terminates.
\end{remark}

\subsection{The compression phenomenon}\label{subsec:compression}

We close this section with a description of the compression that
restricts the string-syzygy quiver of
Section~\ref{sec:stringsyzygy} to the relation-chain quiver of this
section, and of the algebraic limits of this compression.

\begin{remark}\label{rem:compressionmap}
Let $\Lambda$ be a string algebra. For a nonzero path $p$, the principal
right ideal $p\Lambda$ is a string module with a one-dimensional top
generated by $p$. The nonprojective principal path ideals that occur as
direct summands of iterated syzygies of simple modules are in bijection
with the nonterminal vertices of the relation-chain quiver
$\mathcal{R}(\Lambda)$. Under this bijection, the transition
\[
\Omega_\Lambda(p\Lambda) \cong \bigoplus_{u \in \mathcal{M}(p)} u\Lambda
\]
of Proposition~\ref{prop:relationchainsyzygy} corresponds to the arrows
$p \to u$ of $\mathcal{R}(\Lambda)$. A terminal vertex $u$ corresponds
to a projective summand $u\Lambda$, which is discarded in the
string-syzygy quiver but retained in $\mathcal{R}(\Lambda)$ to record
the termination of the resolution. The relation-chain quiver is
therefore obtained from the string-syzygy quiver by restricting to the
principal path ideals that occur in the resolution, replacing each such
ideal $p\Lambda$ by the path $p$, and adjoining the terminal projective
principal path ideals as sinks. This combinatorial model computes the
global dimension by Theorem~\ref{thm:zhengzhang}. It is not a direct
quotient of the full string-syzygy quiver, since not every string module
is a principal path ideal (see Remark~\ref{rem:compressionnotall}).
\end{remark}

\begin{remark}\label{rem:compressionnotall}
Not every string module over a string algebra is a principal path
ideal. Consider the string algebra with quiver
\[
1 \xrightarrow{a} 2 \xrightarrow{b} 3, \qquad 1 \xrightarrow{c} 3,
\]
and relation $ab = 0$. The string module $M(a)$ has basis
$\{z_1, z_2\}$ with $z_1 \cdot a = z_2$, so it has dimension $2$.
The principal path ideals of $\Lambda$ are
$P_1 = \langle e_1, a, c\rangle$ of dimension $3$,
$P_2 = \langle e_2, b\rangle$ of dimension $2$,
$P_3 = \langle e_3\rangle$ of dimension $1$,
$a\Lambda = \langle a\rangle$ of dimension $1$, and
$b\Lambda, c\Lambda$ of dimension $1$. None of these equals $M(a)$.
Nevertheless, the projective dimension of $M(a)$ is determined by
the relation-chain quiver. The projective cover of $M(a)$ is $P_1$,
and its kernel is $c\Lambda$. Since $c$ is terminal (its target is
a sink), the chain is $(c)$ of length $1$ and $\pd M(a) = 1$. In this example, the relation-chain quiver computes the projective dimension of the string module $M(a)$ via its first syzygy, even 
though $M(a)$ itself is not a principal path ideal.
\end{remark}

\begin{remark}\label{rem:compressionalgorithm}
The relation-chain quiver can be computed from the bound quiver
presentation in finitely many steps. Starting from the arrows of
$Q$, one computes the minimal right annihilators of each path by
inspecting the relations in $I$, and one builds the quiver
$\mathcal{R}(\Lambda)$ by adding an arrow $p \to u$ for each such
annihilator. The global dimension and the finitistic dimension are
then obtained by the formulas of
Theorems~\ref{thm:zhengzhang} and~\ref{thm:findimstring}. We refer
to Section~\ref{sec:examples} for concrete computations.
\end{remark}

\begin{remark}\label{rem:compressionfails}
The compression of Remark~\ref{rem:compressionmap} is available in
the monomial case. If the ideal of relations is not generated by
paths, then the kernel of a principal path map may be generated by
a linear combination of paths rather than by a single path, and the
path state alone no longer determines the syzygy. We illustrate
this in Section~\ref{sec:binomial}.
\end{remark}

\section{Binomial relations and string states}\label{sec:binomial}

In Section~\ref{sec:stringalgebra} we showed that for a string
algebra the reachable syzygy quiver compresses to the relation-chain
quiver, whose vertices are nonzero paths. The compression uses the
monomiality of the ideal of relations: when all relations are paths,
the kernel of the map $P_{t(p)} \to P_{s(p)}$ sending $x$ to $px$ is
generated by minimal right annihilators of $p$, and each such
annihilator is again a path. In this section we show that the
compression fails once binomial relations are present, and we
identify the exact point at which it fails.

In Section~\ref{subsec:commutativediamond} we introduce the
commutativity diamond, the smallest special biserial algebra with a
binomial relation, and compute its global dimension. In
Section~\ref{subsec:naiveextension} we show that a naive extension of
the path-chain rule to the commutativity diamond gives an incorrect
answer. In Section~\ref{subsec:stringquotient} we introduce the
associated string quotient and show that it does not preserve the
global dimension in general. In
Section~\ref{subsec:monomialreductionfails} we prove a negative
result: no admissible monomial ideal on the same quiver that contains
every monomial occurring in a relation of the original algebra can
achieve pointwise matching of projective dimensions.

Throughout this section $\Lambda = \mathbf{k}Q/I$ is a special
biserial algebra. We follow the conventions of
Sections~\ref{sec:background} and~\ref{sec:stringsyzygy}.

\subsection{The commutativity diamond}\label{subsec:commutativediamond}

We begin with the standard example that exhibits the difference
between monomial and binomial relations.

\begin{example}[Commutativity diamond]\label{ex:commutativediamond}
Let $Q$ be the quiver with vertices $1, 2, 3, 4$ and arrows
\[
a : 1 \to 2, \quad b : 1 \to 3, \quad c : 2 \to 4, \quad
d : 3 \to 4,
\]
as in Figure~\ref{fig:diamond}, and let
\[
A := \mathbf{k}Q/(ac - bd).
\]
Then $A$ is a finite dimensional special biserial algebra that is
not a string algebra, and $\gldim A = 2$.
\end{example}

\vspace{-1.5em}
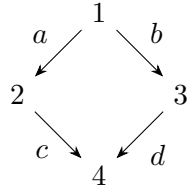
\begin{figure}[h]
\centering
\begin{tikzpicture}[scale=0.9,>=Stealth]
\node (1) at (0,0) {$1$};
\node (2) at (-1.2,-1.2) {$2$};
\node (3) at (1.2,-1.2) {$3$};
\node (4) at (0,-2.4) {$4$};
\draw[->] (1) -- node[above left] {$a$} (2);
\draw[->] (1) -- node[above right] {$b$} (3);
\draw[->] (2) -- node[below left] {$c$} (4);
\draw[->] (3) -- node[below right] {$d$} (4);
\end{tikzpicture}
\caption{The commutativity diamond.}
\label{fig:diamond}
\end{figure}

\vspace{-0.8em}
\begin{proof}
We verify that $A$ is special biserial. For \textup{(SB1)}, the
out-degrees of $Q$ are $(2,1,1,0)$ and the in-degrees are
$(0,1,1,2)$, so both are at most two at every vertex. For
\textup{(SB2)}, vertex $2$ has the unique incoming arrow $a$ and the
unique outgoing arrow $c$, vertex $3$ has the unique incoming arrow
$b$ and the unique outgoing arrow $d$, and vertices $1$ and $4$ are
respectively a source and a sink. Hence for every arrow the possible
continuations, and the possible predecessors, form a set of at most
one element, so \textup{(SB2)} holds regardless of $I$. Therefore $A$
is special biserial.

The ideal $I = (ac - bd)$ is generated by $ac - bd$. We claim that
neither $ac$ nor $bd$ lies in $I$. The ideal $I$ consists of the
elements $u(ac - bd)v$ where $u, v$ are paths in $Q$. The only path
of $Q$ ending at $1$ is $e_1$, and the only path of $Q$ starting at
$4$ is $e_4$. Since $ac - bd$ goes from $1$ to $4$, we have
$u = e_1$ and $v = e_4$, so $I = \mathbf{k}\cdot(ac - bd)$. If
$ac \in I$, then $ac = \lambda(ac - bd)$ for some
$\lambda \in \mathbf{k}$, forcing $\lambda = 1$ and $\lambda = 0$,
a contradiction. Hence $ac \notin I$, and similarly $bd \notin I$.
Therefore $A$ is not a string algebra.

We compute the global dimension. Write $z := ac = bd$ in $A$. The
$\mathbf{k}$-basis of $A$ is the set of residue classes of the paths
of $Q$ not in $I$, namely
\[
e_1, \quad e_2, \quad e_3, \quad e_4, \quad a, \quad b, \quad c,
\quad d, \quad z.
\]
Indeed, $Q$ has no paths of length $\ge 3$, the paths of length two
are $ac$ and $bd$, and these are identified to a single element $z$.

The indecomposable projective right $A$-modules have the
$\mathbf{k}$-bases
\[
P_1 = \langle e_1, a, b, z\rangle, \quad
P_2 = \langle e_2, c\rangle, \quad
P_3 = \langle e_3, d\rangle, \quad
P_4 = \langle e_4\rangle.
\]
The right $A$-module structure is determined by
$a\cdot e_2 = a$, $a\cdot c = z$, $b\cdot e_3 = b$, $b\cdot d = z$,
$c\cdot e_4 = c$, $d\cdot e_4 = d$, $z\cdot e_4 = z$, and all other
concatenations vanish.

The projective cover of $S_1$ is the natural surjection
$\varepsilon_0 : P_1 \to S_1$ with kernel
$\rad P_1 = \langle a, b, z\rangle$. Consider the right $A$-module
homomorphism
\[
\psi : P_2 \oplus P_3 \longrightarrow \rad P_1,
\qquad \psi(x, y) = ax + by.
\]
The map $\psi$ is surjective: $\psi(e_2, 0) = a$, $\psi(0, e_3) = b$,
and $\psi(c, 0) = ac = z$.

We compute the kernel. Let $x = \lambda e_2 + \mu c \in P_2$ and
$y = \nu e_3 + \rho d \in P_3$ with $\lambda, \mu, \nu, \rho \in
\mathbf{k}$. Then
\[
\psi(x, y) = \lambda a + \mu ac + \nu b + \rho bd
= \lambda a + \nu b + (\mu + \rho) z.
\]
This is zero in $\rad P_1$ if and only if $\lambda = \nu = 0$ and
$\mu + \rho = 0$. Hence
\[
\ker \psi
= \{(\mu c, -\mu d) : \mu \in \mathbf{k}\}
= \mathbf{k}\cdot(c, -d).
\]
As a right $A$-module, $(c, -d)\cdot e_4 = (c, -d)$ and
$(c, -d)\cdot \alpha = 0$ for every other arrow $\alpha$. Hence
$\ker \psi \cong S_4$. Since $S_4$ lies in
$\rad(P_2 \oplus P_3) = \langle c\rangle \oplus \langle d\rangle$, the
map $\psi$ is a minimal projective cover of $\rad P_1$. Therefore
\[
\Omega_A^2 S_1 = \ker \psi \cong S_4.
\]
The module $S_4$ is projective, since $P_4 = \langle e_4\rangle$ and
$\rad P_4 = 0$. Hence $\Omega_A^3 S_1 = 0$ and $\pd_A S_1 = 2$.

The remaining simple modules are treated similarly. The projective
cover of $S_2$ is $P_2 \to S_2$ with kernel
$\langle c\rangle = cA$. Since $c\cdot e_4 = c$ and
$c\cdot \alpha = 0$ for $\alpha \neq e_4$, we have
$cA \cong S_4 = P_4$, which is projective. Hence $\pd_A S_2 = 1$.
The same argument gives $\pd_A S_3 = 1$. Finally $\pd_A S_4 = 0$
since $S_4 = P_4$. By Lemma~\ref{lem:simplegldim},
\[
\gldim A = \sup\{2, 1, 1, 0\} = 2.
\qedhere
\]
\end{proof}

\begin{remark}\label{rem:stringmodulediamond}
The first syzygy $\rad P_1 = \langle a, b, z\rangle$ of
Example~\ref{ex:commutativediamond} is a string module. Consider the
quiver $W$ with vertices $u, v, w$ and arrows $u \xrightarrow{c} v$
and $w \xrightarrow{d} v$, together with the morphism
$\ell : W \to Q$ sending $u \mapsto 2$, $v \mapsto 4$, $w \mapsto 3$,
and the arrows to $c$ and $d$. The underlying unoriented graph of $W$
is the path $2 - 4 - 3$, so $W$ is a string diagram. The associated
module $M(W)$ has basis $\{z_u, z_v, z_w\}$ with actions
$z_u\cdot c = z_v$ and $z_w\cdot d = z_v$, and all other actions
zero. The linear map $M(W) \to \rad P_1$ sending $z_u \mapsto a$,
$z_v \mapsto z$, $z_w \mapsto b$ is an isomorphism of right
$A$-modules. The relation $ac = bd$ acts as zero on $M(W)$ because
no basis element is mapped to vertex $1$. Hence $\rad P_1 \cong M(W)$
is a string module.
\end{remark}

\subsection{Failure of a naive path-chain extension}
\label{subsec:naiveextension}

The relation-chain formula of Theorem~\ref{thm:zhengzhang} relies on
the monomiality of the ideal of relations. A natural question is
whether the same formula holds, with the obvious modifications, when
binomial relations are allowed. We show in this subsection that it
does not.

\begin{proposition}\label{prop:naivepathchain}
In Example~\ref{ex:commutativediamond} neither $ac$ nor $bd$ is zero
in $A$. Consequently, an algorithm that computes minimal right
annihilators of $a$ and of $b$ in $A$ finds none, and would conclude,
if it applied the path-chain rule verbatim, that the principal right
ideals $aA$ and $bA$ are projective and that $\rad P_1 = aA \oplus bA$
is projective. The actual projective dimension of $S_1$ is $2$.
\end{proposition}

\begin{proof}
By the proof of Example~\ref{ex:commutativediamond}, neither $ac$ nor
$bd$ lies in $I = (ac - bd)$, so both are nonzero in $A$. The only
path of $Q$ starting at $t(a) = 2$ is $c$, and $ac \neq 0$; hence
there is no minimal right annihilator of $a$ in $A$. Similarly the
only path of $Q$ starting at $t(b) = 3$ is $d$, and $bd \neq 0$;
hence there is no minimal right annihilator of $b$ in $A$.

An algorithm applying the path-chain rule of
Proposition~\ref{prop:relationchainsyzygy} verbatim would proceed as
follows. For $p = a$, no minimal right annihilator exists, so $aA$
is terminal and, by the third assertion of
Proposition~\ref{prop:relationchainsyzygy}, $aA$ is projective.
Similarly $bA$ is projective. The radical of $P_1$ is the sum
$aA + bA$; if this sum were direct, $\rad P_1$ would be the direct
sum of two projective modules and hence projective, giving
$\pd_A S_1 \le 1$.

We show that this conclusion is incorrect. The right $A$-module
$aA = \langle a, z\rangle$ is isomorphic to $P_2$: the map sending
$a \mapsto e_2$ and $z \mapsto c$ is a bijection of $\mathbf{k}$-bases
and preserves the right $A$-module structure. Similarly $bA \cong P_3$.
However, $aA \cap bA = \langle z\rangle \neq 0$, so
$\rad P_1 = aA + bA$ is not a direct sum. The sum of two projective
submodules of a projective module need not be projective, and here
it is not: by Example~\ref{ex:commutativediamond},
$\pd_A S_1 = 2$, so $\rad P_1$ is not projective.
\end{proof}

\begin{remark}\label{rem:linearcombination}
The kernel vector $(c, -d)$ of the projective cover
$\psi : P_2 \oplus P_3 \to \rad P_1$ is not detected by any rule that
inspects isolated paths. A change of basis of $P_2$ rescales $c$ by
a nonzero scalar, and a change of basis of $P_3$ rescales $d$ by a
nonzero scalar, so neither component can be eliminated. Since
$ac = bd$ in $A$, the element $c$ alone does not lie in the kernel
(as $ac = z \neq 0$), and neither does $d$ alone. The kernel is the
linear combination $(c, -d)$, and it is responsible for the second
syzygy of $S_1$.
\end{remark}

\subsection{The associated string quotient does not preserve homology}\label{subsec:stringquotient}

A natural attempt to reduce a special biserial algebra to a string
algebra on the same quiver is to replace the binomial relations by
the monomials appearing in them, and to work with the resulting
string algebra. We show that this reduction does not preserve the
global dimension in general.

\begin{definition}\label{def:stringquotient}
Let $A$ be a finite dimensional special biserial algebra. The
associated string quotient of $A$ is the algebra $A/J$, where $J$ is
the two-sided ideal generated by the paths that appear as terms in
the binomial relations of $A$. More precisely, if the binomial
relations of $A$ are of the form $p_r - \lambda_r q_r = 0$ with
$\lambda_r \in \mathbf{k}^\ast$ and $p_r, q_r$ paths, then $J$ is
the two-sided ideal generated by $\{p_r, q_r\}_r$.
\end{definition}

\begin{remark}\label{rem:stringquotientdiamond}
For the commutativity diamond of Example~\ref{ex:commutativediamond},
the associated string quotient is
\[
A/J = \mathbf{k}Q/(ac, bd).
\]
This is a string algebra, since the defining relations are
monomials. A direct computation, similar to that in
Example~\ref{ex:commutativediamond}, gives $\gldim A/J = 2$, which
coincides with $\gldim A$. The two homological kernels, however,
differ. Over $A$ the kernel of the minimal projective cover of
$\rad P_1$ is generated by the coupled vector $(c, -d)$ and is
isomorphic to $S_4$; over $A/J$ the kernel is generated by $(c, 0)$
and $(0, d)$ and is isomorphic to $S_4 \oplus S_4$. The resolutions
agree in length but differ in the syzygy modules.
\end{remark}

The coincidence of global dimensions in
Remark~\ref{rem:stringquotientdiamond} is not a general phenomenon,
as the following example shows.

\begin{example}\label{ex:parallelarrows}
Let $Q'$ be the quiver with vertices $1, 2, 3, 4, 5$ and arrows
\[
a : 1 \to 2, \quad b : 1 \to 3, \quad c : 2 \to 4, \quad
d : 3 \to 4, \quad x : 4 \to 5,
\]
as in Figure~\ref{fig:parallelarrows}, and let
\[
A' := \mathbf{k}Q'/(ac - bd, cx).
\]
Then $A'$ is special biserial with $\gldim A' = 2$, while the
associated string quotient
\[
A'/J = \mathbf{k}Q'/(ac, bd, cx)
\]
is a string algebra with $\gldim(A'/J) = 3$.
\end{example}

\vspace{-1.2em}
\begin{figure}[h]
\centering
\begin{tikzpicture}[scale=0.9,>=Stealth]
\node (1) at (0,0) {$1$};
\node (2) at (-1.2,-1.2) {$2$};
\node (3) at (1.2,-1.2) {$3$};
\node (4) at (0,-2.4) {$4$};
\node (5) at (0,-3.6) {$5$};
\draw[->] (1) -- node[above left] {$a$} (2);
\draw[->] (1) -- node[above right] {$b$} (3);
\draw[->] (2) -- node[below left] {$c$} (4);
\draw[->] (3) -- node[below right] {$d$} (4);
\draw[->] (4) -- node[right] {$x$} (5);
\end{tikzpicture}
\caption{The quiver of $A' = \mathbf{k}Q'/(ac-bd,\, cx)$.}
\label{fig:parallelarrows}
\end{figure}
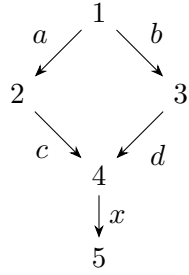

\begin{proof}
Condition \textup{(SB1)} is immediate from the shape of $Q'$, whose
out-degrees $(2,1,1,1,0)$ and in-degrees $(0,1,1,2,1)$ are at most
two at every vertex. For \textup{(SB2)}, the only point that is not
immediate from the presentation is $dx \notin I$: since $I$ is
generated by the degree-$2$ elements $ac-bd$ and $cx$, its degree-$2$
component is $\mathbf{k}(ac-bd)\oplus\mathbf{k}(cx)$, which does not
contain the basis path $dx$. Given this, every arrow has at most one
continuation and at most one predecessor outside $I$: the arrow $a$
has the unique continuation $c$ and no predecessor; $b$ has the
unique continuation $d$ and no predecessor; $c$ has the unique
predecessor $a$ and no continuation, since $cx \in I$; $d$ has the
unique predecessor $b$ and the unique continuation $x$; and $x$ has
the unique predecessor $d$ and no continuation, since $cx \in I$
while $dx \notin I$. Hence $A'$ is special biserial. It is not a
string algebra because $ac-bd$ is a binomial relation.

\emph{The global dimension of $A'$ is $2$.} Write $z := ac = bd$ in
$A'$. The basis of $A'$ is
\[
e_1, \ldots, e_5, \ a, b, c, d, x, \ z, \ dx.
\]
The indecomposable projective right $A'$-modules have
$\mathbf{k}$-bases
\[
P_1 = \langle e_1, a, b, z\rangle, \quad
P_2 = \langle e_2, c\rangle, \quad
P_3 = \langle e_3, d, dx\rangle, \quad
P_4 = \langle e_4, x\rangle, \quad
P_5 = \langle e_5\rangle.
\]
The radicals are
\[
\rad P_1 = \langle a, b, z\rangle, \quad
\rad P_2 = \langle c\rangle, \quad
\rad P_3 = \langle d, dx\rangle, \quad
\rad P_4 = \langle x\rangle, \quad
\rad P_5 = 0.
\]
The projective cover of $S_1$ is $P_1 \to S_1$ with kernel
$\rad P_1$. The projective cover of $\rad P_1$ is
$\psi : P_2 \oplus P_3 \to \rad P_1$ sending $e_2 \mapsto a$ and
$e_3 \mapsto b$. Its kernel consists of the elements
$(\mu c, -\mu d + \tau dx)$ for $\mu, \tau \in \mathbf{k}$. The map
sending $(\mu c, -\mu d + \tau dx) \mapsto \mu e_4 - \tau x$ is an
isomorphism of right $A'$-modules onto $P_4$: the nonzero right
actions are by $e_4$ (identity on both sides) and by $x$ (which sends
$(c, -d)$ to $(0, -dx)$ on the left and $e_4$ to $x$ on the right,
and these match under the identification
$(0, -dx) \leftrightarrow x$). Hence $\Omega_{A'}^2 S_1 \cong P_4$,
projective, and $\pd_{A'} S_1 = 2$.

The simple $S_2$ has $\Omega_{A'} S_2 = \langle c\rangle \cong S_4$,
and $\Omega_{A'} S_4 = \langle x\rangle = P_5$, so
$\pd_{A'} S_2 = 2$. The simple $S_3$ has
$\Omega_{A'} S_3 = \langle d, dx\rangle \cong P_4$, so
$\pd_{A'} S_3 = 1$. Finally $\pd_{A'} S_4 = 1$ and
$\pd_{A'} S_5 = 0$. Hence
\[
\gldim A' = \sup\{2, 2, 1, 1, 0\} = 2.
\]

\emph{The global dimension of $A'/J$ is $3$.} The string quotient
$A'/J = \mathbf{k}Q'/(ac, bd, cx)$ has the same projective modules
as $A'$ except that $ac$ and $bd$ vanish independently, so
$P_1 = \langle e_1, a, b\rangle$ has dimension $3$ and radical
$\langle a, b\rangle$. The projective cover
$P_2 \oplus P_3 \to \rad P_1$ sending $e_2 \mapsto a$ and
$e_3 \mapsto b$ has kernel $S_4 \oplus P_4$: the component $S_4$
arises because $cA'/J$ is annihilated by the radical, and the
component $P_4$ arises because
$\langle d, dx\rangle \cong P_4$. Hence
$\Omega_{A'/J}^2 S_1 = S_4 \oplus P_4$, and the nonprojective
summand $S_4$ has projective dimension $1$ (via
$\Omega S_4 = P_5$). Therefore $\pd_{A'/J} S_1 = 3$. The other
simple modules satisfy $\pd_{A'/J} S_2 = 2$, $\pd_{A'/J} S_3 = 1$,
$\pd_{A'/J} S_4 = 1$, $\pd_{A'/J} S_5 = 0$, and hence
\[
\gldim(A'/J) = \sup\{3, 2, 1, 1, 0\} = 3. \qedhere
\]
\end{proof}

\begin{remark}\label{rem:stringquotientanomaly}
Example~\ref{ex:parallelarrows} shows that the global dimension of a
special biserial algebra and that of its associated string quotient
can differ. The coincidence in
Remark~\ref{rem:stringquotientdiamond} is therefore not a general
phenomenon.
\end{remark}

\subsection{Failure of monomial reduction preserving the relations}\label{subsec:monomialreductionfails}

Example~\ref{ex:parallelarrows} shows that the associated string
quotient, which discards binomial relations in favour of their
terms, does not preserve the global dimension. A more refined
reduction keeps every monomial that appears in a relation of the
original algebra, and allows additional monomial relations. We show
that even this refined reduction fails to achieve pointwise matching
of projective dimensions.

\begin{proposition}\label{prop:pointwisefails}
There exists a special biserial algebra $\Lambda = \mathbf{k}Q/I$
such that no admissible monomial ideal $I'$ on $Q$ containing every
monomial occurring in a relation of $I$ yields a string algebra
$\Lambda' = \mathbf{k}Q/I'$ with
$\pd_{\Lambda'}(S_v) = \pd_\Lambda(S_v)$ for every vertex $v$ of
$Q$.
\end{proposition}

\begin{proof}
Take $\Lambda = A' = \mathbf{k}Q'/(ac - bd, cx)$, where $Q'$ is the
quiver of Figure~\ref{fig:parallelarrows}. The relations of $I$
contain the monomials $ac$, $bd$, and $cx$. By
Example~\ref{ex:parallelarrows}, $\pd_{A'}(S_1) = 2$ and
$\pd_{A'}(S_2) = 2$.

Let $I'$ be an admissible monomial ideal on $Q'$ containing
$\langle ac, bd, cx\rangle$. The length-two paths in $Q'$ that are
not in $\langle ac, bd, cx\rangle$ are $dx$; paths of length at
least three are automatic consequences of the length-two paths. The
only paths of length three in $Q'$ are $acx$ and $bdx$, and both
contain a path already in the ideal. Hence there are exactly two
admissible monomial ideals on $Q'$ containing
$\langle ac, bd, cx\rangle$, and they correspond to the two cases in
the following table.

\begin{center}
\begin{tabular}{lcc}
\hline
Generators of $I'$ & $\pd_{\Lambda'}(S_1)$ &
$\pd_{\Lambda'}(S_2)$ \\
\hline
$\langle ac, bd, cx\rangle$ & $3$ & $2$ \\
$\langle ac, bd, cx, dx\rangle$ & $3$ & $2$ \\
\hline
\end{tabular}
\end{center}

The first row is $\Lambda' = A'/J$, treated in
Example~\ref{ex:parallelarrows}, with $\pd(S_1) = 3$ and
$\pd(S_2) = 2$. For the second row, set
$\Lambda'' = \mathbf{k}Q'/(ac, bd, cx, dx)$. Then
$P_1 = \langle e_1, a, b\rangle$ and $P_3 = \langle e_3, d\rangle$
(since $dx = 0$), and the projective cover
$P_2 \oplus P_3 \to \rad P_1$ has kernel
$\langle(c, 0), (0, d)\rangle \cong S_4 \oplus S_4$. Since
$c\Lambda'' \cong S_4$ and $\Omega S_4 = P_5$, we have
$\pd_{\Lambda''}(S_2) = 2$. The syzygy
$\Omega_{\Lambda''}(S_3) = \langle d\rangle \cong S_4$ is
nonprojective, so $\pd_{\Lambda''}(S_3) = 2$. Hence
$\pd_{\Lambda''}(S_1) = 1 + \max\{\pd_{\Lambda''}(S_2),
\pd_{\Lambda''}(S_3)\} = 1 + 2 = 3$. So
$(\pd(S_1), \pd(S_2)) = (3, 2)$ in this case as well.

Neither case yields $(2, 2)$, and the pointwise matching fails.
\end{proof}

\begin{remark}\label{rem:weakerproblem}
Proposition~\ref{prop:pointwisefails} shows that the pointwise
reduction to a monomial ideal on the same quiver, under the
hypothesis that every monomial occurring in a relation of $I$ is
contained in $I'$, is impossible. It does not rule out the weaker
possibility that some monomial ideal on the same quiver, possibly
discarding some of the monomials of the relations of $I$, matches
the pointwise projective dimensions. Indeed, the ideal
$\langle bd, cx\rangle$ on $Q'$ yields
$(\pd S_1, \pd S_2, \pd S_3, \pd S_4, \pd S_5) = (2, 2, 1, 1, 0)$,
which coincides with the tuple for $A'$; this is not covered by
Proposition~\ref{prop:pointwisefails} because the monomial $ac$ is
discarded. The gap between preserving the relations and matching
the pointwise dimensions is therefore genuine, and the intermediate
problem is stated in Section~\ref{sec:future}.
\end{remark}

\begin{remark}\label{rem:stringquotientconclusion}
The only construction available for computing the global dimension
of a special biserial algebra is the string-syzygy quiver of
Section~\ref{sec:stringsyzygy}, applied to the string modules that
occur as syzygies of the simple modules of the algebra. The
associated string quotient, though useful in other contexts, is not
a shortcut for the computation of homological invariants. The
failure of the naive path-chain rule in
Proposition~\ref{prop:naivepathchain} and the failure of the
associated string quotient in
Example~\ref{ex:parallelarrows} arise from the same phenomenon: a
binomial relation couples two paths into a single syzygy, and no
reduction to individual paths can reproduce the coupling.
\end{remark}

\section{The universal formula in the biserial case}
\label{sec:biserial}

In Sections~\ref{sec:stringsyzygy} and~\ref{sec:stringalgebra} we
obtained combinatorial formulas for the global dimension of special
biserial and string algebras. In this section we consider the larger
class of biserial algebras. The universal formula of
Section~\ref{sec:universal} remains valid, since every biserial
algebra is finite dimensional; but the string-syzygy compression of
Section~\ref{sec:stringsyzygy} is no longer available, because the
Liu--Morin closure theorem does not extend beyond special biserial
algebras. We give an explicit example, due to Skowro\'nski and
Waschb\"usch, of a biserial algebra whose first syzygies are not
string modules. This example serves as the standard obstruction to
the compression. We then record what is known and what remains open,
and we note that the invariants we study are Morita invariant.

We follow the conventions of Sections~\ref{sec:background}
through~\ref{sec:stringsyzygy}. Throughout, $\Lambda$ is a finite
dimensional biserial algebra over $\mathbf{k}$.

\subsection{The universal formula for biserial algebras}
\label{subsec:universalbiserial}

Since every biserial algebra is finite dimensional by
Definition~\ref{def:biserial}, the universal height formula of
Section~\ref{sec:universal} applies without change.

\begin{corollary}\label{cor:universalbiserial}
Let $\Lambda$ be a finite dimensional biserial algebra. Then the
following hold.
\begin{enumerate}[label=\textup{(\arabic*)}]
\item For every vertex $[X]$ of the reachable syzygy quiver
$\Gamma_\Omega(\Lambda)$,
\[
\pd_\Lambda X = 1 + \operatorname{ht}([X]).
\]
\item $\gldim\Lambda = \sup\bigl(\{0\} \cup
\{1 + \operatorname{ht}([S]) : S \text{ a nonprojective simple
right } \Lambda\text{-module}\}\bigr)$.
\item $\findim\Lambda = \sup\bigl(\{0\} \cup \{1 + \operatorname{ht}([X]) :
[X] \in \Gamma_\Omega^{\mathrm{fin}}(\Lambda)\}\bigr)$.
\item $\gldim\Lambda < \infty$ if and only if
$\Gamma_\Omega(\Lambda)$ is finite and acyclic.
\end{enumerate}
\end{corollary}

\begin{proof}
Since biserial algebras are finite dimensional by
Definition~\ref{def:biserial}, the corollary is a direct application
of Theorems~\ref{thm:universalheight} and~\ref{thm:findimheight}.
\end{proof}

\begin{remark}\label{rem:universalbiserialnotcombinatorial}
Corollary~\ref{cor:universalbiserial} gives an exact formula for the
global and finitistic dimensions of a biserial algebra, but it is not
combinatorial. The vertices of $\Gamma_\Omega(\Lambda)$ are
isomorphism classes of arbitrary modules, and there is no general
statement that identifies them with a class of objects that can be
described from the bound quiver alone. This is the gap that
separates the biserial case from the special biserial case resolved
in Section~\ref{sec:stringsyzygy}.
\end{remark}

\subsection{A genuinely biserial, non-special algebra}
\label{subsec:skowronski}

We present the standard example of a biserial algebra that is not
special biserial, due to Skowro\'nski and Waschb\"usch. We verify
that it is biserial, that it is not special biserial, and that its
first syzygies are not string modules.

\begin{example}[Skowro\'nski--Waschb\"usch]\label{ex:skowronski}
Let $Q$ be the quiver with vertices $1, 2, 3, 4, 5$ and arrows
\[
a : 1 \to 2, \quad b : 2 \to 4, \quad c : 2 \to 3, \quad
d : 3 \to 4, \quad e : 4 \to 5,
\]
as in Figure~\ref{fig:skowronski}. Let
\[
C := \mathbf{k}Q/(be, ab - acd).
\]
Then $C$ is a finite dimensional biserial algebra that is not
special biserial. Moreover $\gldim C = 2$, and the first syzygy of
$S_1$ is a nonprojective module whose coefficient graph contains a
triangle, so it is not a string module. $\triangle$
\end{example}

\vspace{-0.8em}
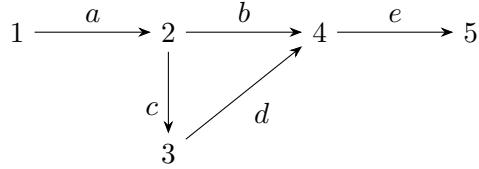
\begin{figure}[h]
\centering
\begin{tikzpicture}[scale=1.0,>=Stealth]
\node (1) at (0,0) {$1$};
\node (2) at (2,0) {$2$};
\node (3) at (2,-1.6) {$3$};
\node (4) at (4,0) {$4$};
\node (5) at (6,0) {$5$};
\draw[->] (1) -- node[above] {$a$} (2);
\draw[->] (2) -- node[above] {$b$} (4);
\draw[->] (2) -- node[below left] {$c$} (3);
\draw[->] (3) -- node[below right] {$d$} (4);
\draw[->] (4) -- node[above] {$e$} (5);
\end{tikzpicture}
\caption{The Skowro\'nski--Waschb\"usch quiver.}
\label{fig:skowronski}
\end{figure}

\begin{proposition}\label{prop:skowronskibiserial}
The algebra $C$ of Example~\ref{ex:skowronski} is biserial.
\end{proposition}

\begin{proof}
We use the notation $z := ab = acd$ in $C$. The $\mathbf{k}$-basis of
$C$ consists of the residue classes of the paths of $Q$ that do not
lie in the ideal $(be, ab - acd)$; these are
\[
e_1, e_2, e_3, e_4, e_5, \ a, b, c, d, e, \ ac, cd, de, cde, \ z,
\]
fifteen elements in total.

\emph{Right projectives.} The indecomposable projective right
$C$-modules $P_i = e_i C$ have the $\mathbf{k}$-bases consisting of
the surviving paths starting at $i$:
\begin{align*}
P_1 &= \langle e_1, a, ac, z\rangle, \\
P_2 &= \langle e_2, b, c, cd, cde\rangle, \\
P_3 &= \langle e_3, d, de\rangle, \\
P_4 &= \langle e_4, e\rangle, \\
P_5 &= \langle e_5\rangle.
\end{align*}
The radicals are
\begin{align*}
\rad P_1 &= \langle a, ac, z\rangle = aC, \\
\rad P_2 &= \langle b, c, cd, cde\rangle = bC \oplus cC, \\
\rad P_3 &= \langle d, de\rangle = dC, \\
\rad P_4 &= \langle e\rangle = eC, \\
\rad P_5 &= 0.
\end{align*}
The modules $aC$, $bC$, $cC$, $dC$, $eC$ are uniserial: the complete
submodule chain of $aC$ is
\[
0 \subset \mathbf{k} z \subset \langle ac, z\rangle \subset aC,
\]
the complete submodule chain of $cC$ is
\[
0 \subset \mathbf{k} cde \subset \langle cd, cde\rangle \subset cC,
\]
the module $dC$ has the submodule chain
$0 \subset \mathbf{k} de \subset dC$, and $bC$, $eC$ are
one-dimensional. In $\rad P_2 = bC \oplus cC$, the two uniserial
summands intersect at zero, so the biserial condition holds for
$P_2$. The remaining radicals are uniserial, so the condition holds
for them as well.

\emph{Left projectives.} An analogous computation gives
\begin{align*}
\rad(Ce_1) &= 0, \\
\rad(Ce_2) &= Ca = \langle a\rangle, \\
\rad(Ce_3) &= \langle c, ac\rangle = Cc, \\
\rad(Ce_4) &= \langle b, d, cd, z\rangle = Cb + Cd, \\
\rad(Ce_5) &= \langle e, de, cde\rangle = Ce.
\end{align*}
The modules $Ca$, $Cc$, $Cb$, $Cd$, $Ce$ are uniserial; in
particular $Cb = \langle b, z\rangle$ and
$Cd = \langle d, cd, z\rangle$, and their intersection
$Cb \cap Cd = \mathbf{k}z$ is simple. Hence all indecomposable left
projectives satisfy the biserial condition.

Therefore $C$ is biserial.
\end{proof}

\begin{proposition}\label{prop:skowronskinonspecial}
The algebra $C$ of Example~\ref{ex:skowronski} is not special
biserial.
\end{proposition}

\begin{proof}
The Gabriel quiver of $C$ is the original quiver $Q$. Indeed, under
our left-to-right path composition convention, the dimension of
$e_i(\rad C/\rad^2 C)e_j$ counts the number of surviving paths of
positive length from $i$ to $j$ that do not lie in $\rad^2 C$. For
$(i, j) = (1, 2)$, $(2, 3)$, $(3, 4)$, and $(4, 5)$, the spaces
$e_i(\rad C)e_j$ are one-dimensional and spanned by $a$, $c$, $d$,
and $e$ respectively. For $(i, j) = (2, 4)$, we have
$e_2(\rad C)e_4 = \langle b, cd\rangle$, and since $cd = c \cdot d
\in \rad^2 C$, its image in $e_2(\rad C/\rad^2 C)e_4$ vanishes, so
$\dim_{\mathbf{k}} e_2(\rad C/\rad^2 C)e_4 = 1$. For all other pairs
$(i, j)$ the space is zero, so the Gabriel quiver of $C$ is $Q$.

Let $\Phi : \mathbf{k}Q \twoheadrightarrow C$ be any admissible bound
quiver presentation of $C$ with kernel $I' = \ker \Phi$. Since the
spaces $e_1(\rad C)e_2$, $e_2(\rad C)e_3$, $e_3(\rad C)e_4$, and
$e_4(\rad C)e_5$ are one-dimensional, after rescaling the arrows of
$Q$ by nonzero scalars we may assume that $\Phi(a) = a$,
$\Phi(c) = c$, $\Phi(d) = d$, and $\Phi(e) = e$, while
\[
\Phi(b) = b + \mu cd \quad \text{for some } \mu \in \mathbf{k},
\]
since $e_2(\rad C)e_4 = \langle b, cd\rangle$ and $cd \in \rad^2 C$.

We consider two cases according to the value of $\mu$.
\begin{itemize}
\item If $\mu \neq -1$, then in $C$ we have
$\Phi(ab) = a(b + \mu cd) = ab + \mu acd = (1 + \mu)z \neq 0$ and
$\Phi(ac) = ac \neq 0$. Thus $ab \notin I'$ and $ac \notin I'$, so
the arrow $a$ admits two distinct continuations $b$ and $c$ outside
$I'$, violating condition \textup{(SB2)} at $a$.
\item If $\mu = -1$, then in $C$ we have
$\Phi(be) = (b - cd)e = be - cde = -cde \neq 0$ and
$\Phi(de) = de \neq 0$. Thus $be \notin I'$ and $de \notin I'$, so
the arrow $e$ admits two distinct predecessors $b$ and $d$ outside
$I'$, violating condition \textup{(SB2)} at $e$.
\end{itemize}
In every admissible presentation of $C$, condition \textup{(SB2)}
fails. Therefore $C$ is not special biserial.
\end{proof}

\begin{proof}[Completion of the proof of Example~\ref{ex:skowronski}]
We compute the global dimension of $C$ and show that the first
syzygy of $S_1$ is not a string module.

\emph{The global dimension of $C$ is $2$.} Write $z := ab = acd$ in
$C$. The projective cover of $S_1$ is $P_1 \to S_1$ with kernel
$\rad P_1 = aC = \langle a, ac, z\rangle$. Consider the right
$C$-module homomorphism
\[
\psi : P_2 \longrightarrow aC, \qquad e_2 \longmapsto a.
\]
This map is well defined: assigning $e_2 \mapsto a$ extends uniquely
to a $C$-linear map $P_2 \to aC$ because $a \in aC e_2$, and the
assignment is compatible with the relations, since for a path $p$
starting at $2$ that vanishes in $C$, the product $ap$ also vanishes
in $C$. For instance
$\psi(e_2 \cdot c) = \psi(c) = a \cdot c = ac = \psi(e_2) \cdot c$,
and $\psi(e_2 \cdot b) = \psi(b) = ab = z = \psi(e_2) \cdot b$.

The map $\psi$ is surjective: $\psi(e_2) = a$, $\psi(c) = ac$, and
$\psi(b) = z$. We compute the kernel. An element of $P_2$ has the
form $\lambda e_2 + \mu b + \nu c + \rho cd + \sigma cde$ with
$\lambda, \mu, \nu, \rho, \sigma \in \mathbf{k}$. Its image is
\[
\lambda a + \mu ab + \nu ac + \rho acd + \sigma acde
= \lambda a + \nu ac + (\mu + \rho) z,
\]
since $acde = 0$ in $C$. This vanishes if and only if
$\lambda = \nu = 0$ and $\mu + \rho = 0$. Hence
\[
\ker \psi
= \{\mu b - \mu cd + \sigma cde : \mu, \sigma \in \mathbf{k}\}
= \langle b - cd, cde\rangle.
\]
Set $M := \langle b - cd, cde\rangle$. The set $\{b - cd, cde\}$ is a
$\mathbf{k}$-basis of $M$, and $M$ is a right $C$-submodule of
$P_2$: indeed,
\[
(b - cd) \cdot e_4 = b - cd, \qquad
(b - cd) \cdot e = be - cde = -cde, \qquad
cde \cdot e_5 = cde,
\]
and all other concatenations of $b - cd$ and $cde$ by elements of
$C$ evaluate to zero. The map
\[
P_4 \longrightarrow M, \qquad e_4 \mapsto b - cd, \quad
e \mapsto -cde,
\]
is an isomorphism of right $C$-modules; the verification follows
the same pattern as the isomorphism \(dC \cong P_4\) established
below. Hence \(M \cong P_4\), which is projective, and
\[
\Omega_C^2 S_1 = M \cong P_4
\]
is projective. Therefore $\pd_C S_1 = 2$.

The remaining simple modules are treated as follows. The projective
cover of $S_2$ is $P_2 \to S_2$ with kernel
$\rad P_2 = bC \oplus cC$. The module $bC = \langle b\rangle$ is
isomorphic to $S_4$; the module $cC = \langle c, cd, cde\rangle$ is
isomorphic to $P_3$ via \(e_3 \mapsto c\), \(d \mapsto cd\),
\(de \mapsto cde\). Hence
\(\Omega_C S_2 = bC \oplus cC \cong S_4 \oplus P_3\), and the
nonprojective summand is $bC \cong S_4$. Then
$\Omega_C^2 S_2 = \Omega_C(S_4) = eC \cong S_5 = P_5$, projective,
so $\pd_C S_2 = 2$.

The projective cover of $S_3$ is $P_3 \to S_3$ with kernel
$\rad P_3 = dC = \langle d, de\rangle$. The map \(P_4 \to dC\)
sending \(e_4 \mapsto d\) and \(e \mapsto de\) is an isomorphism of
right $C$-modules, so \(dC \cong P_4\) and $\pd_C S_3 = 1$. The
projective cover of $S_4$ is $P_4 \to S_4$ with kernel
\(eC \cong S_5 = P_5\), so $\pd_C S_4 = 1$. Finally $\pd_C S_5 = 0$.
By Lemma~\ref{lem:simplegldim},
\[
\gldim C = \sup\{2, 2, 1, 1, 0\} = 2.
\]

\emph{The first syzygy of $S_1$ is not a string module.} The module
$aC = \langle a, ac, z\rangle$ has nonzero arrow actions
\[
a \cdot c = ac, \qquad ac \cdot d = z, \qquad a \cdot b = z.
\]
Since the nonzero weight spaces $aCe_2 = \mathbf{k}a$,
$aCe_3 = \mathbf{k}ac$, and $aCe_4 = \mathbf{k}z$ are all
one-dimensional, any homogeneous basis of $aC$ is a scalar multiple
of $\{a, ac, z\}$; thus the underlying unoriented graph of its
coefficient graph is rigidly a triangle for every homogeneous basis,
and $aC$ is not a string module.
\end{proof}

\begin{remark}\label{rem:skowronskiconsequence}
Example~\ref{ex:skowronski} shows that for a general biserial
algebra the first syzygy of a simple module need not be a string
module. This is the obstruction to extending the string-syzygy
formula of Theorem~\ref{thm:gldimspecialbis} to biserial algebras.
The universal formula of Corollary~\ref{cor:universalbiserial}
still applies, but it requires the explicit computation of $aC$ and
its syzygies as abstract modules rather than as string diagrams.
\end{remark}

\begin{remark}\label{rem:skowronskiquiver}
The reachable syzygy quiver of $C$ has vertices $[S_1]$, $[S_2]$,
$[S_3]$, $[S_4]$, and $[aC]$. The arrows are
\[
[S_1] \to [aC], \qquad [S_2] \to [S_4],
\]
while $[S_3]$, $[S_4]$, and $[aC]$ are sinks. The longest directed
path originating from a simple state has length one, so
Corollary~\ref{cor:universalbiserial} gives
\(\gldim C = 1 + 1 = 2\), in agreement with the computation above.
\end{remark}

\subsection{What is known and what is not known}
\label{subsec:knownunknown}

We state what is known about the global dimension of biserial
algebras and what remains open.

\begin{theorem}\label{thm:whatisknown}
Let $\Lambda$ be a finite dimensional biserial algebra over a field
$\mathbf{k}$. Then the following hold.
\begin{enumerate}[label=\textup{(\arabic*)}]
\item The universal syzygy-quiver formula of
Corollary~\ref{cor:universalbiserial} holds for $\Lambda$.
\item If $\mathbf{k}$ is algebraically closed and $\Lambda$ is
representation finite, then $\Lambda$ is Morita equivalent to a
special biserial algebra. After passing to the basic representative,
the combinatorial string-syzygy formula of
Theorem~\ref{thm:gldimspecialbis} applies.
\item If $\mathbf{k}$ is algebraically closed, every basic biserial
algebra admits a bisected presentation in the sense of
\cite{VFCB98}, and in particular a bound quiver presentation;
consequently the syzygy transitions can be computed from this
presentation by linear algebra.
\item There exists a representation-infinite biserial algebra for
which the reachable syzygy quiver cannot be uniformly compressed to
a path-state or string-state quiver.
\end{enumerate}
\end{theorem}

\begin{proof}
Part \textup{(1)} is Corollary~\ref{cor:universalbiserial}.

For part \textup{(2)}, every representation-finite algebra over an
algebraically closed field has a distributive ideal lattice by
Jans's theorem (see \cite[p.~173]{SW83}), and every distributive
biserial algebra is Morita equivalent to a special biserial algebra
by \cite[Lemma~2]{SW83}.

Part \textup{(3)} follows from the main theorem in
\cite[Section~1]{VFCB98}, which is stated for basic biserial
algebras over an algebraically closed field.

For part \textup{(4)}, let $C$ be the Skowro\'nski--Waschb\"usch
algebra of Example~\ref{ex:skowronski}. Over an algebraically closed
field, since $C$ is basic, biserial, and not special biserial, it is
representation-infinite by the contrapositive of part \textup{(2)}
(see also \cite[p.~176]{SW83}). In fact, $C$ is
representation-infinite over every field $\mathbf{k}$: the quotient
algebra $C / C(e_1 + e_5)C$ is isomorphic to the hereditary path
algebra $\mathbf{k}Q_{\{2,3,4\}}$ of the Euclidean quiver of type
$\widetilde{A}_2$ formed by the vertices $\{2, 3, 4\}$ and arrows
$b : 2 \to 4$, $c : 2 \to 3$, $d : 3 \to 4$, which is
representation-infinite over every field $\mathbf{k}$ (see
\cite[Chapter~VII]{ASS06}). Since the first syzygy of $S_1$ over $C$
is not a string module (Example~\ref{ex:skowronski}), no compression
of the reachable syzygy quiver to a path-state or string-state
quiver is available for $C$.
\end{proof}

\begin{remark}\label{rem:biserialquotient}
The question of whether every finite dimensional biserial algebra
$\Lambda$ can be realised as an admissible quotient
$\Gamma \twoheadrightarrow \Lambda$ of a special biserial algebra
$\Gamma$ has a negative answer.

If $\Gamma = \mathbf{k}Q/I$ is special biserial and
$\Lambda = \Gamma/J$ is an admissible quotient (so that
$J \subseteq \rad^2\Gamma$), then $\Lambda \cong \mathbf{k}Q/J'$ for
an admissible ideal $J'$ containing $I$. Condition \textup{(SB1)} is inherited from $(Q, I)$.
For condition \textup{(SB2)}, if $\alpha\beta \notin J'$ for some
arrows $\alpha$, $\beta$, then $\alpha\beta \notin I$. Hence
$\beta$ is a continuation of $\alpha$ outside $I$, and by the
special biserial property of $\Gamma$ there is at most one such
continuation. Consequently condition \textup{(SB2)} is inherited,
and $\Lambda$ is special biserial.

Therefore the Skowro\'nski--Waschb\"usch algebra $C$ of
Example~\ref{ex:skowronski}, which is biserial but not special
biserial, cannot be a quotient of any special biserial algebra.
\end{remark}

\begin{remark}\label{rem:gaprepresentationinfinite}
The gap in the theory is isolated within biserial algebras for
which the first syzygy of a simple module fails to be a string
module. For such algebras, the universal formula of
Corollary~\ref{cor:universalbiserial} holds, but the state space is
not under combinatorial control. The problem of finding a class of
states for these algebras that admits a combinatorial description
from the bound quiver and is closed under first syzygies is stated
as Problem~\ref{prob:biserialcompression}.
\end{remark}

\subsection{Morita invariance}\label{subsec:moritainvariance}

We close this section by verifying that the invariants we study are
Morita invariant. This justifies the choice of a basic
representative throughout the paper.

\begin{proposition}\label{prop:moritainvariance}
Let $\Lambda$ and $\Lambda'$ be Morita equivalent finite dimensional
$\mathbf{k}$-algebras. Then
\[
\gldim\Lambda = \gldim\Lambda', \qquad
\findim\Lambda = \findim\Lambda'.
\]
Moreover, if $\Lambda$ is biserial, then so is $\Lambda'$.
\end{proposition}

\begin{proof}
Let \(F : \operatorname{mod}\Lambda \to \operatorname{mod}\Lambda'\)
be an equivalence of module categories. Such an equivalence
preserves projective modules, projective covers, exact kernels, and
finite direct sums, and it carries the composition length of a
module to the composition length of the image. It follows that for
every right $\Lambda$-module $M$, the projective dimension of $M$
equals the projective dimension of $F(M)$: the minimal projective
resolution of $M$ is carried by $F$ to a minimal projective
resolution of $F(M)$ of the same length. Consequently, the
supremum of the projective dimensions of the simple modules is
preserved, and by Lemma~\ref{lem:simplegldim} the global dimension
is preserved. The finitistic dimension, being the supremum of the
projective dimensions of the modules of finite projective
dimension, is likewise preserved.

For the biserial assertion, suppose $\Lambda$ is biserial. We use
the two-sided condition of Definition~\ref{def:biserial}: for every
indecomposable projective right $\Lambda$-module $P$,
$\rad P = U + V$ with $U$, $V$ uniserial and $U \cap V$ zero or
simple, and analogously for left projectives.

The equivalence $F$ carries $P$ to an indecomposable projective
right $\Lambda'$-module $F(P)$, since $F$ preserves projectives and
indecomposability. Because $F$ is an equivalence of categories, it
preserves maximal submodules and their intersections; in
particular $F(\rad P) = \rad F(P)$. It carries $U$, $V$ to
submodules $F(U)$, $F(V)$ of $\rad F(P)$ and preserves their
intersection. Uniseriality is preserved, since the linear
submodule lattice of a module is preserved by an equivalence.
Therefore $\rad F(P) = F(U) + F(V)$ with $F(U)$, $F(V)$ uniserial
and $F(U) \cap F(V)$ zero or simple.

For left projectives, note that a Morita equivalence on right
modules induces one on left modules by taking opposite algebras, so
the preceding argument applies to both sides. Therefore $\Lambda'$
is biserial.
\end{proof}

\begin{remark}\label{rem:basicrepresentative}
By Proposition~\ref{prop:moritainvariance}, we may replace a
biserial algebra by its basic representative when computing the
global or the finitistic dimension. The basic representative is a
bound quiver algebra, so the bound quiver presentation is available
for computations.
\end{remark}

\begin{remark}\label{rem:reachablequiverdepends}
The reachable syzygy quiver $\Gamma_\Omega(\Lambda)$ depends on the
choice of the basic representative: different basic representatives
of the same Morita class are isomorphic as algebras, and the quiver
is defined only up to isomorphism. The quiver is not an absolute
invariant of the Morita class in the sense of admitting a
description independent of the presentation. What is Morita
invariant is the height of the simple states and the resulting
supremum, not the specific combinatorial presentation of the graph.
\end{remark}

\section{Examples}\label{sec:examples}

In this section we collect several examples, accompanied by quiver
diagrams, that illustrate the strict inclusions
\[
\text{gentle} \subsetneq \text{string} \subsetneq
\text{special biserial} \subsetneq \text{biserial},
\]
the homological difference between monomial and binomial relations,
and the failure of the string-state compression for general biserial
algebras. Examples~\ref{ex:infinite}, \ref{ex:binomialexample}, and
\ref{ex:biserialexample} refer to
Sections~\ref{sec:gabriel}, \ref{sec:binomial}, and
\ref{sec:biserial} for the detailed computations. Examples
\ref{ex:gentle}, \ref{ex:stringnotgentle}, and
\ref{ex:symmetric} are proved in this section.

\subsection{Gentle algebras and non-gentle string algebras}
\label{subsec:gentlestring}

The following two examples exhibit the strict inclusion
$\text{gentle} \subsetneq \text{string}$. Both algebras have global
dimension two, so the class inclusion does not correlate with the
global dimension.

\begin{example}[A gentle algebra]\label{ex:gentle}
Consider the algebra $\Lambda := \mathbf{k}Q/(\alpha\beta)$, where
$Q$ is the linear quiver of Figure~\ref{fig:twocycle_gentle}.

\vspace{-1.2em}
\begin{figure}[h]
\centering
\begin{tikzpicture}[scale=1.4,>=Stealth]
\node (1) at (0,0) {$1$};
\node (2) at (1.5,0) {$2$};
\node (3) at (3,0) {$3$};
\draw[->] (1) -- node[above] {$\alpha$} (2);
\draw[->] (2) -- node[above] {$\beta$} (3);
\end{tikzpicture}
\caption{The quiver $Q$ of the gentle algebra $\Lambda$.}
\label{fig:twocycle_gentle}
\end{figure}

Then $\Lambda$ is a gentle algebra, and $\gldim\Lambda = 2$.
$\triangle$
\end{example}

\begin{proof}
The ideal $I = (\alpha\beta)$ is monomial, so $\Lambda$ is a string
algebra. Condition \textup{(SB1)} holds because every vertex of $Q$
has out-degree and in-degree at most one. For condition
\textup{(SB2)}: the only composable pair of arrows in $Q$ is
$\alpha\beta$, which lies in $I$; consequently no arrow of $Q$
admits a continuation or a predecessor outside $I$, and
\textup{(SB2)} holds vacuously. For the gentle condition of
Definition~\ref{def:gentle}, the unique arrow $\beta'$ with
$\alpha\beta' \in I$ is $\beta$, and the unique arrow $\gamma$ with
$\gamma\beta \in I$ is $\gamma = \alpha$. Hence $\Lambda$ is gentle.

The indecomposable projective right $\Lambda$-modules have
$\mathbf{k}$-bases
\[
P_1 = \langle e_1, \alpha\rangle, \qquad
P_2 = \langle e_2, \beta\rangle, \qquad
P_3 = \langle e_3\rangle.
\]
The radicals are
\[
\rad P_1 = \mathbf{k}\alpha \cong S_2, \qquad
\rad P_2 = \mathbf{k}\beta \cong S_3 = P_3, \qquad
\rad P_3 = 0.
\]
Hence $\Omega_\Lambda S_1 \cong S_2$,
$\Omega_\Lambda S_2 \cong S_3 = P_3$, and
$\Omega_\Lambda S_3 = 0$. The minimal projective resolution of $S_1$
is
\[
0 \to P_3 \to P_2 \to P_1 \to S_1 \to 0,
\]
which gives $\pd_\Lambda S_1 = 2$, $\pd_\Lambda S_2 = 1$,
$\pd_\Lambda S_3 = 0$, and $\gldim\Lambda = 2$ by
Lemma~\ref{lem:simplegldim}.
\end{proof}

\begin{example}[A string algebra that is not gentle]
\label{ex:stringnotgentle}
Consider the algebra $\Lambda' := \mathbf{k}Q/(\alpha\beta_1,
\alpha\beta_2)$, where $Q$ is the quiver of
Figure~\ref{fig:stringnotgentle}.

\vspace{-1.2em}
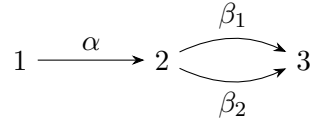
\begin{figure}[h]
\centering
\begin{tikzpicture}[scale=1.25,>=Stealth]
\node (1) at (0,0) {$1$};
\node (2) at (1.5,0) {$2$};
\node (3) at (3,0) {$3$};
\draw[->] (1) -- node[above] {$\alpha$} (2);
\draw[->] (2) to[bend left=25] node[above] {$\beta_1$} (3);
\draw[->] (2) to[bend right=25] node[below] {$\beta_2$} (3);
\end{tikzpicture}
\caption{The quiver $Q$ of the non-gentle string algebra
$\Lambda'$.}
\label{fig:stringnotgentle}
\end{figure}

Then $\Lambda'$ is a string algebra that is not gentle, and
$\gldim\Lambda' = 2$. $\triangle$
\end{example}

\begin{proof}
The algebra $\Lambda'$ is a string algebra: condition \textup{(SB1)}
is immediate from the shape of $Q$, and condition \textup{(SB2)}
holds vacuously because the only composable pairs of arrows in $Q$
are $\alpha\beta_1$ and $\alpha\beta_2$, both of which lie in the
monomial ideal $I$. It is not gentle, since the arrow $\alpha$
admits the two continuations $\beta_1,\beta_2$ with
$\alpha\beta_i \in I$, so the uniqueness condition of
Definition~\ref{def:gentle} fails.

The indecomposable projective right $\Lambda'$-modules have
$\mathbf{k}$-bases
\[
P_1 = \langle e_1, \alpha\rangle, \qquad
P_2 = \langle e_2, \beta_1, \beta_2\rangle, \qquad
P_3 = \langle e_3\rangle.
\]
The radicals are
\[
\rad P_1 = \mathbf{k}\alpha \cong S_2, \qquad
\rad P_2 = \mathbf{k}\beta_1 \oplus \mathbf{k}\beta_2
\cong S_3^{\oplus 2}, \qquad
\rad P_3 = 0.
\]
Here each $\mathbf{k}\beta_i$ is one-dimensional and annihilated by
the radical of $\Lambda'$, so $\mathbf{k}\beta_i \cong S_3$. Hence
$\Omega_{\Lambda'} S_1 \cong S_2$,
$\Omega_{\Lambda'} S_2 \cong S_3^{\oplus 2}$ (projective), and
$\Omega_{\Lambda'} S_3 = 0$. The minimal projective resolution of
$S_1$ is
\[
0 \to P_3^{\oplus 2} \to P_2 \to P_1 \to S_1 \to 0,
\]
which gives $\pd_{\Lambda'} S_1 = 2$, $\pd_{\Lambda'} S_2 = 1$,
$\pd_{\Lambda'} S_3 = 0$, and $\gldim\Lambda' = 2$.
\end{proof}

\begin{remark}\label{rem:gentlestringcomparison}
Examples~\ref{ex:gentle} and~\ref{ex:stringnotgentle} exhibit the
strict inclusion $\text{gentle} \subsetneq \text{string}$. Both
algebras have global dimension two, so the structural hierarchy is
not detected by the global dimension alone.
\end{remark}

\subsection{A string algebra with infinite global dimension}
\label{subsec:exampleinfinite}

\begin{example}\label{ex:infinite}
Consider the algebra $B = \mathbf{k}Q/(\alpha\beta, \beta\alpha)$ of
Theorem~\ref{thm:twostring}, where $Q$ is the oriented $2$-cycle of
Figure~\ref{fig:twocycle}. Then $B$ is a string algebra with
$\gldim B = \infty$ and $\findim B = 0$. $\triangle$
\end{example}

\begin{proof}
The algebra $B$ is a string algebra by
Theorem~\ref{thm:twostring}, where the direct computation of the
periodic resolution of $S_1$ gives $\gldim B = \infty$ and the
conclusion $\findim B = 0$ is obtained from the classification of
indecomposable $B$-modules. Alternatively, the relation-chain quiver
$\mathcal{R}(B)$ has vertices $\alpha$, $\beta$ and arrows
$\alpha \to \beta$, $\beta \to \alpha$, since
$\alpha\beta = \beta\alpha = 0$ in $B$. This is a directed $2$-cycle,
so Theorem~\ref{thm:zhengzhang} gives $\gldim B = \infty$, and
Remark~\ref{rem:findimstringnaive} gives $\findim B = 0$.
\end{proof}

\subsection{A special biserial algebra with a binomial relation}
\label{subsec:examplebinomial}

\begin{example}\label{ex:binomialexample}
Consider the commutativity diamond $A = \mathbf{k}Q/(ac - bd)$ of
Example~\ref{ex:commutativediamond}, with quiver $Q$ as in
Figure~\ref{fig:diamond}. Then $A$ is a special biserial algebra that
is not a string algebra, and $\gldim A = 2$. $\triangle$
\end{example}

\begin{proof}
The proof of Example~\ref{ex:commutativediamond} shows that $A$ is
special biserial and not a string algebra, and computes
$\gldim A = 2$ by an explicit syzygy computation. The string-syzygy
quiver $\Gamma_\Omega^{\mathrm{str}}(A)$ has vertices $[S_1]$,
$[S_2]$, $[S_3]$, and $[\rad P_1]$, with a single arrow
$[S_1] \to [\rad P_1]$ and with $[S_2]$, $[S_3]$, and
$[\rad P_1]$ acting as sinks because
$\Omega_A(\rad P_1) \cong S_4 = P_4$ is projective.
Theorem~\ref{thm:gldimspecialbis} yields
$\gldim A = 1 + 1 = 2$.
\end{proof}

\subsection{A symmetric special biserial algebra}
\label{subsec:examplesymmetric}

\begin{example}\label{ex:symmetric}
Let $Q$ be the quiver of Figure~\ref{fig:twoloops} with a single
vertex $1$ and two loops $\alpha$, $\beta$, and let
\[
\Lambda := \mathbf{k}Q/(\alpha^2, \beta^2, \alpha\beta - \beta\alpha).
\]
Then $\Lambda$ is a symmetric special biserial algebra with
$\gldim\Lambda = \infty$ and $\findim\Lambda = 0$. $\triangle$
\end{example}

\vspace{-1.5em}
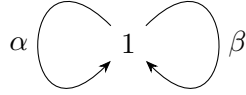
\begin{figure}[h]
\centering
\begin{tikzpicture}[scale=1.0,>=Stealth]
\node (1) at (0,0) {$1$};
\draw[->] (1) to[out=135,in=225,looseness=10] node[left] {$\alpha$} (1);
\draw[->] (1) to[out=45,in=-45,looseness=10] node[right] {$\beta$} (1);
\end{tikzpicture}
\caption{The two-loop quiver $Q$ of the symmetric special biserial
algebra $\Lambda$.}
\label{fig:twoloops}
\end{figure}

\begin{proof}
Vertex $1$ is the source and target of exactly two arrows, so
\textup{(SB1)} holds. Since $\alpha^2, \beta^2 \in I$ while
$\alpha\beta \notin I$, the arrow $\alpha$ has $\beta$ as its unique
continuation and unique predecessor outside $I$, and symmetrically
for $\beta$; hence \textup{(SB2)} holds and $\Lambda$ is special
biserial. The relations imply
\[
(\alpha\beta)^2 = \alpha\beta\alpha\beta
= \alpha(\alpha\beta)\beta
= \alpha(\beta\alpha)\beta
= \alpha^2\beta^2 = 0,
\]
so $\Lambda$ has $\mathbf{k}$-basis
$\{e_1, \alpha, \beta, \alpha\beta\}$ and dimension four. The
radical is $\rad\Lambda = \langle\alpha, \beta\rangle$, with
$\rad^2\Lambda = \langle\alpha\beta\rangle$ and
$\rad^3\Lambda = 0$. The socle is
$\soc\Lambda = \langle\alpha\beta\rangle$, which is one-dimensional.

Since $\Lambda$ is a finite dimensional local algebra with
one-dimensional socle, it is a Frobenius algebra. Because $\Lambda$
is commutative, it is symmetric. By
Corollary~\ref{cor:symmetricinfinite}, $\gldim\Lambda = \infty$ and
$\findim\Lambda = 0$.
\end{proof}

\begin{remark}\label{rem:symmetricexamplebasis}
The algebra $\Lambda$ of Example~\ref{ex:symmetric} is a local
commutative symmetric algebra whose radical requires two generators,
so it is not a Nakayama algebra. In the framework of
Section~\ref{subsec:liumorin}, a string diagram for $\Lambda$
contains no two consecutive identical direct syllables (since
$\alpha^2 = \beta^2 = 0$), and no consecutive direct syllables
$\alpha\beta$ or $\beta\alpha$ (since, at the base vertex of the
first syllable, the other factor of the binomial acts as zero,
contradicting the annihilation condition). By the symmetry of the
inverse subpaths, the string diagram also contains no two
consecutive inverse syllables. Hence every string diagram of length at least two strictly
alternates direct and inverse syllables. Up to cyclic permutation
and inversion, the unique band over $\Lambda$ is therefore
$\beta^{-1}\alpha$. Consequently, $\Lambda$ is a domestic symmetric
special biserial algebra. The conclusion $\findim\Lambda = 0$ stated
in Example~\ref{ex:symmetric} follows from self-injectivity: every
module of finite projective dimension over a self-injective algebra
is projective, so every nonprojective module has infinite projective
dimension.
\end{remark}

\subsection{A biserial non-special algebra}
\label{subsec:examplebiserial}

\begin{example}\label{ex:biserialexample}
Consider the Skowro\'nski--Waschb\"usch algebra
$C = \mathbf{k}Q/(be, ab - acd)$ of Example~\ref{ex:skowronski},
with quiver $Q$ as in Figure~\ref{fig:skowronski}. Then $C$ is a
biserial algebra that is not special biserial, and $\gldim C = 2$.
The first syzygy of $S_1$ over $C$ is a nonprojective module whose
coefficient graph contains a triangle, so it is not a string module.
$\triangle$
\end{example}

\begin{proof}
The proof of Example~\ref{ex:skowronski} verifies that $C$ is
biserial, that $C$ is not special biserial, and that
$\gldim C = 2$. The first syzygy
$\Omega_C S_1 = aC = \langle a, ac, z\rangle$, where
$z = ab = acd$, has nonzero arrow actions
\[
a \cdot c = ac, \qquad ac \cdot d = z, \qquad a \cdot b = z.
\]
Its coefficient graph is therefore the triangle with vertices $a$,
$ac$, $z$ and edges
\[
a \xrightarrow{c} ac, \qquad ac \xrightarrow{d} z, \qquad
a \xrightarrow{b} z.
\]
Since the basis elements $a$, $ac$, $z$ reside in distinct
components $Ce_2$, $Ce_3$, $Ce_4$ respectively (as determined by the
targets $2, 3, 4$ of the corresponding paths), this homogeneous basis
is unique up to scalar multiplication. Since the underlying
unoriented graph is rigidly a triangle and not a simple path, $aC$
cannot be a string module.
\end{proof}

\begin{remark}\label{rem:exampleconclusion}
The examples of
Sections~\ref{subsec:gentlestring}--\ref{subsec:examplebiserial}
together witness the strict inclusions
\[
\text{gentle} \subsetneq \text{string} \subsetneq
\text{special biserial} \subsetneq \text{biserial},
\]
and illustrate the homological phenomena discussed in
Sections~\ref{sec:gabriel}--\ref{sec:biserial}: the failure of the
Gabriel quiver to determine the global dimension, the passage from
string algebras to special biserial algebras, the failure of the
path-state compression for binomial relations, and the failure of
the string-state compression for biserial algebras that are not
special biserial.
\end{remark}

\section{Further directions and open problems}\label{sec:future}

In this paper we introduced the reachable syzygy quiver of a finite
dimensional algebra and showed that the projective dimension of a
vertex equals one plus its directed path height. For special
biserial algebras the reachable syzygy quiver admits a combinatorial
description in terms of string modules, and for string algebras it
compresses to the relation-chain quiver of Xin and Zhang. The
biserial case, in which string states are no longer sufficient, was
discussed in Section~\ref{sec:biserial}, and the negative result on
pointwise monomial reduction was proved in
Section~\ref{sec:binomial}. Several questions related to this
material remain open. We collect them in this section.

\emph{A combinatorial state space for biserial algebras.} The
Skowro\'nski--Waschb\"usch algebra (Example~\ref{ex:skowronski})
shows that string modules are not sufficient for biserial algebras
in general: the first syzygy of $S_1$ over this algebra is not a
string module. The string-syzygy quiver of
Section~\ref{sec:stringsyzygy} therefore has no analogue for general
biserial algebras. The following question asks whether some other
class of modules can serve as a state space.

\begin{problem}\label{prob:biserialcompression}
Let $\Lambda$ be a finite dimensional biserial algebra. Does there
exist a class $\mathcal{C}$ of finite dimensional right
$\Lambda$-modules, describable from the bound quiver, such that the
following hold?
\begin{enumerate}[label=\textup{(\alph*)}]
\item Every indecomposable nonprojective summand of
$\Omega_\Lambda^n S$ belongs to $\mathcal{C}$, for every simple
right $\Lambda$-module $S$ and every $n \ge 0$.
\item The first syzygy of every module in $\mathcal{C}$ decomposes
into a finite direct sum of projective modules and modules in
$\mathcal{C}$.
\item The projective covers, the radicals, and the syzygy
transitions of modules in $\mathcal{C}$ can be computed from the
bound quiver.
\end{enumerate}
If such a class exists, the global and finitistic dimensions of
$\Lambda$ can be computed from the bound quiver by the analog of
Algorithm~\ref{alg:breadthfirst} applied to the class $\mathcal{C}$.
\end{problem}

Example~\ref{ex:skowronski} shows that the class of string modules
does not satisfy the conditions of
Problem~\ref{prob:biserialcompression} for general biserial
algebras. Whether some other class $\mathcal{C}$ satisfies these
conditions is open.

\emph{Weak monomial reduction.} Proposition~\ref{prop:pointwisefails}
shows that no admissible monomial ideal on the same quiver
containing every monomial occurring in a relation of the original
algebra can achieve pointwise matching of the projective dimensions
in general. Once the requirement of containing every monomial of
the relations is dropped, the following questions remain open.

\begin{problem}\label{prob:weakmonomial}
Let $\Lambda = \mathbf{k}Q/I$ be a special biserial algebra with
$\gldim\Lambda = d < \infty$. Does there exist an admissible
monomial ideal $I'$ on $Q$ such that $\Lambda' = \mathbf{k}Q/I'$ is
a string algebra and $\gldim\Lambda' = d$? More strongly, does
there exist such $I'$ with $\pd_{\Lambda'}(S_v) = \pd_\Lambda(S_v)$
for every $v \in Q_0$, or with
$\findim\Lambda' = \findim\Lambda$?
\end{problem}

As noted in Remark~\ref{rem:weakerproblem}, the algebra $A'$ of
Example~\ref{ex:parallelarrows} admits the monomial witness
$I' = \langle bd, cx\rangle$ on $Q'$, which discards the monomial
$ac$ and satisfies both $\pd_{\Lambda'}(S_v) = \pd_{A'}(S_v)$ for
all $v \in Q'_0$ and $\gldim(\mathbf{k}Q'/I') = 2 = \gldim A'$. It is not known whether such a monomial witness exists for every
special biserial algebra.

\emph{Derived invariance.} Section~\ref{sec:gabriel} shows that two
string algebras with the same Gabriel quiver may have different
global dimensions, and Corollary~\ref{cor:notderived} shows that
the two examples are not derived equivalent. Derived equivalence
preserves the finiteness of the global dimension, but the value
itself may change.

\begin{problem}\label{prob:derivedinvariant}
Let $\Lambda$ and $\Lambda'$ be finite dimensional special biserial
algebras that are derived equivalent. What are the sharpest bounds
on $|\gldim\Lambda - \gldim\Lambda'|$ in terms of the structure of
the tilting complex? Is $\findim\Lambda = \findim\Lambda'$?
\end{problem}

For gentle algebras, the derived category determines the surface
model up to homeomorphism, and the finiteness of the global
dimension corresponds to the absence of unmarked boundary
components. It would be interesting to know whether the exact
values of the global and finitistic dimensions are also determined
by the surface model.

\emph{Computational complexity.}
Algorithm~\ref{alg:breadthfirst} computes the global dimension of a
special biserial algebra when the string-syzygy quiver is finite
and acyclic, and detects an infinite global dimension when a
directed cycle is reachable from a simple state. When the
string-syzygy quiver is infinite and acyclic, the algorithm does
not terminate, although the global dimension is infinite by
Theorem~\ref{thm:gldimspecialbis}.

\begin{problem}\label{prob:complexity}
What is the worst-case computational complexity of constructing the
string-syzygy quiver of a finite dimensional special biserial
algebra from its bound quiver presentation? Is there a
polynomial-time algorithm? Is the problem of deciding whether the
global dimension is finite decidable in polynomial time?
\end{problem}

The question is open even for the class of string algebras, where
the relation-chain quiver replaces the string-syzygy quiver. The
size of the relation-chain quiver is bounded by the number of paths
in $Q$ that survive modulo $I$, which can be exponential in
$|Q_1|$ in the worst case. Whether this bound can be improved using
the structure of the algebra is not known.

\emph{Other classes of algebras.} The string-syzygy compression of
Section~\ref{sec:stringsyzygy} uses conditions \textup{(SB1)} and
\textup{(SB2)} of Definition~\ref{def:specialbis}. Several
neighbouring classes of algebras share some of the properties of
special biserial algebras, but the compression does not extend to
them.

\begin{problem}\label{prob:skewgentle}
Let $\Lambda$ be a skew-gentle, semi-gentle, or folded gentle
algebra. Does there exist a geometric criterion for the finiteness
of the global dimension of $\Lambda$?
\end{problem}

For skew-gentle algebras, the surface models suggest that such a
criterion should exist, but the algebraic translation has not been
carried out. Semi-gentle algebras arise from gentle algebras by
allowing relations of length at least three, and folded gentle
algebras arise from group actions on gentle algebras. The universal
formula of Theorem~\ref{thm:universalheight} applies to all these
classes because the algebras are finite dimensional, but no
combinatorial compression of the reachable syzygy quiver to a
smaller string-like state space is known.

\emph{The stable rank.} The stable rank of a module category is the
smallest ordinal $\lambda$ such that
$\operatorname{rad}_\Lambda^\lambda
= \operatorname{rad}_\Lambda^{\lambda+1}$, where
$\operatorname{rad}_\Lambda$ is the radical of the module category.
In \cite[Corollary~4.3.4]{GKS21} it is shown that the stable rank of
a meta-torsion-free string algebra takes values in
$\{\omega, \omega+1, \omega+2\}$, and in
\cite[Conjecture~4.4.1]{GKS21} it is conjectured that the stable
rank of any special biserial algebra with at least one band
satisfies
$\omega \le \operatorname{st}(\Lambda) < \omega^2$. The
string-syzygy quiver determines the global and finitistic
dimensions by Theorems~\ref{thm:gldimspecialbis}
and~\ref{thm:findimspecialbis}. It is natural to ask whether it also
determines the stable rank.

\begin{problem}\label{prob:stablerank}
Is there a relationship between the stable rank of the module
category of a special biserial algebra and the height function of
its string-syzygy quiver? Does the string-syzygy quiver determine
the stable rank in the same way that it determines the global and
finitistic dimensions?
\end{problem}

The methods of \cite{GKS21} rely on the prime-band version of the
bridge quiver and on recursive systems for showing that graph maps
lie in the stable radical. It would be interesting to know whether
the string-syzygy quiver, which determines the projective
dimension, can be used to recover these results and to extend them
to the class of special biserial algebras.

\vspace{-0.3em}
\section*{Acknowledgements}

The author gratefully acknowledges the members of Centro Marplatense de
Investigaciones Matem\'aticas (CEMIM) at Universidad Nacional de Mar del Plata (UNMDP), especially Prof.\
Sonia Trepode, for stimulating and generous discussions on
homological aspects in the representation theory of algebras, which
have greatly enriched this work. The author is also indebted to the
PICT 202101154 ANPCyT project ``Dimer Models, $\tau$-tilting
Theory, Trivial Extensions and Brauer Configurations'' at FECyN,
UNMDP, for invaluable assistance.


\begin{thebibliography}{99}

\bibitem[ASS06]{ASS06}
I.~Assem, D.~Simson, and A.~Skowro\'nski,
\emph{Elements of the Representation Theory of Associative Algebras,
Vol.~1: Techniques of Representation Theory},
London Mathematical Society Student Texts \textbf{65},
Cambridge University Press, Cambridge, 2006.

\bibitem[BR87]{BR87}
M.~C.~R. Butler and C.~M.~Ringel,
Auslander--Reiten sequences with few middle terms and applications to
string algebras,
\emph{Comm. Algebra} \textbf{15} (1987), no.~1--2, 145--179.

\bibitem[CB98]{CB98}
W.~W. Crawley-Boevey,
Infinite-dimensional modules in the representation theory of
finite-dimensional algebras,
in: \emph{Algebras and Modules, I} (Trondheim, 1996),
CMS Conf. Proc. \textbf{23},
Amer. Math. Soc., Providence, RI, 1998, pp.~29--54.

\bibitem[EHIS04]{EHIS04}
K.~Erdmann, T.~Holm, O.~Iyama, and J.~Schr\"oer,
Radical embeddings and representation dimension,
\emph{Adv. Math.} \textbf{185} (2004), no.~1, 159--177.

\bibitem[GKS21]{GKS21}
E.~Gupta, A.~Kuber, and S.~Sardar,
On the stable radical of some non-domestic string algebras,
\emph{Algebr. Represent. Theory} \textbf{25} (2022), no.~5,
1207--1230.

\bibitem[Kra91]{Kra91}
H.~Krause,
Maps between tree and band modules,
\emph{J. Algebra} \textbf{137} (1991), no.~1, 186--194.

\bibitem[Lak16]{Lak16}
R.~Laking,
\emph{String algebras in representation theory},
Ph.D. thesis, University of Manchester, 2016.

\bibitem[LM04]{LM04}
S.~Liu and J.-P. Morin,
The strong no loop conjecture for special biserial algebras,
\emph{Proc. Amer. Math. Soc.} \textbf{132} (2004), no.~12,
3513--3523.

\bibitem[Sch15]{Sch15}
S.~Schroll,
Trivial extensions of gentle algebras and Brauer graph algebras,
\emph{J. Algebra} \textbf{444} (2015), 183--200.

\bibitem[SW83]{SW83}
A.~Skowro\'nski and J.~Waschb\"usch,
Representation-finite biserial algebras,
\emph{J. Reine Angew. Math.} \textbf{345} (1983), 172--181.

\bibitem[VFCB98]{VFCB98}
R.~Vila-Freyer and W.~Crawley-Boevey,
The structure of biserial algebras,
\emph{J. London Math. Soc. (2)} \textbf{57} (1998), no.~1, 41--54.

\bibitem[WW85]{WW85}
B.~Wald and J.~Waschb\"usch,
Tame biserial algebras,
\emph{J. Algebra} \textbf{95} (1985), no.~2, 480--500.

\bibitem[ZZ26]{ZZ26}
Z.~Xin and L.~Zhang,
Global dimension of a string algebra,
\emph{Arch. Math. (Basel)} (2026),
\url{https://doi.org/10.1007/s00013-026-02262-x}.

\end{thebibliography}
\end{document}